\documentclass[ twocolumn,
 amsmath,amssymb,
 aps, pre, amsthm, amsfonts,
]{revtex4-2}

\usepackage{amsthm}

\newtheorem{theorem}{Theorem}
\newtheorem{lemma}{Lemma}

\usepackage{bbm}

\usepackage{graphicx}%

\usepackage{dcolumn}
\usepackage{bm}
\usepackage{hyperref}

\usepackage{mathtools,enumitem}

\usepackage{csquotes}
\usepackage{xcolor}
\usepackage{todonotes}

\DeclareMathOperator{\EX}{\mathbb{E}}

\newcommand{\e}{\mathrm{e}}
\newcommand{\prob}{\mathbb{P}}
\newcommand{\nn}{\nonumber}

\newcommand{\sss}{\scriptscriptstyle}
\newcommand{\convd}{\stackrel{d}{\longrightarrow}}

\begin{document}

\title{Clustering without geometry in sparse networks with independent edges}

\date{\today}

\author{Alessio Catanzaro$^{1,2,3}$}
\author{Remco van der Hofstad$^4$}
\author{Diego Garlaschelli$^{1,2}$}

\affiliation{$^1$IMT School for Advanced Studies, Lucca, Italy}
\affiliation{$^2$Instituut-Lorentz for Theoretical Physics, {Leiden University, The Netherlands}}
\affiliation{$^3$Universit\`a di Palermo, Palermo, Italy
}
\affiliation{$^4$Department of Mathematics and Computer Science, Eindhoven University of Technology, The
Netherlands.
}%

\begin{abstract}
The coexistence of sparsity and clustering (non-vanishing average fraction of triangles per node) is one of the few structural features that, irrespective of finer details, are ubiquitously observed across large real-world networks. This fact calls for generic models  producing sparse clustered graphs. 
Earlier results suggested that sparse random graphs with independent edges fail to reproduce clustering, unless edge probabilities are assumed to depend on underlying metric distances that, thanks to the triangle inequality, naturally favour triadic closure. 
This observation has opened a debate on whether clustering implies (latent) geometry in real-world networks. 
Alternatively, recent models of higher-order networks can replicate clustering by abandoning edge independence.
In this paper, we mathematically prove that a sparse random graph with independent edges, recently identified in the context of network renormalization as an invariant model under node aggregation, produces finite clustering without any geometric or higher-order constraint.
The underlying mechanism is an infinite-mean node fitness, which also implies a power-law degree distribution. Further, as a novel phenomenon that we characterize rigorously, we observe the breakdown of self-averaging of various network properties.
Therefore, as an alternative to geometry or higher-order dependencies, node aggregation invariance emerges as a basic route to realistic network properties.
\end{abstract}

\maketitle

\paragraph*{Introduction: clustering versus geometry.}

While real-world networks differ in various structural details, most of them are characterized by virtually ubiquitous features, including a broad (i.e., infinite-variance) distribution of node degrees (numbers of links per node), the \emph{small-world} property (i.e., average shortest path length growing at most logarithmically with the number of nodes), and the coexistence of \emph{sparsity} (i.e., vanishing link density) and \emph{local clustering} (i.e., finite average density of triangles around nodes) in the limit of infinitely large networks~\cite{caldarelli2007scale,newman2018networks}. 
While several models based on generic mechanisms can successfully replicate broad degree distributions and short path lengths~\cite{newman2018networks,van2016random}, the coexistence of sparsity and local clustering is much more difficult to replicate and hence more informative, as it requires specific mechanisms that discriminate models at a finer level. 

The prototypical example of this difficulty is illustrated by \emph{random graph models with independent edges}, where different pairs of nodes are independently connected.
In the Erd\H{o}s-R\'enyi (ER) model~\cite{erdHos1960evolution}, where edges are independent and identically distributed (i.i.d.) with the same connection probability $p$, it is easy to show that the link density (expected fraction of realized links) and the local clustering coefficient (defined as the fraction of triangles around a node, averaged over all nodes) have the same expected value $p$ for networks with a large number $n$ of nodes. 
This automatically implies that, as $n$ diverges, one cannot simultaneously produce finite local clustering and vanishing link density. 
To overcome this limitation of the ER model, 
two main routes have been explored. 

One route introduces dependencies between edges. This can happen via the addition of explicit triadic interactions in an Exponential Random Graph setting~\cite{efe04841-7bd2-388d-abca-ebaec75f4392, PhysRevE.69.026106, ef7cae8a-3889-3d9f-9a3b-d4f0fbbf0154, newman2003properties, newman2009random}, which, however, tend to produce phase transitions whereby the desired expected values of density and clustering become unstable, with typical realizations `condensing' away from the expectation (strategies to overcome this limitation have been proposed in~\cite{https://doi.org/10.1111/j.1467-9531.2006.00176.x, Akara-pipattana_2021,PhysRevE.81.046115}). Other mechanisms include the random superposition of small subgraphs~\cite{bollobas2011sparse} or the monopartite projection~\footnote{In monopartite projections of bipartite networks, the original nodes are first initially connected (possibly independently of each other) to `auxiliary’ nodes of a different type (representing for instance the possibile affiliations or memberships of the original nodes), and the auxiliary nodes are then eliminated while connecting the original nodes to each other, if they share connections to the same auxiliary nodes.} of random bipartite networks~\cite{newman2003social}, both producing networks of overlapping cliques (the edges of which are all mutually dependent) where local clustering can remain very large even for vanishing density. 
More recent attempts generalize these ideas to higher-order networks such as simplicial complexes and hypergraphs~\cite{battiston2022higher,boccaletti2023structure}, where edges belonging to the same simplex or hyperedge are mutually dependent. 
While these more complicated settings can realize the coexistence of sparsity and clustering at the expense of introducing higher-order dependencies, whether that coexistence can be achieved for more parsimonious models with independent edges remains an open question.

The second route preserves edge independence while introducing heterogeneity. 
More precisely, one can turn the \emph{unconditional} link independence of the ER model into a \emph{conditional} one, by replacing, for each pair of nodes $i$ and $j$ ($i,j=1,n$), the constant $p$ with a probability $p_{ij}$ that is a function $f(w_i,w_j)$ of the realized values $w_i$, $w_j$ of a certain node-specific variable $w$ (called \emph{fitness} or weight, hidden variable, latent variable, Lagrange multiplier, etc.) that is preliminarily assigned to vertices via an i.i.d.\ draw from a certain probability density function (pdf) $\rho(w)$~\cite{caldarelli2002scale,van2016random, boguna2003class,squartini2017maximum}.
Typical choices of $f(w_i,w_j)$ are monotonically increasing in both arguments (and necessarily symmetric for undirected graphs), implementing the idea that, the larger the fitness, the larger the probability that a node connects to other nodes~\footnote{Some popular examples are  $f(w_i,w_j)=zw_i w_j$, $f(w_i,w_j)=\Theta(w_i+w_j-z)$, $f(w_i,w_j)=zw_i w_j/(1+zw_i w_j)$, and $f(w_i,w_j)=1-\e^{-zw_i w_j}$, where $z>0$ is a global parameter controlling for the overall link density and $w_i\ge0$ $\forall i$.}.

In the studies carried out so far, it turned out that, while several combinations of $f$ and $\rho$ can generate sparse small-world networks with a broad degree distribution~\cite{caldarelli2002scale,van2016random, boguna2003class,squartini2017maximum}, they systematically fail to generate (positive) clustering, casting doubts on whether models with independent edges can replicate local clustering. 
The only exception that has become popular is that of \emph{random geometric graphs}, where the fitness $w$ acquires the role of a \emph{node coordinate} (which, more generally, can be a $D$-dimensional vector), and the connection function $f$ depends on $w_i$ and $w_j$ through some \emph{metric distance} $d_{ij}=d(w_i,w_j)$ (in which case, $f$ generally loses the property of being monotonic in $w_i$~\footnote{For instance, if $w_i$ and $w_j$ are scalar (one-dimensional) coordinates and their distance is defined as $d_{ij}=|w_i-w_j|$, then $w_i$ and $w_j$ can simultaneously increase while keeping $d_{ij}$ (and hence $p_{ij}$) unchanged.}).
A remarkable example of this success is the \emph{hyperbolic model}~\cite{krioukov2016clustering, aliakbarisani2025clustering, serrano2006clusteringI, serrano2006clusteringII, boguna2021network, newman2009random, krioukov2010hyperbolic, michielan2022detecting}, which assumes that nodes are sprinkled uniformly at random in a $D$-dimensional hyperbolic space and links are formed with a probability that is a certain decreasing function of the geodesic distance between nodes.
In this setting, the triangle inequality ($d_{ij}\le d_{ik}+d_{jk}$ $\forall i,j,k$) guarantees that pairs of nodes $i,j$ that successfully connect to a third node $k$ are very likely to connect to each other (\emph{triadic closure}), provably leading to positive clustering~\cite{allard2024geometric, candellero2014clustering}.
Additionally, hyperbolicity ensures that other desirable properties such as power-law degree distributions emerge~\cite{krioukov2010hyperbolic}.
These investigations have convincingly concluded that, in random graphs with conditionally independent edges and distance-dependent connection probabilities, {\em geometry implies clustering}.

On the other hand, the potential reverse (non-obvious) implication that \emph{clustering implies geometry}, i.e., that the presence of finite clustering could suffice to indicate that the network is (with high probability) the realization of a random geometric graph, is highly debated~\cite{krioukov2016clustering, aliakbarisani2025clustering, serrano2006clusteringI, serrano2006clusteringII, boguna2021network, newman2009random, krioukov2010hyperbolic, michielan2022detecting}.
Whether clustering implies geometry is crucial for network theory (because, if true, it would demand for explanations and interpretations of the hidden geometry of real-world networks~\cite{allard2024geometric, boguna2021network}), for data science (for instance because techniques such as graph embedding and topological data analysis could then be optimized for real-world network geometries), and for the control of processes on networks (because geometry could aid network navigability~\cite{krioukov2010hyperbolic,boguna2009navigability,boguna2010sustaining}, and the understanding of how clustering impacts epidemic and other processes~\cite{newman2003properties}). 

In this Letter, we prove that positive clustering can emerge {\em independently} of geometry or higher-order constraints in random graphs with independent edges conditional on infinite-mean node weights~\cite{garuccio2023multiscale, lalli2024geometry}. While the finite-mean case systematically produces vanishing clustering~\cite{van2016random}, we show that in the infinite-mean case clustering is positive and strong. 
Notably, the infinite-mean fitness also implies the breakdown of self-averaging in various properties, including clustering itself.

\paragraph*{The Multi-Scale Model (MSM).}  
We consider the inhomogeneous random graph model introduced in~\cite{garuccio2023multiscale} and studied in~\cite{avena2022inhomogeneous,lalli2024geometry,  catanzaro2025spectra}. 
Each node $i$ is assigned a positive weight (fitness) $w_i$, sampled i.i.d.\ from a certain pdf $\rho(w)$. 
Given $w_i,w_j$, nodes $i,j$ connect with probability 
\begin{equation}
    p_{ij}=1-\e^{-\delta_n w_i w_j}, \label{eq:pizza}
\end{equation}
where the scaling $\delta_n$ can be chosen in such a way  (e.g., $\delta_n=n^{-1/\alpha}$ for the case $\rho(w)\propto w^{-1-\alpha}$ considered below) that the link density vanishes as $n\to\infty$~\cite{avena2022inhomogeneous}.
The connection probability~\eqref{eq:pizza} is also known in the literature as the Norros–Reittu model~\cite{norros2006conditionally} and has been used for various purposes~\cite{rodgers2005eigenvalue,bhamidi2010scaling}, by considering a finite-mean $\rho(w)$. 
In this work we are however interested in the specific `annealed' regime considered within the framework of network renormalization~\cite{gabrielli2025network, catanzaro2026renormalization}, where the model emerges as a fixed point of a renormalization flow in the space of models, induced by coarse-graining the network by aggregating nodes into super-nodes and recalculating the induced connection probability between super-nodes~\cite{garuccio2023multiscale}. 
The specific connection rule in~\eqref{eq:pizza} makes the model invariant, provided the weight of each block is defined as (a random variable equal to) the sum of the weights of the constituent nodes~\cite{garuccio2023multiscale}. 
If one further demands the invariance of the weight distribution, i.e., that the coarse-grained weights are distributed according to the same pdf as the constituent weights, then $\rho(w)$ should be a positively-supported, one-sided $\alpha$-stable law $\tilde{\rho}_\alpha(w)$ with stability parameter $0<\alpha<1$~\cite{samorodnitsky1994stable,Nola20} and tails decaying as $w^{-1-\alpha}$ for large $w$, hence with diverging mean and higher moments~\cite{garuccio2023multiscale,avena2022inhomogeneous}.
In this way, one gets a model interpretable as describing connections between arbitrarily nested (blocks of) nodes, describing the graph at multiple scales of resolution (coarse-graining) simultaneously~\cite{garuccio2023multiscale}.

Since a stable law $\tilde{\rho}_\alpha(w)$ with $0<\alpha<1$ is known in closed form only for $\alpha=1/2$ (L\'evy distribution), we first follow~\cite{avena2022inhomogeneous} and replace it with a pure Pareto distribution
\begin{equation}
    \rho_\alpha(w) = \alpha w^{-1-\alpha}, \quad w\ge1, \; 0<\alpha<1, \label{eq:pareto}
\end{equation}
(which has exactly the same tail behaviour as an $\alpha$-stable distribution and serves as an analytically tractable counterpart~\cite{samorodnitsky1994stable,avena2022inhomogeneous}) and then discuss the extension to actual $\alpha$-stable laws $\tilde{\rho}_\alpha(w)$ via both analytical and numerical arguments. 
 Even though the invariance of the MSM under node aggregation holds for any choice of $\delta_n$, one can show~\cite{avena2022inhomogeneous} that, given~\eqref{eq:pareto}, choosing $\delta_n=n^{-1/\alpha}$ implies that the link density vanishes as $\log n/n$, i.e., the average degree grows only as $\log n$, and the degree distribution $P(k)$ is power-law with universal tail $P(k)\propto k^{-2}$, irrespective of the value of $\alpha\in(0,1)$. See Supplementary Information (SI) for more details. This is the relevant sparse regime that we are interested in here~\footnote{We notice that the MSM can be enriched with additional dyadic features~\cite{garuccio2023multiscale,lalli2024geometry} which allow the connection probabilities to depend, for instance, on geometric distances between coordinates assigned to nodes. Nonetheless, to investigate whether clustering can be produced \emph{without} geometry, here we obviously consider the simplest, distance-independent version in~\eqref{eq:pizza}.}.

\begin{figure*}[bt!]
\centering
    \includegraphics[width=0.32\linewidth, height= 0.2\linewidth]{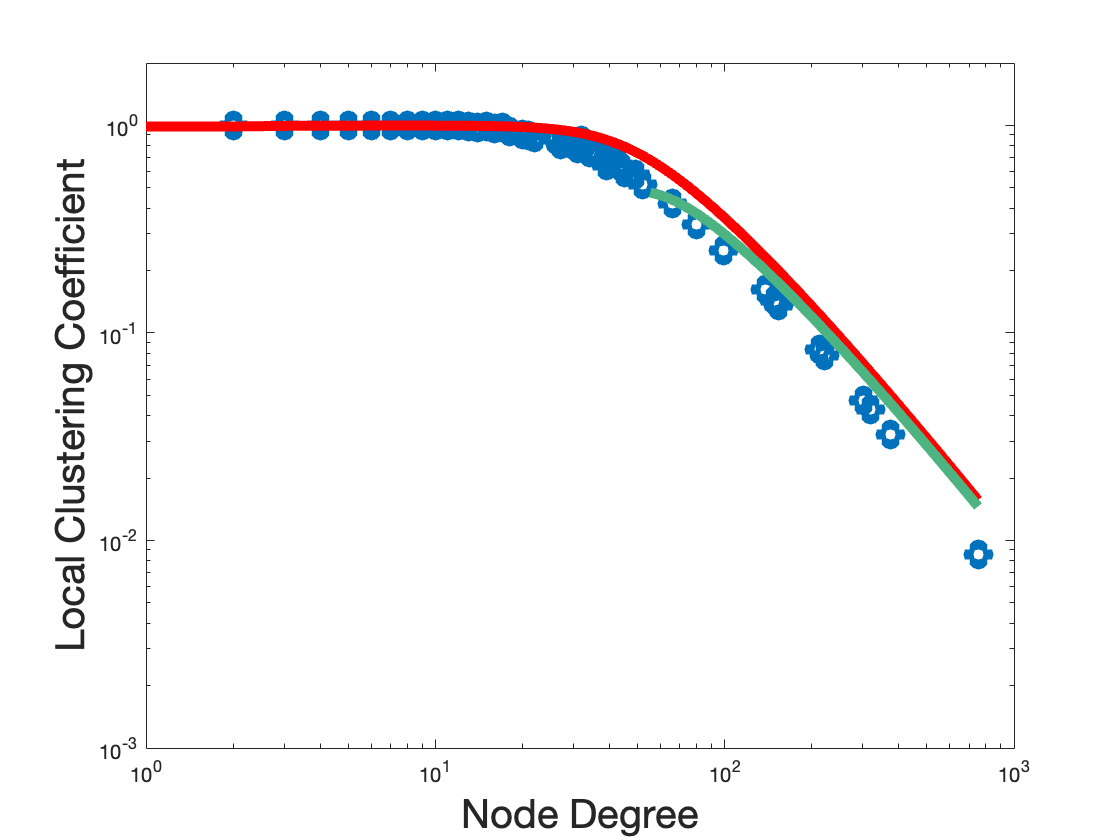} 
    \includegraphics[width=0.32\linewidth, height= 0.2\linewidth]{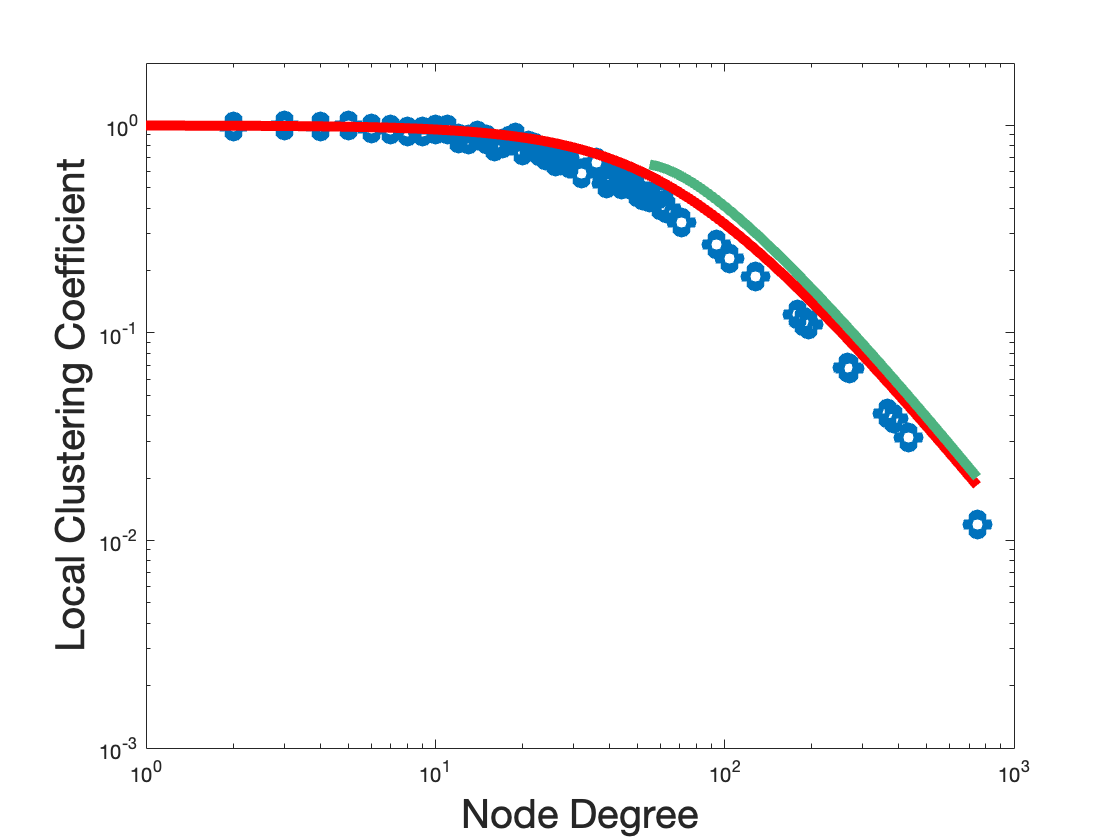}
    \includegraphics[width=0.32\linewidth, height= 0.2\linewidth]{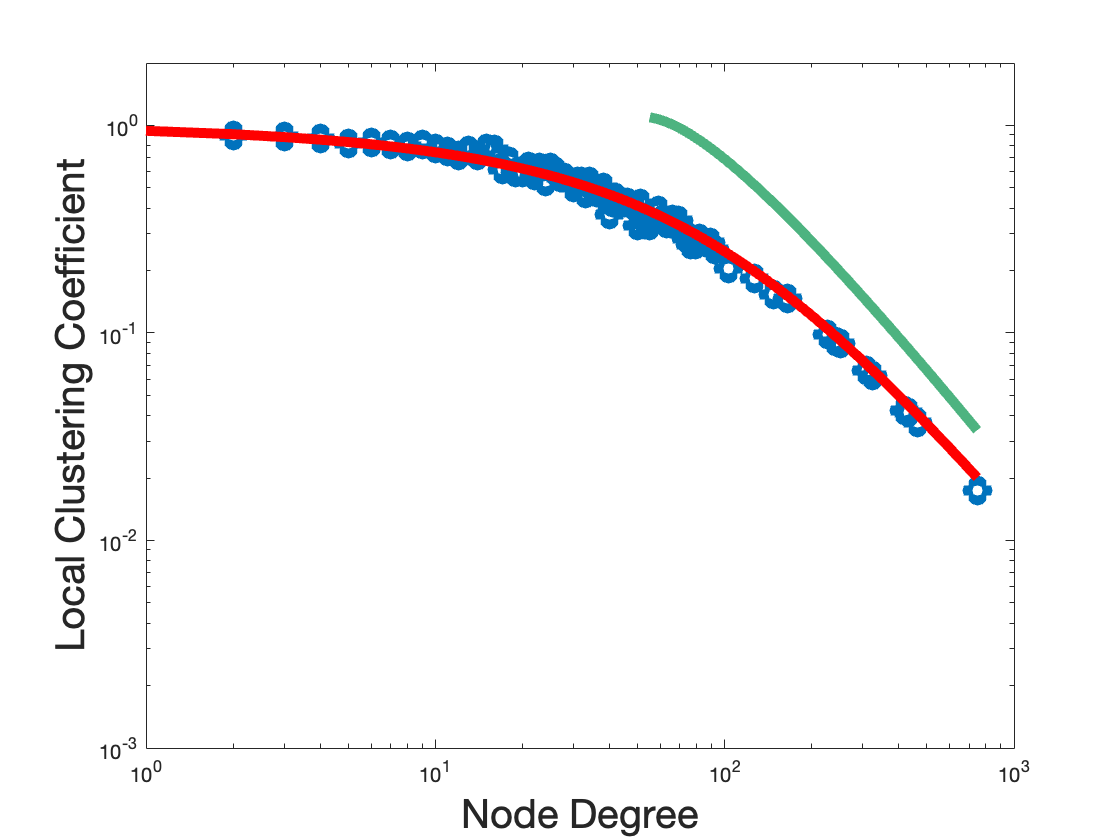}\\
    \includegraphics[width=0.32\linewidth, height= 0.2\linewidth]{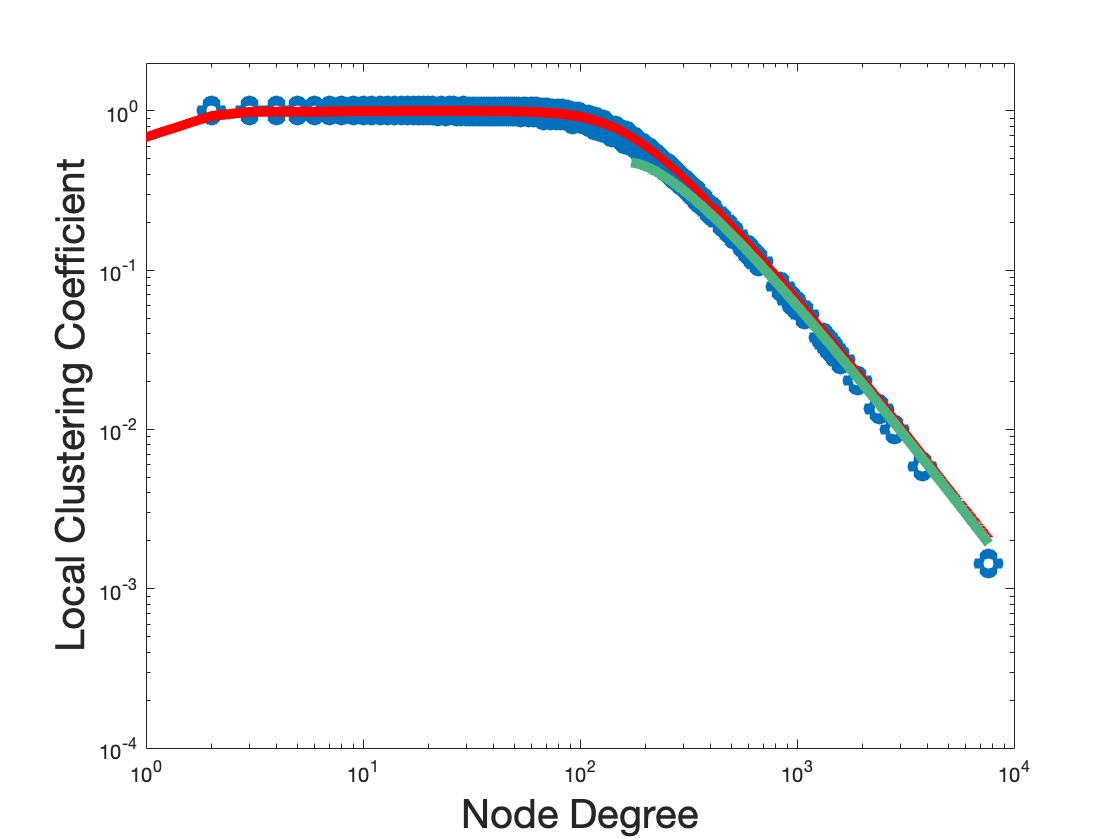} \includegraphics[width=0.32\linewidth, height= 0.2\linewidth]{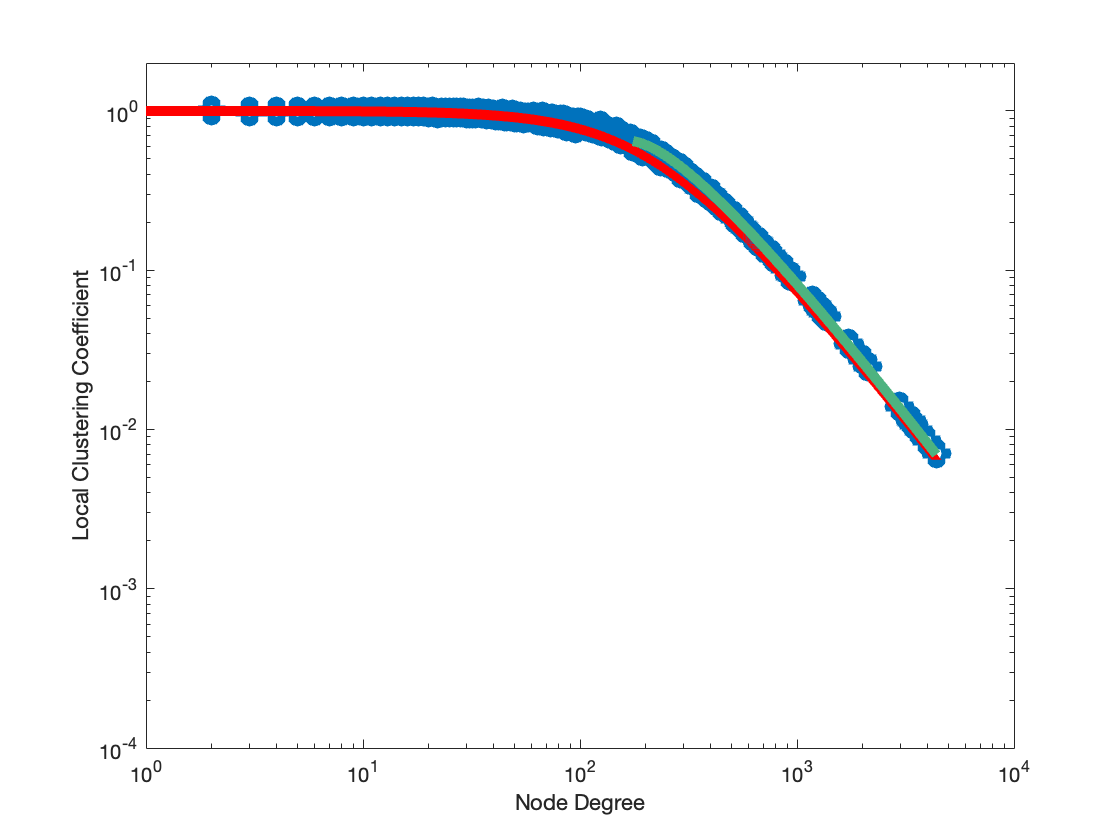}
    \includegraphics[width=0.32\linewidth, height= 0.2\linewidth]{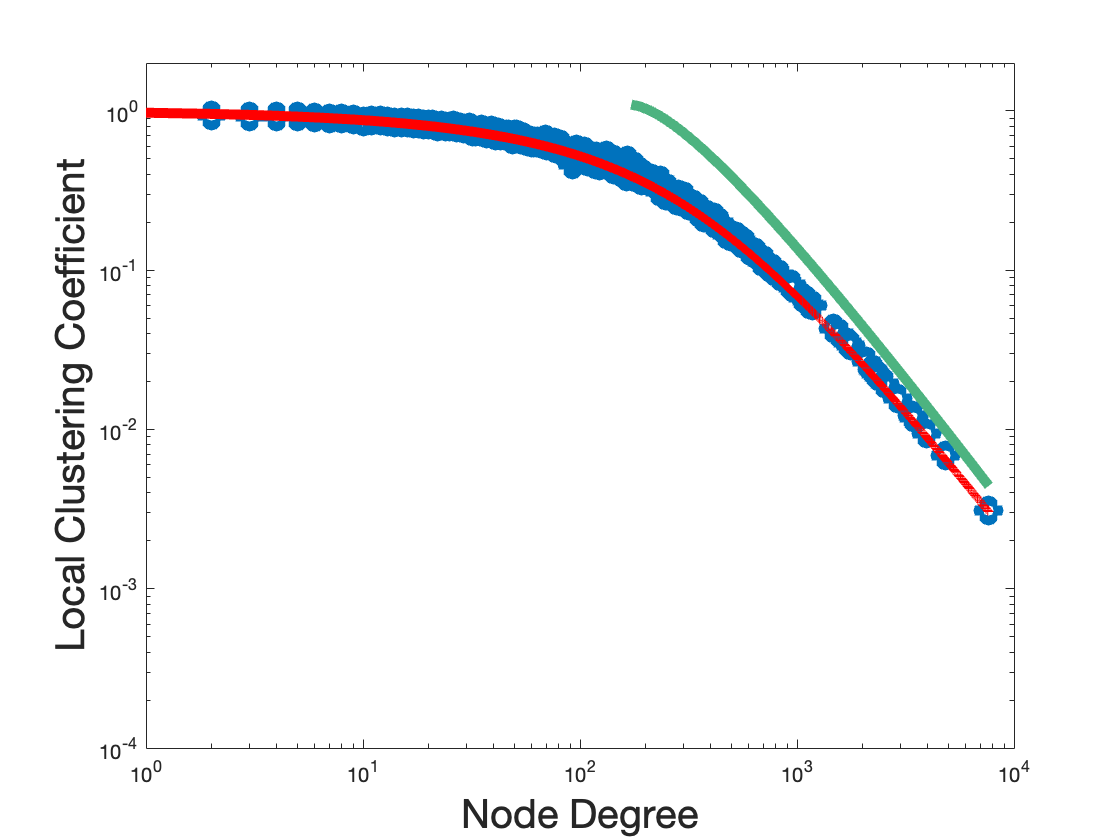}
    \caption{{\bf Clustering functions for different values of $n$ and $\alpha$ for the MSM~\eqref{eq:pizza} with Pareto weights~\eqref{eq:pareto}.} Blue circles: empirical clustering function~\eqref{eq:cc_empirical} versus the degree $k$ computed on actual realized graphs (obtained by sampling the weights once, and sampling the graph once conditionally on the realized weights). Red curves: our analytical expression~\eqref{eq:analytical_clustering} for the annealed clustering function evaluated on the actual degree $k=a \sqrt{n}$. 
    Green curves: our asymptotic calculation~\eqref{eq:CC_hub} valid for diverging reduced degrees. We evaluate and plot all functions only for degrees larger than $1$, to avoid ambiguities in the definition of the clustering coefficient for $k<2$.
    Top: $n=10^3$, bottom: $n=10^4$. From left to right: $\alpha=0.3,0.5,0.7$.}
    \label{fig:clusterings}
\end{figure*}

\paragraph*{The Local Clustering Coefficient.}
Abstractly, the clustering coefficient quantifies the tendency of a graph to form closed triangles. Globally, it can be defined as the ratio between the total number of closed triangles and the number of potential  (either open or closed) triangles (also called cherries, or wedges). In our regime with $0<\alpha<1$, this ratio vanishes for the MSM~\cite{avena2022inhomogeneous}, in line with what is observed in sparse real-world networks. 
Our interest here is the (more distinctive) {\em local} clustering coefficient $C_v$, which quantifies the tendency of the neighbours of the specific node $v$ to form closed triangles around $v$.
For a node $v$ with degree $D_v$, $C_v$ is defined as the ratio between the number $\Delta_v$ of triangles containing $v$ and the number of wedges centered at $v$. In terms of the adjacency matrix $A=(A_{ij})_{i,j=1}^n$, this reads 
\begin{equation}
    C_v = \frac{\Delta_v}{\binom{D_v}{2}} = \frac{\sum_{i\ne v}\sum_{j\ne v,i} A_{vi}A_{ij}A_{jv}}{D_v (D_v - 1)},\quad D_v > 1.\label{eq:CC_exact}
\end{equation}
The behaviour of $C_v$ versus $D_v$, as well as the average value $C$ of $C_v$ over nodes, can effectively characterize real-world networks and discriminate among different models -- especially in determining whether a finite (expected) value of $C$ can emerge in large sparse graphs, as observed in real-world networks. 
Note that $C_v$ in~\eqref{eq:CC_exact} is defined only for nodes with $D_v > 1$, as only nodes with at least two neighbours can host a triangle. 
For nodes with $D_v = 0, 1$, one can either leave $C_v$ undefined (and exclude it from the calculation of $C$) or conventionally set $C_v\equiv 0$ (and include it in the calculation of $C$). We will consider both cases and discuss their implications.

\paragraph*{Clustering Function.}
For a given graph, we consider the average clustering coefficient as a function of the degree $k$ of each node, i.e., the \emph{empirical clustering function}
\begin{equation} 
C(k) = \frac{1}{N_{k}} \sum_{v\colon D_v = k} \frac{\Delta_v}{\binom{k}{2}},
\label{eq:cc_empirical}
\end{equation}
where $N_k = \sum_{v\colon D_v = k} \delta_{D_v,k}$ is the number of nodes with degree $k$. In terms of expectations ($\EX$) over the model, we introduce the {\em annealed clustering function} as
\begin{equation}
    \bar{C}(k) = \frac{1}{\EX[N_k]} \EX \left[ \sum_{v\colon D_v = k} \frac{\Delta_v}{\binom{k}{2}} \right]. \label{eq:cc_annealed}
\end{equation}
One of our main results is a rigorous calculation of $\bar{C}(k)$ in the asymptotic (large $n$) limit, valid for both small and large $k$, which shows excellent agreement with the empirical clustering function $C(k)$~\footnote{We also believe that, through a second-moment analysis, it could be proven that in the large $n$ limit, with high probability, the two clustering functions become the same. We will however not prove this claim in the present work.}. 
In this setting, we redefine $\bar{C}(\cdot)$ as a function of the \textit{reduced degree} $a = {k}/{\sqrt{n}}$, as this rescaling allows us to prove that, defining $g(x)\equiv 1-\e^{-x}$, the limiting form of $\bar{C}(a)$ as $n \to \infty$ is 
\begin{equation}
    \bar{C}(a) = \frac{\alpha^2}{a^2}\! \int_{0}^\infty \!\!\!\!\!{\mathrm d}x
    \int_{0}^\infty  \!\!\!\!\!{\mathrm d}y\,\frac{g(a^{1/\alpha} {\tau_\alpha}x)}{x^{1+\alpha}}
    \frac{g(a^{1/\alpha} {\tau_\alpha}y)}{y^{1+\alpha}} g(x y), 
    \label{eq:analytical_clustering}
\end{equation}
where ${\tau_\alpha} \equiv \Gamma(1-\alpha)^{-1/\alpha}$, and $\Gamma(\cdot)$ is the Euler Gamma function. The proof is in SI, and here we intuitively explain the central ideas. 
The restriction of vertices to have degree $k$, and division by $\EX[N_k]$, effectively conditions on the degree being $k$. A vertex of large weight $w$ has degree close to $\Gamma(1-\alpha)w^{\alpha}$. Conditioning a vertex to have  degree $k=a\sqrt{n}$ thus corresponds to conditioning its weight to be close to $a^{1/\alpha}\tau_\alpha n^{1/(2\alpha)}.$ In order for a triangle with vertices of weights $a^{1/\alpha}\tau_\alpha n^{1/(2\alpha)}$, $w_1$ and $w_2$ to have positive probability, we thus need $w_1$ and $w_2$ to be of order $n^{1/(2\alpha)}$. Write $w_1=xn^{1/(2\alpha)}$ and $w_2=yn^{1/(2\alpha)}$: the densities of the weights of these vertices give two factors $\alpha (xn^{1/(2\alpha)})^{-\alpha-1}$ and $\alpha (xn^{1/(2\alpha)})^{-\alpha-1}$. After rescaling the double sum over pairs of vertices with weights $w_1$ and $w_2$ to integrals, the powers of $n$ are canceled by the number of pairs of vertices and the rescaled degree $1/k^2=1/(a^2 n)$, which also explains the factor $a^{-2}$. Finally, the factors involving $g$ in \eqref{eq:analytical_clustering} arise from the edge probabilities in \eqref{eq:pizza}. In Figure~\ref{fig:clusterings}, we compare the numerical evaluation of the integral~\eqref{eq:analytical_clustering} with the direct computation of the empirical clustering function~\eqref{eq:cc_empirical} on the actual realizations of the random graph. The agreement is remarkably good, already for moderate values of $n$. 

\paragraph*{Hub and Leaf Behaviour.}
To refine the characterization of clustering in the model, we investigate the limiting behavior of the annealed clustering function for nodes in two distinct  regimes of the degree spectrum: \emph{leaf} and \emph{hub} nodes, corresponding to poorly and highly connected ones, respectively. 
For leaf ($L$) nodes, defined as nodes with vanishing reduced degree $a \to 0$, we find (see SI) that the function in~\eqref{eq:analytical_clustering} has the asymptotic plateau 
\begin{equation}
    \bar{C}_L \equiv \lim_{a \to 0} \bar{C}(a) = 1,
    \label{eq:CC_leaf}
\end{equation}
which implies that in the low-degree limit the local neighborhood of a node is almost fully clustered (excluding of course nodes with $k=0,1$). This behaviour can be intuitively explained in terms of an interconnected core of hubs coexisting with a majority of low-degree nodes, whose few links are largely connected with the hubs. Therefore, low-degree nodes have a high chance to close triangles between neighbours because the hubs are densely connected among themselves.
For hub ($H$) nodes, defined as nodes with diverging reduced degree $a \to \infty$, we find (see SI) that the function~\eqref{eq:analytical_clustering} decays as
\begin{equation}
    \bar{C}(a) \sim \bar{C}_{H}(a) \equiv 2 \Gamma(1-\alpha)\frac{\log a}{a^2},\quad a\to\infty.
    \label{eq:CC_hub}
\end{equation}
Figure~\ref{fig:clusterings} shows that Eqs.~\eqref{eq:CC_leaf} and~\eqref{eq:CC_hub} are in excellent agreement with both the empirical function~\eqref{eq:cc_empirical} computed on actual realizations of the graph and the full analytical expression~\eqref{eq:analytical_clustering} for the annealed function~\eqref{eq:cc_annealed}.

\paragraph*{Average Clustering Coefficient.}
The previous results finally allow us to estimate the overall expected value $\bar{C}$ of the average local clustering coefficient $C$ by averaging the annealed clustering function $\bar{C}(a)$ over the distribution $P(a)$ of the reduced degree $a$ as
\begin{eqnarray}
    \bar{C} &\equiv& \int da~ \bar{C}(a) P(a) \nonumber\\
    &\simeq& \int_{a\in L} {\mathrm d}a~ \bar{C}_L(a) P(a) + \int_{a\in H} {\mathrm d}a~ \bar{C}_H(a) P(a),
    \label{eq:cc_avg} 
\end{eqnarray}
where we have formally split the integral into two contributions, produced by leaf ($a\in L$) and hub ($a\in H$) nodes respectively, consistent with the two regimes characterized asymptotically in~\eqref{eq:CC_leaf} and~\eqref{eq:CC_hub}. 
A rigorous definition of this split is provided in SI, where we also show that, in order to compute $\bar{C}$ exactly, it is not necessary to identify precisely the crossover value(s) of $a$ separating the two regions. 
Notably, it turns out that the integral over $a\in H$ in~\eqref{eq:cc_avg} vanishes, implying that hubs do not contribute to the overall clustering (see SI). 
The remaining contribution of the leaf nodes $a\in L$ is always finite, but its precise value depends on whether $C_{k}$ for nodes with degree $k<2$ is defined as $C_{k<2}\equiv0$ and hence included in the integration over $L$ (in which case $\bar{C}_L(a)\equiv 0$ for $0\le a\le 1/\sqrt{n}$) or undefined and not included in the integration. 
In SI we find that, depending on the choice,
\begin{align}
    \lim_{n\to\infty} \bar{C} = \begin{cases}
        1 & \text{ if } C_{k<2} \text{ undefined}, \\
        1 - r_{0/1} & \text{ if } C_{k<2} \equiv 0,
    \end{cases} \label{eq:cc_final} 
\end{align}
where $r_{0/1}$ is the expected fraction of isolated ($k=0$) and pendant ($k=1$) nodes. We plot both versions of the coefficient in Figure~\ref{fig:averagec}, confirming the agreement with our theoretical calculations. This result indicates that the network is strongly clustered: for almost all nodes that can form triangles, almost all triangles are closed. 

\begin{figure}[t]
    \centering \includegraphics[width=\linewidth, height=0.6\linewidth]{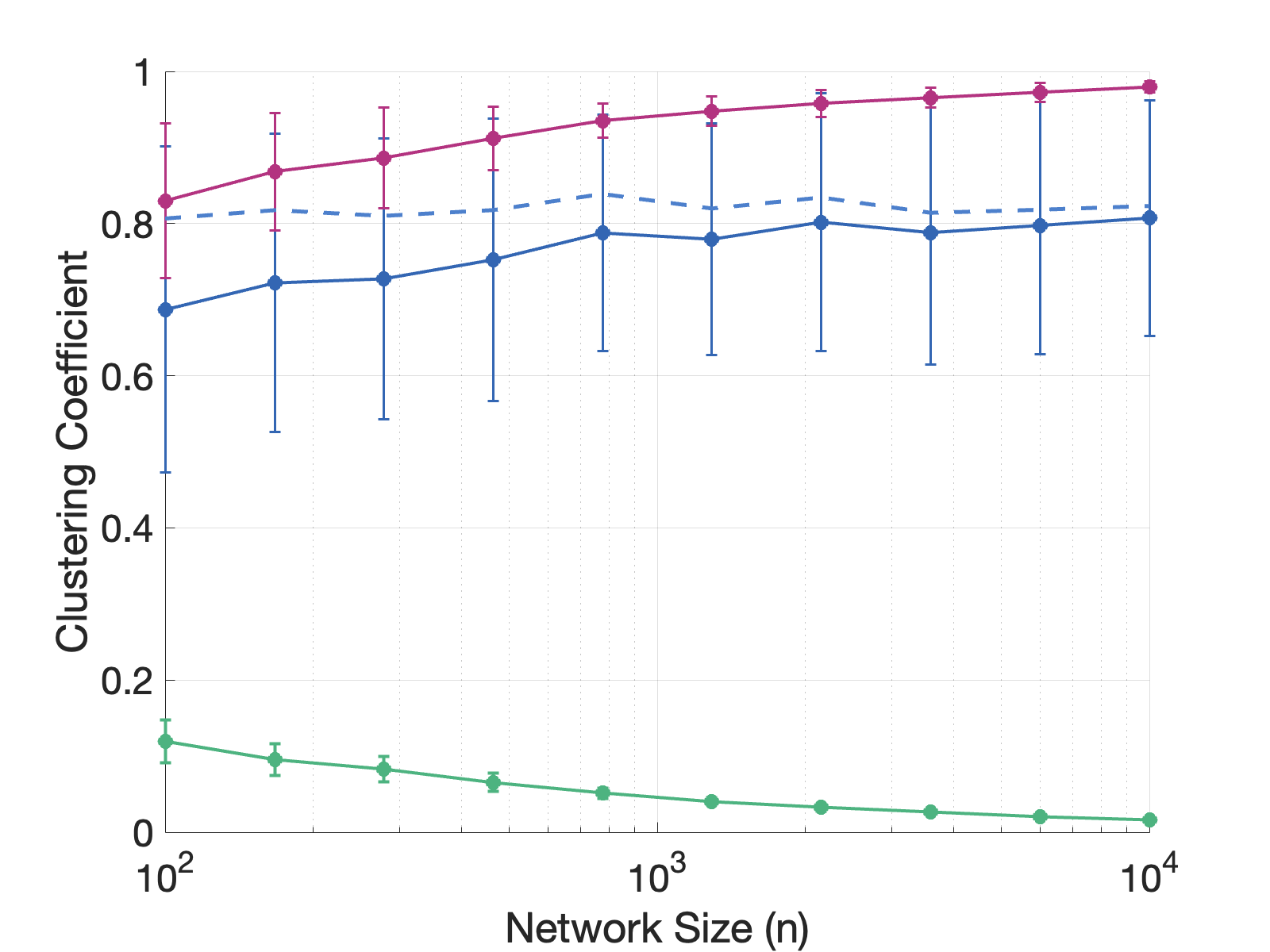}
    \caption{ \textbf{Average local clustering coefficient $C$ and distance to $r_{0/1}$ versus network size $n$ (for $\alpha=0.5$).} Results are obtained by sampling the weights $10^3$ times independently for each value of $n$, sampling a single graph conditionally on each realization of the weights, computing the relevant quantities on that graph, and finally calculating averages and standard deviations over the $10^3$ realizations. The trends show the node-averaged local clustering coefficient ${C}$, both including (blue) and excluding (purple) nodes with degree $k<2$: the latter evolves smoothly towards 1 with shrinking error bars as $n$ increases, while the former fluctuates with non-vanishing error bars and becomes progressively closer to $1-r_{0/1}$ over realizations (dashed blue line), which is also fluctuating. The difference between $1-r_{0/1}$ and ${C}$ including nodes with $k<2$ (green) vanishes with shrinking error. }
    \label{fig:averagec}
\end{figure}

\paragraph*{Lack of Self-Averaging.}
The result~\eqref{eq:cc_final} is calculated conditionally on the realized weights $\{w_i\}_{i=1}^n$.
It shows that, if nodes with $k<2$ are excluded, $\bar{C}$ converges to the (deterministic) maximum value 1, with vanishing fluctuations. If one looks instead at isolated and pendant nodes (see SI), the fractions concentrate around functions of the rescaled total weight $S_n=\delta_n\sum_{j=1}^nw_j$, which in turn converges in distribution to an $\alpha$-stable random variable. This remains a fluctuating quantity even in the large $n$ limit, as proven in SI and confirmed numerically in Figure~\ref{fig:averagec}. 
This means that, for given $n$ and $\alpha$, different realizations of the weights produce different values of $r_{0/1}$ and of $\bar{C}$, if nodes with $k<2$ are included. Therefore, under this second definition, $\bar{C}$ converges in distribution to the random variable $1-r_{0/1}$, remaining macroscopically fluctuating even in the large graph limit. 
Notably, the lack of self-averaging extends to further `low degree' nodes, as we show in SI. This means that any macroscopic property that is a function of the full degree distribution, and not just its tail, will in general keep memory of the fluctuations of the fractions of low-degree nodes. 
Figure~\ref{fig:macrosfluc} confirms that other macroscopic properties (spectral gap, largest connected component size) besides the clustering coefficient keep fluctuating in the thermodynamic limit.

\begin{figure}[t]
    \centering \includegraphics[width=\linewidth]{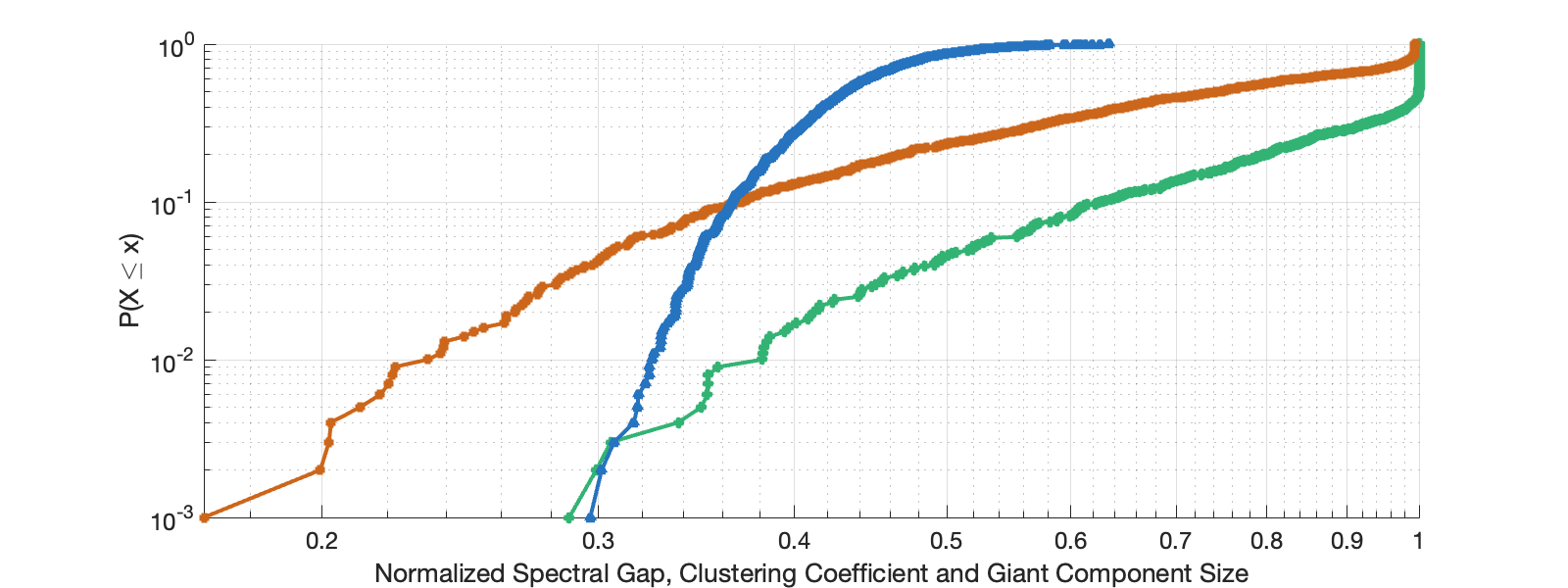}
    \caption{ \textbf{Cumulative distributions of network properties across $10^3$  realizations with resampled weights ($n=10^4$, $\alpha=0.5$).} Orange: average local clustering coefficient including isolated and pendant nodes. Green: fraction of nodes in the largest connected component. Blue: normalized spectral gap (difference between the two largest eigenvalues of the adjacency matrix, divided by the largest one).  
    } 
    \label{fig:macrosfluc}
\end{figure}

\paragraph*{Concluding Remarks.}
This Letter contributes to the open debate about whether, in models with conditionally independent edges, the coexistence of sparsity and local clustering (along with other desirable features such as a broad degree distribution and the small-world property) can be attained \emph{without geometry}, i.e., without assuming that the connection probability $p_{ij}$ depends on the node attributes $w_i$ and $w_j$ necessarily through a function of some metric distance $d(w_i,w_j)$~\cite{krioukov2016clustering, aliakbarisani2025clustering, serrano2006clusteringI, serrano2006clusteringII, boguna2021network, newman2009random, krioukov2010hyperbolic, michielan2022detecting,allard2024geometric}.
We considered a geometry-free model with independent edges and infinite-mean fitness~\cite{garuccio2023multiscale, lalli2024geometry}, recently introduced in the context of network renormalization~\cite{gabrielli2025network}, and calculated its full clustering function -- rigorously proving the resulting \emph{finite} local clustering in the large $n$ limit and highlighting that clustering does \emph{not} imply geometry.
The infinite-mean nature of the fitness arises from a requirement of model invariance under node aggregation~\cite{garuccio2023multiscale} and
generates other realistic properties including sparsity, a power-law degree distribution~\cite{lalli2024geometry}, and the breakdown of self-averaging of various network properties (including the fraction of isolated nodes and its contribution to the average clustering itself), which is a novel result characterized here.
The model then exposes, as an alternative to geometry, the hypothesis of node aggregation invariance as a basic candidate mechanism producing realistic network properties.
In particular, since the MSM is the fixed point of a renormalization transformation that aggregates nodes into supernodes~\cite{garuccio2023multiscale}, it is also an  attractor for an associated renormalization flow acting on an appropriate basin of more general random graph models~\cite{avena2026renormalisation}. 
It therefore represents a natural effective model obtained from the iteration of a fundamental coarse-graining move, corresponding to a change in the resolution at which vertices are observed. Concurrently, the $\alpha$-stable distribution is the only non-Gaussian attractor for positive node-additive weights~\cite{samorodnitsky1994stable,Nola20}, thereby providing a natural mechanism for the onset of effective hidden variables of the type considered here.

\paragraph*{Acknowledgments}
This publication is part of the projects ``Network renormalization: from theoretical physics to the resilience of societies'' with file number NWA.1418.24.029 of the research programme NWA L3 - Innovative projects within routes 2024 (grant \url{https://doi.org/10.61686/PBSEC42210}) and ``Redefining renormalization for complex networks'' with file number OCENW.M.24.039 of the research programme Open Competition Domain Science Package 24-1 (grant \url{https://doi.org/10.61686/AOIJP05368}), which are (partly) financed by the Dutch Research Council (NWO).
The work of RvdH is supported in part by NWO through the Gravitation {\sc Networks} grant 024.002.003.
Early interactions between RvdH and DG took place between 29/01/2024 and 02/02/2024 during the workshop ``Challenges in Networks'' (ICTS/NETWORKS2024/01) at the ICTS-TIFR in Bangalore, India.

\nocite{*}

\bibliography{Biblio_clustering}

%
%
\clearpage
\newpage
\setcounter{equation}{0}

\setcounter{figure}{0}
\setcounter{table}{0}
\setcounter{page}{1}
\setcounter{section}{0}

\renewcommand{\theequation}{S\arabic{equation}}
\renewcommand{\thetable}{S\arabic{table}}
\renewcommand{\thefigure}{S\arabic{figure}}
\renewcommand{\thesection}{S.\Roman{section}} 
\renewcommand{\thesubsection}{\thesection.\roman{subsection}}

\onecolumngrid
{\center
\textbf{SUPPLEMENTARY INFORMATION}\\
$\quad$\\
accompanying the paper\\
\emph{``Clustering without geometry in sparse networks with independent edges''}\\
by A. Catanzaro, R. van der Hofstad and D. Garlaschelli\\
$\quad$\\
$\quad$\\
}

\section{Average Degree}
\label{app:logdegree}
The MSM, with the scaling $\delta_n = n^{-1/\alpha}$ and node fitness distributed as $\rho_\alpha(w) = \alpha w^{-1-\alpha}$ for $w\ge1$ with $0<\alpha<1$, exhibits the remarkable property of having an average degree that grows only as $\log n$, so that the link density vanishes as $\log n/n$, despite the infinite-mean nature of the fitness. 
We sketch here the proof of this property (details can be found in \cite{avena2022inhomogeneous}). The main idea is to evaluate the Laplace Transform of the product of two Pareto-distributed random variables ($w_1, w_2$). 
Specifically, let us consider the variables $\log w_1$ and $\log w_2$, which are then distributed according to an Exponential distribution, or a Gamma distribution with shape parameter $2$ and scale $1/\alpha$. 
Then
\begin{equation}
    \prob(w_1 w_2 > x) =\prob(\log w_1 + \log w_2 > \log x) 
    = \alpha^2 \int_{\log x}^{\infty} dy\, y\, \e^{-\alpha y} 
    = \alpha^2 \int_{x}^{\infty} dt\, t^{-\alpha-1}\,\log t,
\end{equation}
and the last term can be evaluated in the large $x$ limit as $\sim \alpha x^{-\alpha} \log x$ (see \cite[Lemma 4.1]{jessen2006regularly} for a thorough and general  computation).
From Karamata's Tauberian theorem \cite[Theorem 8.1.6]{Bingham_Goldie_Teugels_1987}, we can evaluate the Laplace transform of the product $w_1 w_2$ in the same diverging regime. Indeed, if one has the random variable $W$ with distribution $P$ that follows $1-P(W) \sim W^{-\alpha} L(W)$, where $L(W)$ is a slowly varying function and $W \to \infty$, then the Laplace transform $\tilde{P}(s)$ can be evaluated from
\begin{equation}
    1 - \tilde{P}(s) \sim \Gamma(1-\alpha) s^{\alpha} L(1/s),
\end{equation}
as $s \to 0$. From this result, we can evaluate the expected value of the degree $k_i$ of a generic node $i$ as
\begin{equation}
    \EX[k_i] = \sum_{j \neq i}\EX[1-\e^{-\delta_n w_i w_j}] = (n-1)\EX[1-\e^{-\delta_n w_1 w_2}] \sim \alpha (n-1) \Gamma(1-\alpha) \delta_{n}^{\alpha} \log\frac{1}{\delta_{n}}.
\end{equation}
As all nodes are statistically equivalent, for the scaling $\delta_{n}= n^{-1/\alpha}$, we obtain the scaling for the expected average degree as
\begin{equation}
    \EX[k] = \EX[k_i] \sim  \Gamma(1-\alpha) \log n,
\end{equation}
which proves the fact that the average degree of the model grows only logarithmically in the large $n$ regime, so that the expected link density $\EX[k]/(n-1)$ vanishes as $\log n /n$.

\section{Annealed Clustering Function}
\label{app:firstMoments}

\begin{theorem}
Let the vertex space of the network be $[n]=\{1, \ldots, n\}$. 
Let $(W_v)_{v\in[n]}$ be i.i.d.\ Pareto random variables, so that their density equals

\begin{equation}
\label{density-Pareto}
    \rho_\alpha(w)=\frac{\alpha}{w^{\alpha+1}}, \qquad w\geq 1,
\end{equation}
with $\alpha\in(0,1)$, so that $\EX[W]=\infty.$ Conditionally on $(W_v)_{v\in[n]}$, edges are present independently, with the edge between $u,v\in [n]$ being present with probability
\begin{equation}
    p_{uv}=1-\e^{-W_u W_v/n^{1/\alpha}}.
\end{equation}
Throughout, we let $(I_{u,v})_{1\leq u<v\leq n}$ denote the edge indicator variables (or the elements of the adjacency matrix), so that
\begin{equation}
    \prob\big(I_{u,v}=1\mid (W_v)_{v\in[n]}\big)=1-\e^{-W_uW_v/n^{1/\alpha}},
\end{equation}
and $(I_{u,v})_{1\leq u<v\leq n}$ are conditionally independent given $(W_v)_{v\in[n]}$. Let the Annealed Clustering Function be defined as
\begin{equation}
    \bar{C}(k) = \frac{1}{\EX\left[N_k\right]} \EX \left[ \sum_{v/D_v = k} \frac{\Delta_v}{\binom{k}{2}} \right] \label{eq:CC_annealed_app}
\end{equation}
Then, as $n\to\infty$, for $k = a \sqrt{n}$,
\begin{equation}
    \bar{C}_n(a\sqrt{n})\to \bar{C}(a)=\frac{\alpha^2}{a^2} \int_{0}^{\infty} \int_{0}^{\infty}[1-\e^{-xa^{1/\alpha} {\tau_\alpha}}][1-\e^{-ya^{1/\alpha} {\tau_\alpha}}][1-\e^{-xy}](xy)^{-(\alpha+1)}
	{\mathrm d}x {\mathrm d}y, \label{eq:analytical_clustering_app}
\end{equation}
where ${\tau_\alpha} = \Gamma(1-\alpha)^{-1/\alpha}$.
\end{theorem}

\begin{proof}
    We compute 
\begin{equation}
\EX[N_k]=n\EX\big[\prob(D_1=k\mid W_1)\big],
\end{equation}
where we note that, conditionally on $W_1=w$, $D_1 \sim \operatorname{Bin}\left(n-1, q_w\right),$ where

\begin{equation}
    q_w=\EX\big[1-\e^{-w W/n^{1/\alpha}}\big].
\end{equation}

We also compute the numerator 
\begin{equation}
    \EX\Big[\sum_{v\in [n]}\Delta_v\mathbbm{1}_{D_v=k}\Big] =n\EX\Big[\EX\Big[\Delta_1\mathbbm{1}_{D_1=k}\mid W_1\Big]\Big] =n{{n-1}\choose{2}}\EX\Big[\prob(D_1'=k-2\mid W_1)\prob(I_{1,2}I_{1,3}I_{2,3}=1\mid W_1)\Big],
\end{equation}
where $D_1'\sim \operatorname{Bin}(n-3, q_w)$, and we have used conditional independence. Here, we can identify
\begin{equation}
D_1'=\sum_{u\in[n]\setminus [3]} I_{1,u}.
\end{equation}

We see that all these expectations involve binomial random variables, which we continue to investigate, starting with the conditional success probability $q_w$.

\subsection{Analysis of success probabilities}
We compute
\begin{align}
    q_w &=\EX\big[1-\e^{-w W/n^{1/\alpha}}\big]
	\\
    & =\int_1^{\infty} [1-\e^{-w s/n^{1/\alpha}}]\frac{\alpha}{s^{\alpha+1}}{\mathrm d}s =\big(\frac{w}{n^{1/\alpha}}\big)^\alpha \int_{w n^{-1/\alpha}}^{\infty} [1-\e^{-s}]\frac{\alpha}{s^{\alpha+1}}{\mathrm d}s\nn\\
    &=\frac{w^{\alpha}}{n} \int_{0}^{\infty} [1-\e^{-s}]\frac{\alpha}{s^{\alpha+1}}{\mathrm d}s
	-\frac{w^{\alpha}}{n}\int_0^{w n^{-1/\alpha}}[1-\e^{-s}]\frac{\alpha}{s^{\alpha+1}}{\mathrm d}s.\nn
\end{align}
We next use that
\begin{equation}
    \int_{0}^{\infty} [1-\e^{-s}]\frac{\alpha}{s^{\alpha+1}}{\mathrm d}s
	=\int_{0}^{\infty} \int_0^s \e^{-y}{\mathrm d}y\frac{\alpha}{s^{\alpha+1}}{\mathrm d}s  
	=\int_{0}^{\infty} \e^{-y} \int_y^\infty  \frac{\alpha}{s^{\alpha+1}}{\mathrm d}s{\mathrm d}y 
	=\int_{0}^{\infty} \e^{-y} y^{-\alpha}{\mathrm d}y=\Gamma(1-\alpha),
\end{equation}
and
\begin{equation}
    \int_0^{w n^{-1/\alpha}}[1-\e^{-s}]\frac{\alpha}{s^{\alpha+1}}{\mathrm d}s
	\leq \int_0^{w n^{-1/\alpha}}\frac{\alpha}{s^{\alpha}}{\mathrm d}s
	=\frac{\alpha}{\alpha-1} \big(w n^{-1/\alpha}\big)^{1-\alpha}
	=\frac{\alpha}{\alpha-1} w^{1-\alpha}  n^{1-1/\alpha}.
\end{equation}
This gives excellent control over $q_w$ for all $w=o(n^{1/\alpha})$.

\subsection{Binomial local limit approximations}
Let $S_m\sim \operatorname{Bin}(m,p)$. Then, a local central limit approximation shows that
\begin{equation}
    \label{Bin-LCLT-pre}
    \prob(S_m=k)=\frac{1}{\sqrt{2 \pi mp(1-p)}}{\mathrm{e}}^{-(k-mp)^2/[2mp(1-p)]}(1+o(1)),
\end{equation}
when $|k-mp|=o((mp)^{2/3})$, and uniformly in $k$ and $p$. Such approximations have a long history, we refer to the books by Bollob\'as~\cite{bollobas2001random} and Janson, {\L}uczak and Ruci\'nsky~\cite{janson2000random}. In particular, the upper bound in~\eqref{Bin-LCLT-pre} follows from~\cite[Theorem 1.2]{bollobas2001random}, the lower bound from~\cite[Theorem 1.5]{bollobas2001random}, and an analysis of the error terms present there, showing that they are $o(1)$ when $|n-mp|=o((mp)^{2/3})$.

In our case, we have that $m=n-1$ or $m=n-2$, $p=q_w$ with $w=W_1$, so that
\begin{equation}
    \prob(D_1=k\mid W_1=w)
	=\frac{1}{\sqrt{2 \pi mq_w(1-q_w)}}{\mathrm{e}}^{-(k-mq_w)^2/[2mq_w(1-q_w)]}(1+o(1)),
\end{equation}
when $|k-mq_w|=o((mq_w)^{2/3})$, and uniformly in $k$ and $w$. Since, for us, $k$ is fixed, while $w$ is being integrated out over, the above means that the main contribution for $W_1=w$ comes from $w$ for which $mq_w\approx k,$ with values of $w$ for which
\begin{equation}
    |k-mq_w|\geq C\sqrt{k}\log{k}
\end{equation}
receiving probability at most
\begin{equation}
    {\mathrm{e}}^{-C\log{k}}=k^{-C},
\end{equation}
which is negligible. Thus, we can safely think of $mq_w=k$, which means
\begin{equation}
    mq_w=w^{\alpha} \Gamma(1-\alpha)(1+o(1))=k,
\end{equation}
so that
\begin{equation}
    w=k^{1/\alpha} \Gamma(1-\alpha)^{-1/\alpha} (1+o(1)).
\end{equation}
This implies that, with $m=n-1$,
\begin{align}
    \EX[N_k] &=n\EX\big[\prob(D_1=k\mid W_1)\big]\\
	&=(1+o(1)) n\int_1^{\infty} \frac{1}{\sqrt{2 \pi mq_w(1-q_w)}}{\mathrm{e}}^{-(k-mq_w)^2/[2mq_w(1-q_w)]}\frac{\alpha}{w^{\alpha+1}}
	{\mathrm d}w\nn\\
	&= (1+o(1)) \int_1^{\infty} \frac{1}{\sqrt{2 \pi nq_w}}{\mathrm{e}}^{-(k-nq_w)^2/[2nq_w]}\frac{\alpha}{w^{\alpha+1}}
	{\mathrm d}w,\nn
\end{align}
as long as $nq_w\gg 1$, and thus, since $k\approx nq_w$, also $k\gg 1$. Since the above integral is dominated by $ nq_w$ that are close to $k$, we may replace $nq_w$ by $k$ everywhere, except in the $(k-nq_w)^2$ term. This gives
\begin{equation}
    \EX[N_k] = (1+o(1)) \frac{n}{\sqrt{2 \pi k}}\int_1^{\infty} {\mathrm{e}}^{-(k-nq_w)^2/[2k]}\frac{\alpha}{w^{\alpha+1}}
	{\mathrm d}w.
\end{equation}

We next substitute $nq_w\approx w^{\alpha} \Gamma(1-\alpha)$ to arrive at
\begin{align}
    \EX[N_k] &= (1+o(1)) \frac{n}{\sqrt{2 \pi k}}\int_1^{\infty} {\mathrm{e}}^{-(k-w^{\alpha} \Gamma(1-\alpha))^2/[2k]}\frac{\alpha}{w^{\alpha+1}}
	{\mathrm d}w \\
	&=(1+o(1)) \frac{n\alpha }{w_k^{\alpha+1}\sqrt{2 \pi k}}\int_1^{\infty} {\mathrm{e}}^{-(k-w^{\alpha} \Gamma(1-\alpha))^2/[2k]}
	{\mathrm d}w,\nn
\end{align}
where $w_k$ solves $w^{\alpha} \Gamma(1-\alpha)=k$, and thus
\begin{equation}
    w_k=k^{1/\alpha}  \Gamma(1-\alpha)^{-1/\alpha}.
\end{equation}

We then let $z=w^{\alpha} \Gamma(1-\alpha)$, so that ${\mathrm d}z=\alpha w^{\alpha-1} \Gamma(1-\alpha) {\mathrm d}w$, so  that

\begin{equation}
    {\mathrm d}w=\Gamma(1-\alpha)^{-1} (z/\Gamma(1-\alpha))^{(1-\alpha)/\alpha}{\mathrm d}z=z^{(1-\alpha)/\alpha}\Gamma(1-\alpha)^{-1/\alpha}{\mathrm d}z,
\end{equation}

so that

\begin{align}
\label{expec-Nk-fin}
    \EX[N_k] &= (1+o(1))
	\frac{n\alpha }{w_k^{\alpha+1}\sqrt{2 \pi k}}\int_1^{\infty} 
	{\mathrm{e}}^{-(k-z)^2/[2k]} z^{(1-\alpha)/\alpha}\Gamma(1-\alpha)^{-1/\alpha}
	{\mathrm d}z\\
	&=(1+o(1))
	\frac{n\alpha }{w_k^{\alpha+1}\sqrt{2 \pi k}}\Gamma(1-\alpha)^{-1/\alpha}
	k^{(1-\alpha)/\alpha}\int_1^{\infty} {\mathrm{e}}^{-(k-z)^2/[2k]} {\mathrm d}z\nn\\
	&=(1+o(1))
	\frac{n\alpha}{w_k^{\alpha+1}\sqrt{2\pi k}} k^{(1-\alpha)/\alpha}\Gamma(1-\alpha)^{-1/\alpha} \sqrt{2\pi k} \nn\\
	&=n(1+o(1))\alpha k^{-(\alpha+1)/\alpha} \Gamma(1-\alpha)^{(\alpha+1)/\alpha} k^{(1-\alpha)/\alpha}\Gamma(1-\alpha)^{-1/\alpha} \nn\\
    & = n(1+o(1))\frac{\alpha \Gamma(1-\alpha)}{k^2}.\nn
\end{align}

This confirms that $\EX[N_k]=\Theta(1)$ when $k=\Theta(\sqrt{n})$, so that one cannot expect that $N_k/\EX[N_k]\overset{p}{\to} 1$ for such values of $k$. 

Since~\eqref{Bin-LCLT-pre} depends on $k$, for $k$ large, in a relatively weak way, this means that, for values of $w$ close to $w_k$,
\begin{equation}
    \prob(D_1'=k-2\mid W_1=w)\approx \prob(D_1=k\mid W_1=w).
\end{equation}
Thus, since such factors appear both in $\EX[N_k]$ and $\EX\Big[\sum_{v\in [n]}\Delta_v\mathbbm{1}_{D_v=k}\Big]$, they cancel out asymptotically, so that
\begin{equation}
    \frac{\EX\Big[\sum_{v\in [n]}\Delta_v\mathbbm{1}_{D_v=k}\Big]}{\EX[N_k]}
	\approx 
    n{{n-1}\choose{2}}\prob(I_{1,2}I_{1,3}I_{2,3}=1\mid W_1=w_k).
\end{equation}
Recalling~\eqref{eq:CC_annealed_app}, this then shows that
\begin{equation}
    \bar{C}_n(k)=(1+o(1)) \frac{{{n-1}\choose{2}}}{{{k}\choose{2}}} \prob(I_{1,2}I_{1,3}I_{2,3}=1\mid W_1=w_k).
\end{equation}

We now take $k=a\sqrt{n}$, and obtain 
\begin{equation}
    \bar{C}_n(a\sqrt{n})=(1+o(1))\frac{n\alpha^2}{a^2} \int_1^{\infty} \int_1^{\infty}[1-\e^{-xw_k/n^{1/\alpha}}][1-\e^{-yw_k/n^{1/\alpha}}][1-\e^{-xy/n^{1/\alpha}}](xy)^{-(\alpha+1)}
	{\mathrm d}x {\mathrm d}y.
\end{equation}

We first note that, for $k=a\sqrt{n}$,
\begin{equation}
    w_{a\sqrt{n}}=n^{1/(2\alpha)} a^{1/\alpha} \Gamma(1-\alpha)^{-1/\alpha}
	\equiv n^{1/(2\alpha)} a^{1/\alpha} {\tau_\alpha}.
\end{equation}
We rescale $x$ and $y$ by $n^{-1/2\alpha}$ to obtain
\begin{align}
    \bar{C}_n(a\sqrt{n}) & =(1+o(1))\frac{\alpha^2}{a^2} \int_{n^{-1/2\alpha}}^{\infty} \int_{n^{-1/2\alpha}}^{\infty}[1-\e^{-xa^{1/\alpha} {\tau_\alpha}}][1-\e^{-ya^{1/\alpha} {\tau_\alpha}}][1-\e^{-xy}](xy)^{-(\alpha+1)}
	{\mathrm d}x {\mathrm d}y \\
	& =(1+o(1))\frac{\alpha^2}{a^2} \int_{0}^{\infty} \int_{0}^{\infty}[1-\e^{-xa^{1/\alpha} {\tau_\alpha}}][1-\e^{-ya^{1/\alpha} {\tau_\alpha}}][1-\e^{-xy}](xy)^{-(\alpha+1)}
	{\mathrm d}x {\mathrm d}y,\nn
\end{align}
as predicted in~\eqref{eq:analytical_clustering_app}.
\end{proof}

\section{Asymptotics of the Clustering function}
\label{app:asymptotics}

\begin{lemma}[$\bar{C}(a)$ is bounded by 1]
\label{lem-annealed-bd-1}
For all $a>0$, $\bar{C}(a)\leq 1$.
\end{lemma}

\begin{proof}
    Estimating $1-\e^{-xy}\leq 1$ bounds $\bar{C}(a)$ by
    \begin{equation}
        \bar{C}(a) \leq \frac{\alpha^2}{a^2} \Big(\int_{0}^{\infty} [1-\e^{-xa^{1/\alpha} {\tau_\alpha}}]x^{-(\alpha+1)}
	{\mathrm d}x\Big)^2,
    \end{equation}
for which we note that
\begin{align}
    \label{single-integral}
	\int_{0}^{\infty} [1-\e^{-xa^{1/\alpha} {\tau_\alpha}}]x^{-(\alpha+1)}
	{\mathrm d}x
	&=\int_{0}^{\infty} \int_0^{xa^{1/\alpha} {\tau_\alpha}} \e^{-y} {\mathrm d}y x^{-(\alpha+1)}
	{\mathrm d}x \\
	&=\int_{0}^{\infty} \e^{-y} \int_{ya^{-1/\alpha}/{\tau_\alpha}}^{\infty} x^{-(\alpha+1)}{\mathrm d}x
	{\mathrm d}y
	=\frac{1}{\alpha} \int_{0}^{\infty} \e^{-y} (ya^{-1/\alpha}/{\tau_\alpha})^{-\alpha}{\mathrm d}y\nn\\
    &=\frac{a\tau_\alpha^{\alpha}}{\alpha} \int_{0}^{\infty} \e^{-y} y^{-\alpha}{\mathrm d}y
	=\frac{a\tau_\alpha^{\alpha}}{\alpha}\Gamma(1-\alpha)=\frac{a}{\alpha},\nn
\end{align}
since ${\tau_\alpha}=\Gamma(1-\alpha)^{-1/\alpha}$. Thus, indeed, $\bar{C}(a)\leq 1$.
\end{proof}

We next investigate $\bar{C}(a)$ for $a\ll 1$:
\begin{lemma}[$\bar{C}(a)$ is close to 1 for $a$ small]
\label{lem-annealed-small-a}
As $a\searrow 0$,
	\begin{equation}
	    \bar{C}(a)=1+o(1).
	\end{equation}
\end{lemma} 

\begin{proof}
    For $a\ll 1$, we write
\begin{align}
    \bar{C}(a)&=\frac{\alpha^2}{a^2} \int_{0}^{\infty} \int_{0}^{\infty}[1-\e^{-xa^{1/\alpha} {\tau_\alpha}}][1-\e^{-ya^{1/\alpha} {\tau_\alpha}}](xy)^{-(\alpha+1)}
	{\mathrm d}x {\mathrm d}y\\
	&\qquad-\frac{\alpha^2}{a^2} \int_{0}^{\infty} \int_{0}^{\infty}[1-\e^{-xa^{1/\alpha} {\tau_\alpha}}][1-\e^{-ya^{1/\alpha} {\tau_\alpha}}]\e^{-xy}(xy)^{-(\alpha+1)}
	{\mathrm d}x {\mathrm d}y.\nn
\end{align}

The first integral equals $1$ by Lemma~\ref{lem-annealed-bd-1}, so that
\begin{equation}
    \bar{C}(a) =1-\bar{C}_1(a),
\end{equation}
where
\begin{equation}
    \bar{C}_1(a)=\frac{\alpha^2}{a^2} \int_{0}^{\infty} \int_{0}^{\infty}[1-\e^{-xa^{1/\alpha} {\tau_\alpha}}][1-\e^{-ya^{1/\alpha} {\tau_\alpha}}]\e^{-xy}(xy)^{-(\alpha+1)}
	{\mathrm d}x {\mathrm d}y.
\end{equation}

We first bound the integrals where either $x$ or $y$ is in $[0,a^{-\varepsilon}]$, for some $\varepsilon>0$ sufficiently small, as 
\begin{align}
    &\frac{2\alpha^2}{a^2}\int_{0}^{\infty} \int_{0}^{a^{-\varepsilon}}[1-\e^{-xa^{1/\alpha} {\tau_\alpha}}][1-\e^{-ya^{1/\alpha} {\tau_\alpha}}]\e^{-xy}(xy)^{-(\alpha+1)}
	{\mathrm d}x {\mathrm d}y\\
    &\qquad \leq \frac{2\alpha^2}{a^2}\int_{0}^{\infty} [1-\e^{-xa^{1/\alpha} {\tau_\alpha}}] x^{-(\alpha+1)}{\mathrm d}x
	a^{1/\alpha} \int_{0}^{a^{-\varepsilon}} y^{-\alpha}{\mathrm d}x\nn\\
	&\qquad\leq \frac{2\alpha}{a(\alpha-1)} a^{1/\alpha-\varepsilon(1-\alpha)}=o(1),\nn
\end{align}
using~\eqref{single-integral} and the fact that $1/\alpha-\varepsilon(1-\alpha)>1$ for $\varepsilon>0$ sufficiently small since $\alpha\in (0,1)$. Thus,
\begin{equation}
    \bar{C}_1(a)=o(1)+\frac{\alpha^2}{a^2} \int_{a^{-\varepsilon}}^{\infty} \int_{a^{-\varepsilon}}^{\infty}[1-\e^{-xa^{1/\alpha} {\tau_\alpha}}][1-\e^{-ya^{1/\alpha} {\tau_\alpha}}] \e^{-xy}(xy)^{-(\alpha+1)}
	{\mathrm d}x {\mathrm d}y,
\end{equation}
which is bounded by $\e^{-a^{-2\varepsilon}}=o(1)$. This shows that $\bar{C}_1(a)=o(1)$ for $a\searrow 0$, and thus proves that $\bar{C}(a)=1+o(1)$.
\end{proof}

For $a\rightarrow \infty$, the integral should be dominated by small $x,y$, so that one might consider to expand $1-\e^{-xy}=xy$, to obtain
\begin{align}
    \bar{C}(a)&=(1+o(1))
	\frac{\alpha^2}{a^2} \int_{0}^{\infty} \int_{0}^{\infty}[1-\e^{-xa^{1/\alpha} {\tau_\alpha}}][1-\e^{-ya^{1/\alpha} {\tau_\alpha}}](xy)^{-\alpha}
	{\mathrm d}x {\mathrm d}y\\
	&=(1+o(1))
	\frac{\alpha^2}{a^2} \big(a^{1/\alpha} {\tau_\alpha}\big)^{2\alpha-2}\Big(\int_{0}^{\infty}[1-\e^{-x}]x^{-\alpha}
	{\mathrm d}x\Big)^2.\nn
\end{align}

However, the integral that is squared is infinite, so that we need to proceed alternatively. In fact, this analysis is quite intricate, and a careful split is necessary. The main result is as follows:
\begin{lemma}[Large $a$ asymptotics of $\bar{C}(a)$]
\label{lem-annealed-large-a}
As $a\nearrow \infty$,
	\begin{equation}
	    \bar{C}(a)=2 \Gamma(1-\alpha)\frac{\log{a}}{a^2} +O(a^{-2}).
	\end{equation}
\end{lemma} 

\begin{proof}
    We again split the integral conveniently. We split the integral depending on whether $x$ or $y$ is in $[0,a^{-1/\alpha}/{\tau_\alpha}]$ or not. This gives
    \begin{equation}
        \bar{C}(a)=\bar{C}_1(a)+\bar{C}_2(a)+\bar{C}_3(a),
    \end{equation}
where
\begin{equation}
    \bar{C}_1(a)=\frac{\alpha^2}{a^2} \int_{a^{-1/\alpha}/{\tau_\alpha}}^{\infty} \int_{a^{-1/\alpha}/{\tau_\alpha}}^{\infty}[1-\e^{-xa^{1/\alpha} {\tau_\alpha}}][1-\e^{-ya^{1/\alpha} {\tau_\alpha}}][1-\e^{-xy}](xy)^{-(\alpha+1)}
	{\mathrm d}x {\mathrm d}y,
\end{equation}
\begin{equation}
	\bar{C}_2(a)=\frac{\alpha^2}{a^2} \int_0^{a^{-1/\alpha}/{\tau_\alpha}}\int_0^{a^{-1/\alpha}/{\tau_\alpha}}[1-\e^{-xa^{1/\alpha} {\tau_\alpha}}][1-\e^{-ya^{1/\alpha} {\tau_\alpha}}][1-\e^{-xy}](xy)^{-(\alpha+1)}
	{\mathrm d}x {\mathrm d}y.
\end{equation}
We bound, again using~\eqref{single-integral} and $1-\e^{-xy}\leq xy$,
\begin{align}
    \bar{C}_2(a)& \leq \frac{\alpha^2}{a^2}\int_{0}^{a^{-1/\alpha}/{\tau_\alpha}} \int_{0}^{a^{-1/\alpha}/{\tau_\alpha}} [1-\e^{-xa^{1/\alpha} {\tau_\alpha}}][1-\e^{-ya^{1/\alpha} {\tau_\alpha}}](xy)^{-\alpha}
	{\mathrm d}x {\mathrm d}y\\
	&=\frac{\alpha^2}{a^2} \Big(\int_{0}^{a^{-1/\alpha}/{\tau_\alpha}} y^{-\alpha}
	{\mathrm d}y\Big)^2\nn\\
    &=\Theta(1)a^{-2}(a^{-1/\alpha})^{2(1-\alpha)}=o(a^{-2}).\nn
\end{align}
\smallskip

We continue with $\bar{C}_3(a)$, which is the most involved. We bound $1-\e^{-ya^{1/\alpha} {\tau_\alpha}}\leq ya^{1/\alpha}$ to obtain
\begin{equation}
    \bar{C}_3(a)\leq \frac{2\alpha^2 a^{1/\alpha}}{a^2}\int_{0}^{a^{-1/\alpha}/{\tau_\alpha}} \int_{a^{-1/\alpha}/{\tau_\alpha}}^\infty
	(xy)^{-(\alpha+1)} y [1-\e^{-xy}]
	{\mathrm d}x {\mathrm d}y.
\end{equation}
We bound $1-\e^{-xy}\leq (xy)\wedge 1$. The integral where $x\leq 1/y$ is bounded by
\begin{equation}
    \Theta(1) a^{-2+1/\alpha}
	\int_{0}^{a^{-1/\alpha}/{\tau_\alpha}} y^{-\alpha+1} \int_{0}^{1/y}
	x^{-\alpha}
	{\mathrm d}x {\mathrm d}y \leq \Theta(1)a^{-2+1/\alpha} \int_{0}^{a^{-1/\alpha}/{\tau_\alpha}} {\mathrm d}y\leq \Theta(1)a^{-2}.
\end{equation}
Similarly, the integral where $x>1/y$ is bounded by
\begin{equation}
    \Theta(1) a^{-2+1/\alpha}
	\int_{0}^{a^{-1/\alpha}/{\tau_\alpha}} y^{-\alpha} \int_{1/y}^{\infty}
	x^{-(\alpha+1)}
	{\mathrm d}x {\mathrm d}y\leq \Theta(1)a^{-2+1/\alpha} \int_{0}^{a^{-1/\alpha}/{\tau_\alpha}}{\mathrm d}y\leq \Theta(1)a^{-2}.
\end{equation}
This shows that $\bar{C}_3(a)=O(a^{-2})$.
\smallskip

We finally compute $\bar{C}_1(a)$, which we rewrite as
\begin{align}
    \label{c1-bar-bd}
	\bar{C}_1(a) &=\alpha^2 \tau_\alpha^{2\alpha} 
	\int_{1}^{\infty} \int_{1}^{\infty}[1-\e^{-x}][1-\e^{-y}][1-\e^{-xya^{-2/\alpha}/\tau_\alpha^2}](xy)^{-(\alpha+1)}
	{\mathrm d}x {\mathrm d}y \\
	&=\tau_\alpha^{2\alpha}  \EX[1-\e^{-W_1W_2 a^{-2/\alpha}/\tau_\alpha^2}]+\bar{e}_1(a),\nn
\end{align}
where we let $\bar{e}_1(a)$ denote the contribution due to $\e^{-x}$ and $\e^{-y}$. We use~\cite[(4.6)]{avena2022inhomogeneous}, which states that 
\begin{equation}
    \EX[1-\e^{-\delta W_1 W_2}] =(1+o(1)) \alpha \Gamma(1-\alpha) \delta^{\alpha} \log(1/\delta).
\end{equation}

This gives

\begin{equation}
    \bar{C}_1(a)=\tau_\alpha^{2\alpha} \alpha  \tau_\alpha^{-2\alpha} \frac{2}{\alpha}\Gamma(1-\alpha) \frac{\log{a}}{a^2}+\bar{e}_1(a)
	=2\Gamma(1-\alpha)\frac{\log{a}}{a^2}+\bar{e}_1(a).
\end{equation}
Finally,
\begin{align}
    \bar{e}_1(a) & \leq 2\alpha^2 \tau_\alpha^{2\alpha} 
	\int_{1}^{\infty} \int_{1}^{\infty}\e^{-x}[1-\e^{-xya^{-2/\alpha}/\tau_\alpha^2}](xy)^{-(\alpha+1)}
	{\mathrm d}x {\mathrm d}y \\
	&=2\tau_\alpha^{2\alpha} \EX\big[\e^{-W_1}(1-\e^{-W_1W_2 a^{-2/\alpha}/\tau_\alpha^2})\big]=O(a^{-2}).\nn
\end{align}
This can be seen by letting $X$ be defined by
\begin{equation}
    \prob(X\in [a,b])=\frac{1}{\EX\big[\e^{-W_1}]} \EX\big[\e^{-W_1} \mathbbm{1}_{\{X\in [a,b]\}}\big],
\end{equation}
so that
\begin{equation}
    \bar{e}_1(a)=\Theta(1)  \EX\big[1-\e^{-XW a^{-2/\alpha}/\tau_\alpha^2}\big].
\end{equation}

Since $X$ has thin tails (it has an exponential moment), the random variable $XW$ is regularly varying with exponent $\alpha$, so that
\begin{equation}
    \EX\big[1-\e^{-XW a^{-2/\alpha}/\tau_\alpha^2}\big]=\Theta(a^{-2}),
\end{equation}
as claimed. The above follows from the fact that $XW$ is again heavy-tailed with tail exponent $\alpha$, as proved in~\cite[Proposition 5.2]{LeiSauKon23}.
\end{proof}

\section{Computation of Average Clustering Coefficient}

The limiting asymptotics for hubs and leaves allows us to derive an estimate of the average clustering coefficient of the network $\bar{C}$. Indeed, one needs to integrate the annealed clustering function $\bar{C}(a)$ over the distribution $P(a)$ of the reduced degree $a$, and from Figure~\ref{fig:clusterings} it is clear that one can split the integral into the two leaf ($L$) and hub ($H$) regimes as

\begin{align}
    \bar{C} =  \int_{a_{min}}^{\infty} {\mathrm d}a \bar{C}(a) P(a)
    = \int_{a_{min}}^{a^*} {\mathrm d}a \bar{C}_L(a) P(a) + \int_{a^*}^{\infty} {\mathrm d}a \bar{C}_H(a) P(a),
    \label{eq:avg_cc}
\end{align}
where $C_{H}(a) \sim 2 \Gamma(1-\alpha) \frac{\log a}{a^2}$ (for $a>a^*$) is the asymptotic clustering function of hubs, while $C_{L}(a)\sim 1$ (for $a_{min}<a<a^*$) is the asymptotic plateau for leaves, $a_{min}$ is the minimal reduced degree above which the clustering function starts to have meaning, and $a^*$ is an appropriate crossover value between the two regimes. This crossover is visible in Figure~\ref{fig:clusterings}, although its value is not easily identified. Luckily, as we will show in the following computations, the actual value of $a^*$ is not necessary for the computation of the clustering to go through, and we can proceed without specifying it.

The main issue in the evaluation of the integral in~\eqref{eq:avg_cc} is that the degree distribution is not known everywhere but only in the tail~\cite{avena2022inhomogeneous},  where it has the asymptotic behaviour $\Gamma(1-\alpha) k^{-2}$. Transforming to the distribution $P(a)$ of the reduced degree $a = \frac{k}{\sqrt{n}}$,
\begin{equation}
    P(a) \sim \frac{\Gamma(1-\alpha)}{\sqrt{n}} a^{-2}.
\end{equation}

Nonetheless, the first integral can be evaluated by realizing that by definition it evaluates to the fraction of leaves nodes or, equivalently

\begin{equation}
    \int_{a_{min}}^{a^*} {\mathrm d}a \bar{C}_L(a) P(a) = \int_{a_{min}}^{a^*} {\mathrm d}a  P(a) = \begin{cases}
        1 - \int_{a^*}^{\infty} da P(a) \quad &\text{if } C_{k<2} \text{ undefined,} \\
         1 - r_{0/1} - \int_{a^*}^{\infty} da P(a) \quad &\text{if } C_{k<2} = 0,
    \end{cases} 
\end{equation}
and the last term can be recomprised in the second integral as it is evaluated on the same interval, so that the average clustering coefficient has the expression
\begin{equation}
    \bar{C} = \begin{cases}
        1 + \int_{a^*}^{\infty} P(a) (\bar{C}_H(a) - 1)& \text{ if } C_{k<2} \text{ undefined,} \\
        1 - r_{0/1} +\int_{a^*}^{\infty} da P(a) (\bar{C}_H(a) - 1) & \text{ if } C_{k<2} = 0,
    \end{cases}
\end{equation}

We can evaluate this last integral as
\begin{eqnarray}
    \int_{a^*}^{\infty} da P(a) (\bar{C}_H(a) - 1) &=& \int_{a^*}^{\infty} da \left[ 2 \Gamma(1-\alpha) \frac{\log a}{a^2} - 1\right]\frac{\Gamma(1-\alpha)}{a^2 \sqrt{n}}\\
    &=& \frac{\Gamma(1-\alpha)}{\sqrt{n}} \frac{6 \Gamma(1-\alpha) \log a^* + 2 \Gamma(1-\alpha) - 9 (a^*)^2}{ 9 (a^*)^3}\nn\\
    &=& o(1),\nn
\end{eqnarray}
and see that it is suppressed by a square root of $n$ term, whatever the value of $a^*$. 

\section{Computation of $r_{0/1}$}
\label{app:rzerouno}
We start by computing the fraction of nodes that are disconnected from the graph as
\begin{eqnarray}
    r_{0} \equiv \frac{1}{n}\sum_{i=1}^n \prod_{j\colon j\neq i} (1-p_{ij})
    &=& \frac{1}{n}\sum_{i=1}^n\e^{-\delta_n w_i \sum_{j\neq i}w_j}\label{r0-computation-rep},
\end{eqnarray}
where we recall that  $\delta_n=n^{-1/\alpha}$. Since the $j=i$ term in the sum in the exponential makes little difference, we can approximate
\begin{eqnarray}
    r_{0}
    &\approx& \frac{1}{n}\sum_{i=1}^n\e^{-\delta_n w_i \sum_{j}w_j}\label{r0-computation-rep2}.
\end{eqnarray}
We next use that most of the weights have normal size, and they contribute most to the sum over $i\in\{1, \ldots, n\}$. Thus, 
\begin{eqnarray}
    r_{0}
    &\approx& \EX\Big[\e^{- W S_n}\mid S_n\Big]\label{r0-computation-rep3},
\end{eqnarray}
where we take the expectation over the Pareto variable $W$ in the above, we abbreviate
    \begin{equation}
    \label{rescaled-total-weight}
        S_n=\delta_n\sum_{j=1}^nw_j,
    \end{equation}
and condition on this total rescaled weight. We note that the random variable $S_n$ converges in distribution to an $\alpha$-stable distribution, since the random variables $(w_i)_{i=1}^n$ are Pareto variables with infinite mean. Thus,
\begin{eqnarray}
    r_{0}
    &\approx& \EX\Big[\e^{- W \mathcal{S}}\mid \mathcal{S}\Big]\label{r0-computation-rep4},
\end{eqnarray}
where $\mathcal{S}$ is an $\alpha$-stable random variable. 
Equation~\eqref{r0-computation-rep4} requires some interpretation. What it means is that the proportion of vertices with $r_0$ degree 0, given the vertex weights $(w_i)_{i=1}^n$, is close to the rhs of~\eqref{r0-computation-rep3}, i.e., it holds that
$r_0\approx \EX\Big[\e^{- W S_n}\mid S_n\Big]$, where the rescaled total weight $S_n$ is defined in~\eqref{rescaled-total-weight}.
In particular, this means that, when simulating the network several times (and thus redrawing the vertex weights), the proportion of vertices with degree 0 {\em fluctuates macroscopically}, even in the large graph limit. This is a type of {\em non-self-averaging} phenomenon that is rather unusual in network science. The simulations shown in Figures~\ref{fig:pan_03}~\ref{fig:pan_05}~\ref{fig:pan_07} confirm that~\eqref{r0-computation-rep4} is an excellent approximation of $r_0$, and that it remains truly random in the large-graph limit.
We can also use~\eqref{r0-computation-rep4} us to relate our approximation to~\cite[Proposition 4]{avena2022inhomogeneous}, in which $\EX[r_0]$ is shown to be approximated by
    \begin{equation}
    \EX[r_{0}]
    \approx \EX\Big[\e^{- \Gamma(1-\alpha)W^\alpha} \Big],
    \end{equation}
which indeed equals $\EX\big[\e^{- W \mathcal{S}}\big]$, since, for every $a\geq 0$,
    \begin{equation}
    \EX\Big[\e^{- a \mathcal{S}}\Big]
    =\e^{- \Gamma(1-\alpha)a^\alpha}.
    \end{equation}
This follows from~\cite[Theorem 5]{avena2022inhomogeneous}. For a practical computation of the approximation of $r_0$, we use the fact, again for $a\geq 0$,
    \begin{equation}
    \label{MGF-Pareto}
    \EX\big[\e^{- a W}\big]
    =\alpha a^{\alpha} \Gamma(-\alpha, a),
    \end{equation}
where we recall that $\Gamma(r, a)$ is the incomplete gamma function. Applying this to~\eqref{r0-computation-rep3}, we are led to
\begin{equation}
\label{r0-approximation}
    r_{0}
    \approx \alpha S_n^\alpha \Gamma(-\alpha, S_n).
\end{equation}

We show in Figures~\ref{fig:pan_03}~\ref{fig:pan_05}~\ref{fig:pan_07} that~\eqref{r0-approximation} is in very good accordance with the actual fraction of disconnected nodes (central panels).

We extend the above analysis to the proportion $r_1$ of vertices of degree 1. We start from
\begin{eqnarray}
\label{r1-first-step}
    r_{1} &\equiv& \frac{1}{n}\sum_{i,k\colon i\neq k} p_{ik} \prod_{j\colon j\neq i,k} (1-p_{ij})\\
    &=& \frac{1}{n}\sum_{i,k\colon i\neq k}
    (1-\e^{-\delta_n w_iw_k})
    \e^{-\delta_n w_i \sum_{j\colon j\neq i, k} w_j}\nn\\
    &\approx & \sum_{k}
    \EX\Big[(1-\e^{- W \delta_n w_k})
    \e^{-W \delta_n\sum_{j\colon j\neq k} w_j}\Big],\nn
\end{eqnarray}
where the asymptotics follows from the fact that the main contribution in the sum over $i$ arises from normal weight vertices, so that the average over $i$ can be replaced by an expectation with respect to $W$.
\smallskip

Note that the dominant contributions to the sum over $k$ in~\eqref{r1-first-step} correspond to those $k$ for which $w_k$ is of the order $1/\delta_n=n^{1/\alpha},$ which means that these are the maximal-weight vertices. This will be the guiding principle to derive the asymptotics of the above formula.
\smallskip

For simplicity, and without loss of generality, we will assume that the weights are {\em ordered} from large to small, so that $w_1\geq w_2\cdots\geq w_n$. We can rewrite
    \begin{equation}
        \delta_n\sum_{j\colon j\neq k} w_j
        =S_n-\delta_nw_k
        =S_n-y_k^{\sss(n)},
    \end{equation}
where we recall $S_n$ from~\eqref{rescaled-total-weight}, and we define the rescaled vertex weights $(y_k^{\sss(n)})_{k\geq 1}$ as
    \begin{equation}
    \label{ordered-resceled-weights}
        y_k^{\sss(n)}=\delta_nw_k.
    \end{equation}
In terms of these definitions, we obtain the finite-size approximation
    \begin{equation}
    \label{r1-approximation}
     r_1   \approx  \sum_{k}
    \EX\Big[(1-\e^{- W y_k^{\sss(n)}})
    \e^{-W (S_n- y_k^{\sss(n)})}\mid (y_k^{\sss(n)})_{k\geq 1}\Big].
    \end{equation}

This approximation is inspired by the asymptotics properties of $r_1$. Indeed, as $n\rightarrow \infty,$ we have the convergence in distribution 
    \begin{equation}
    \label{conv-sum-Pareto-order-statistics}
    \Big(S_n, (y_k^{\sss(n)})_{k\geq 1}\Big)
    \convd \Big(\sum_{j\geq 1} \Gamma_j^{-1/\alpha}, 
    (\Gamma_k^{-1/\alpha})_{k\geq 1}\Big),
    \end{equation}
where $\convd$ denotes convergence in distribution in the product topology, $(\Gamma_k)_{k\geq 1}$ are gamma random variables, i.e.,
    \begin{equation}
    \Gamma_k=\sum_{i=1}^k E_i,
    \end{equation}
where $(E_i)_{i\geq 1}$ are independent and identically distributed exponential random variables with mean 1. Because of this, we see that the random variables in~\eqref{r1-approximation} converge in distribution in the large graph limit, so that
    \begin{equation}
    r_1\convd 
    \sum_{k}
    \EX\Big[(1-\e^{- W y_k})
    \e^{-W (\mathcal{S}- y_k)}\mid (y_k)_{k\geq 1}\Big],
    \end{equation}
where
    \begin{equation}
     \Big(\mathcal{S}, (y_k)_{k\geq 1}\Big)
     \equiv \Big(\sum_{j\geq 1} \Gamma_j^{-1/\alpha}, 
    (\Gamma_k^{-1/\alpha})_{k\geq 1}\Big).
    \end{equation}
Thus, as for $r_0$, also $r_1$ converges in {\em distribution}, but it remains to fluctuate even in the large graph limit. 
\smallskip

For a practical computation of the approximation of $r_0$, we again use~\eqref{MGF-Pareto} on the representation
    \begin{equation}
    \label{r1-approximation-rep}
     r_1   \approx  \sum_{k=1}^n
    \EX\Big[\e^{-W (S_n-y_k^{\sss(n)})}
    -\e^{-WS_n}\mid (y_k^{\sss(n)})_{k\geq 1}\Big],
    \end{equation}
to arrive at
    \begin{equation}
    \label{r1-approximation-rep2}
     r_1   \approx  \alpha \sum_{k=1}^n
    \Big[(S_n-y_k^{\sss(n)})^\alpha \Gamma(-\alpha,S_n-y_k^{\sss(n)})
    -S_n^\alpha \Gamma(-\alpha, S_n)\Big],
    \end{equation}
which has the advantage that it no longer contains the expectation with respect to $W$. We show in Figures~\ref{fig:pan_03}~\ref{fig:pan_05}~\ref{fig:pan_07} that~\eqref{r1-approximation-rep2} is in very good accordance with the actual fraction of disconnected nodes (bottom panels).

\begin{figure}[h!]
    \centering
    \includegraphics[width=0.49\linewidth, height= 0.32\linewidth]{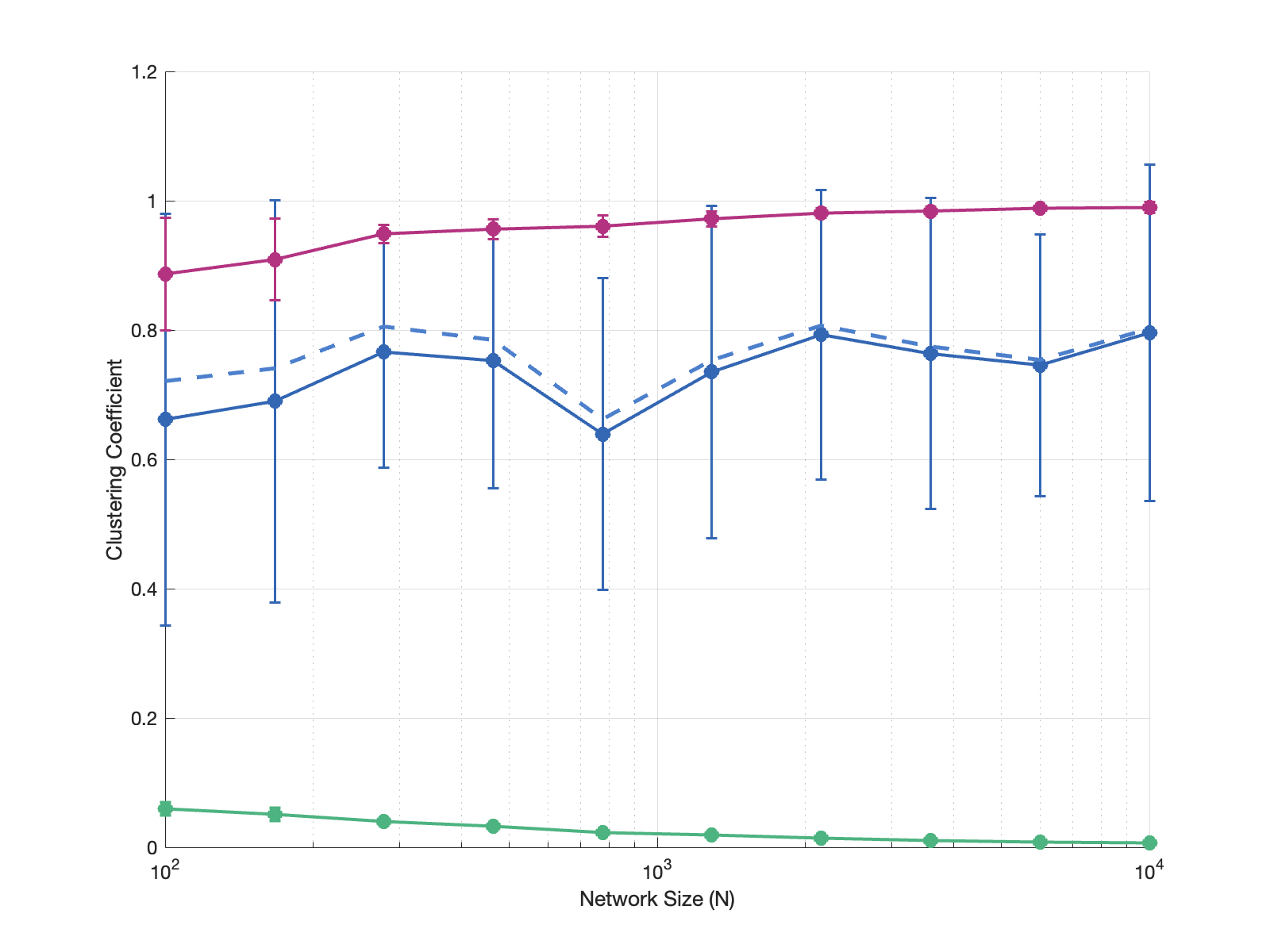}
    \includegraphics[width=0.49\linewidth, height= 0.32\linewidth]{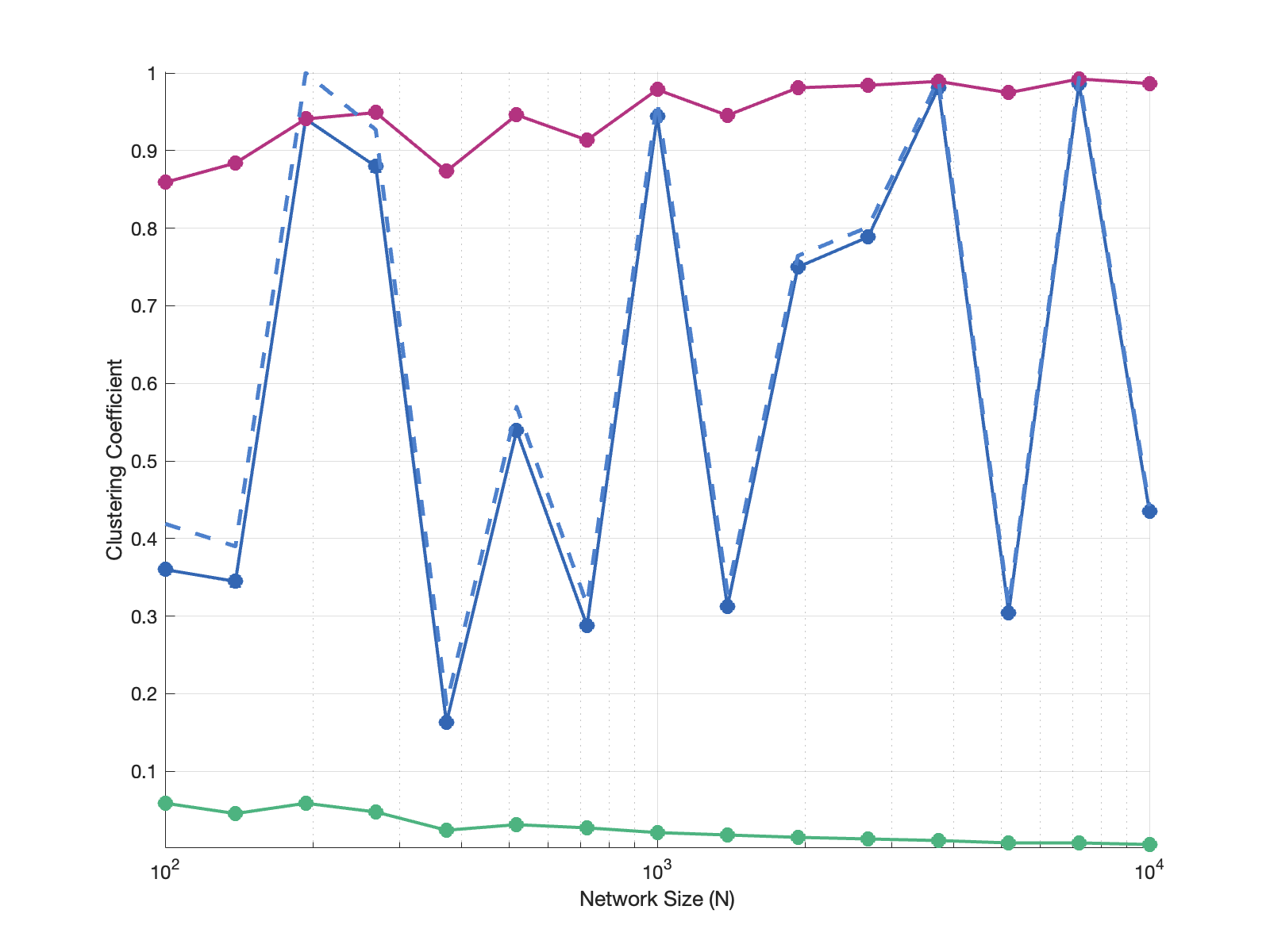}\\
    \includegraphics[width=0.49\linewidth, height= 0.32\linewidth]{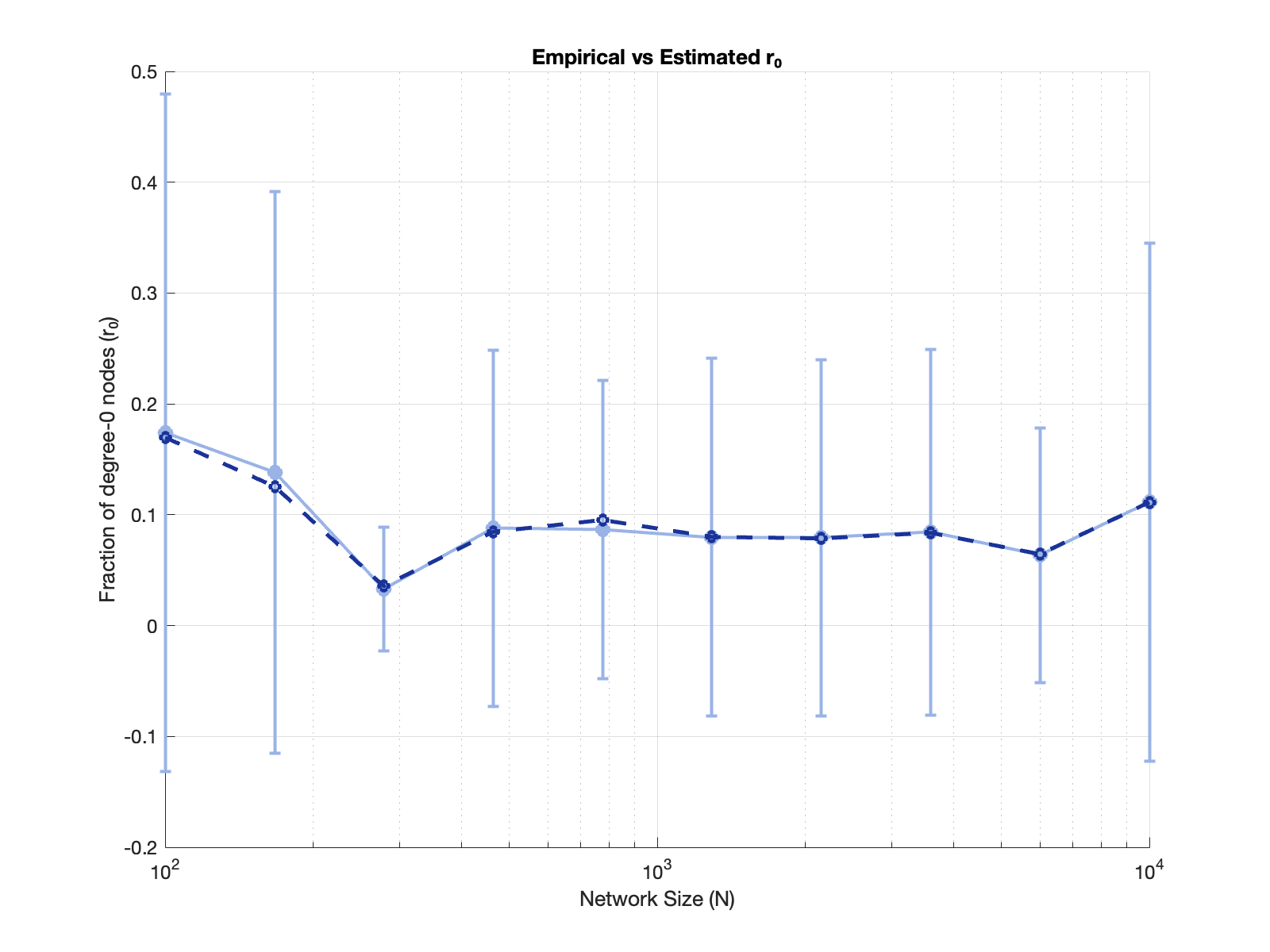}
    \includegraphics[width=0.49\linewidth, height= 0.32\linewidth]{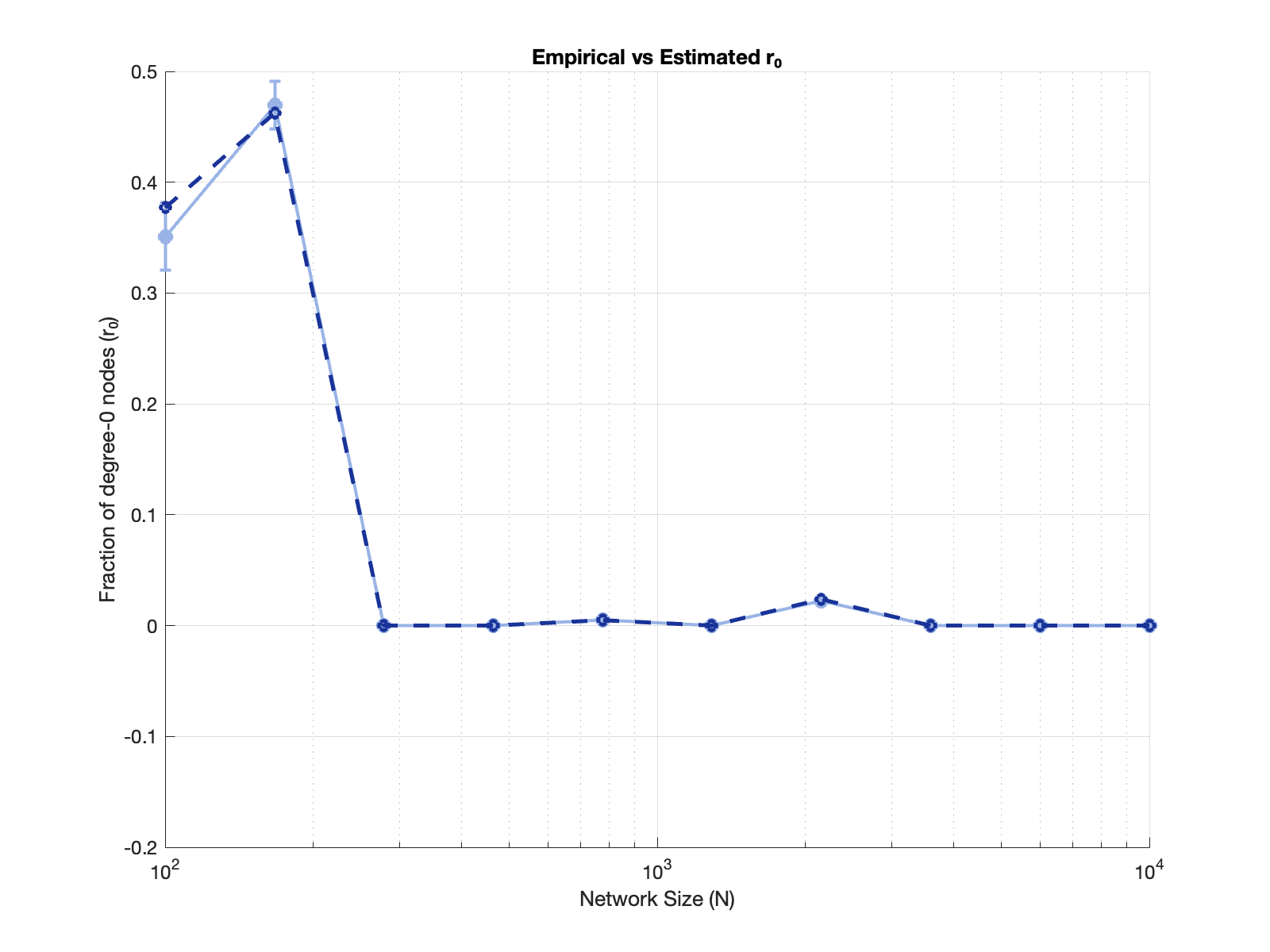}\\
    \includegraphics[width=0.49\linewidth, height= 0.32\linewidth]{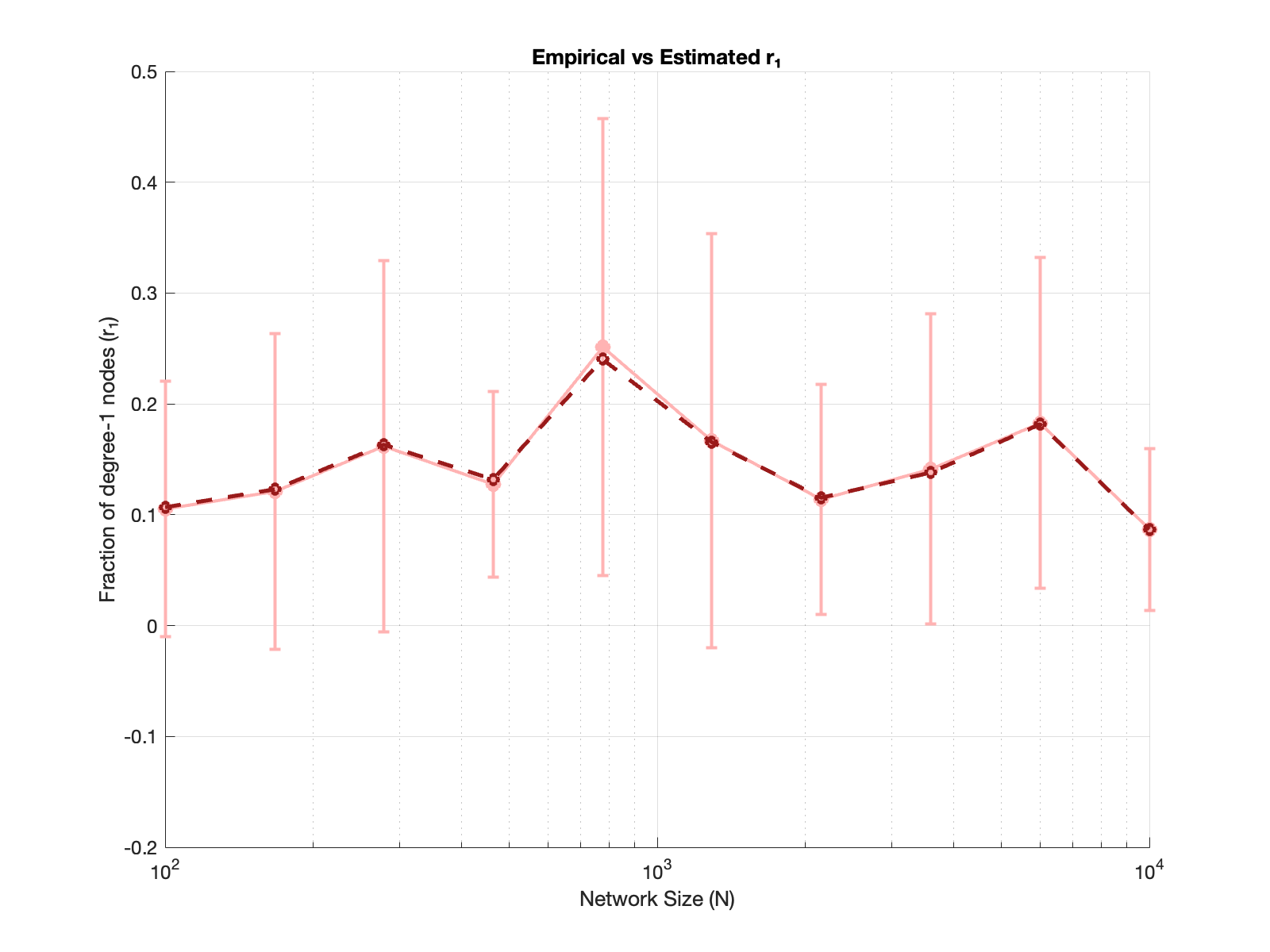}
    \includegraphics[width=0.49\linewidth, height= 0.32\linewidth]{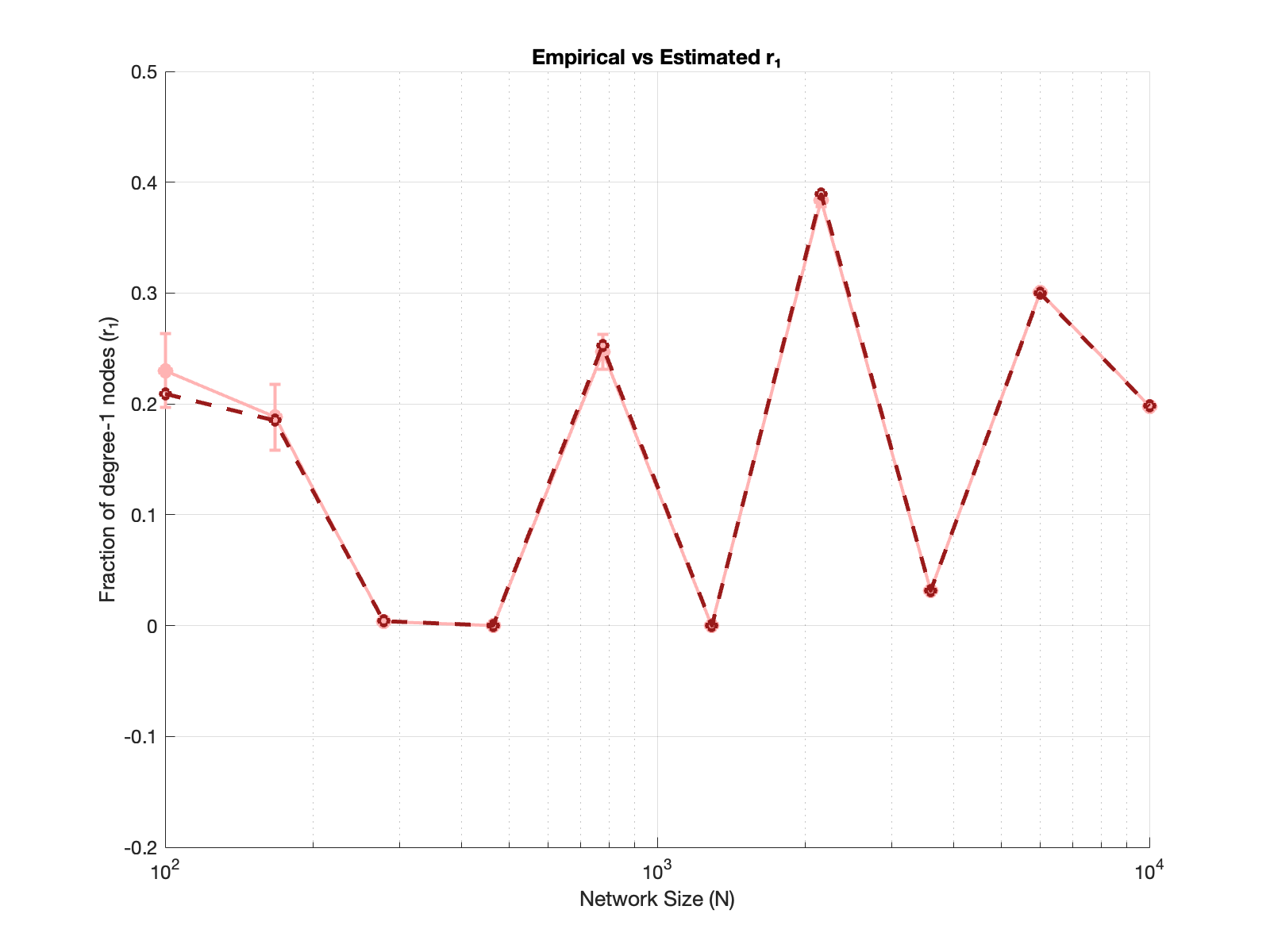}
    \caption{\textbf{Average local clustering coefficient $C$ and distance to $r_{0/1}$ versus network size $n$ (for $\alpha=0.3$).} 
    On the left, simulations are done by resampling weights every time an actual network is realized (sample of $100$), while on the left for each $n$ weights are extracted only once, and $100$ different adjacency matrices are realized from the same weights. On top
    we show the node-averaged local clustering coefficient ${C}$, both including (blue symbols) and excluding (purple symbols) nodes with degree $k=0,1$ (note that the latter evolves smoothly towards 1 with shrinking error bars as $n$ increases, while the former fluctuates with non-vanishing error bars, as a result of non-self-averaging).    
    The dashed blue line is $1$ minus the average of $r_{0/1}$ over realizations. Finally, in green the difference between $1-r_{0/1}$ and ${C}$ computed including nodes with $k<2$ (notice the shrinking error bars). In the middle, the light blue line with error bars is the actual value of $r_0$, while the dashed darker one is the approximation we derive in~\eqref{r0-approximation}. At the bottom, the light red line with error bars is the actual value of $r_1$, while the dashed darker one is the approximation we derive in~\eqref{r1-approximation}. }
    \label{fig:pan_03}
\end{figure}

\begin{figure}[h!]
    \centering
    \includegraphics[width=0.49\linewidth, height= 0.32\linewidth]{Figures/CC_avg.png}
    \includegraphics[width=0.49\linewidth, height= 0.32\linewidth]{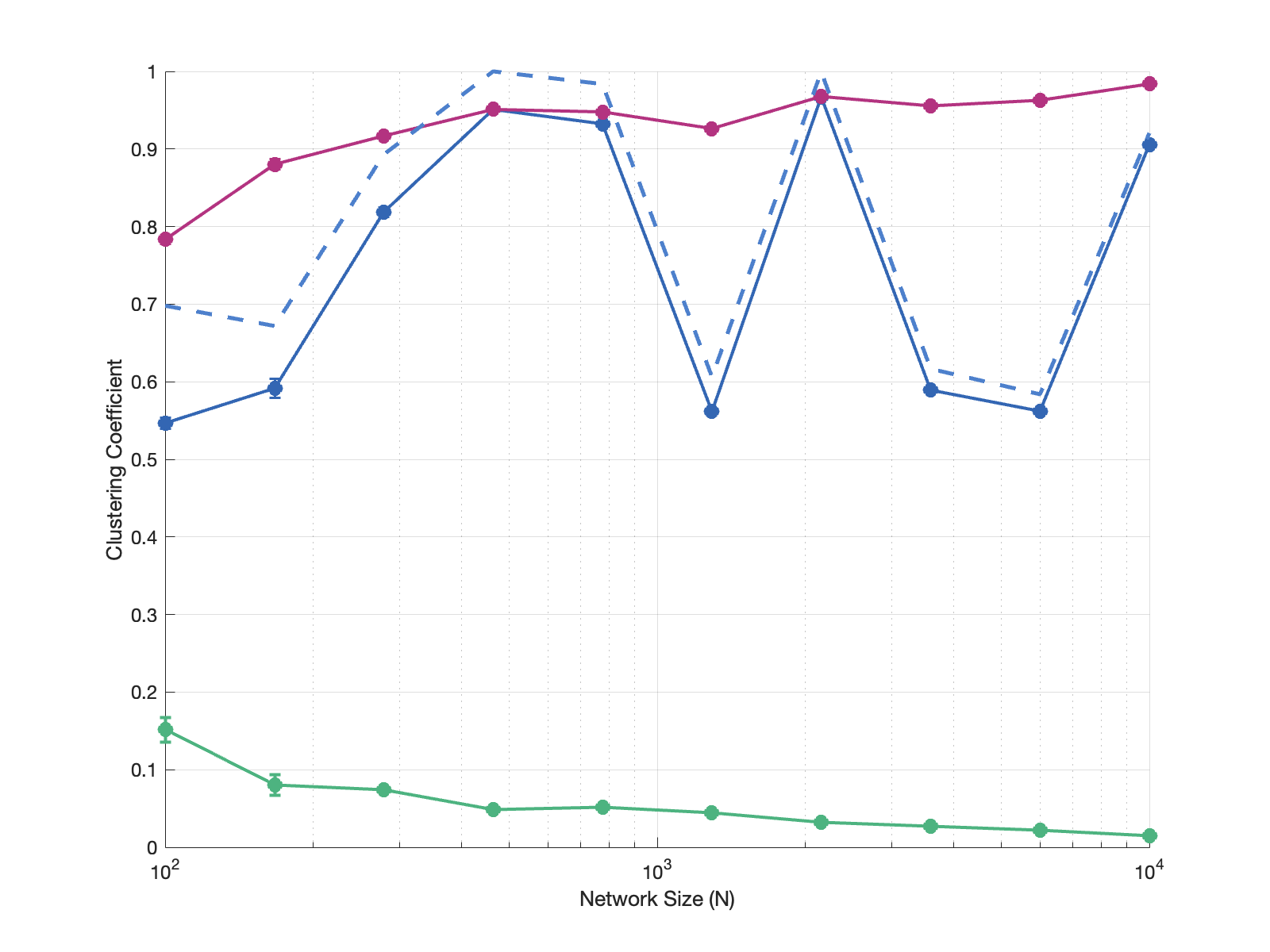}\\
    \includegraphics[width=0.49\linewidth, height= 0.32\linewidth]{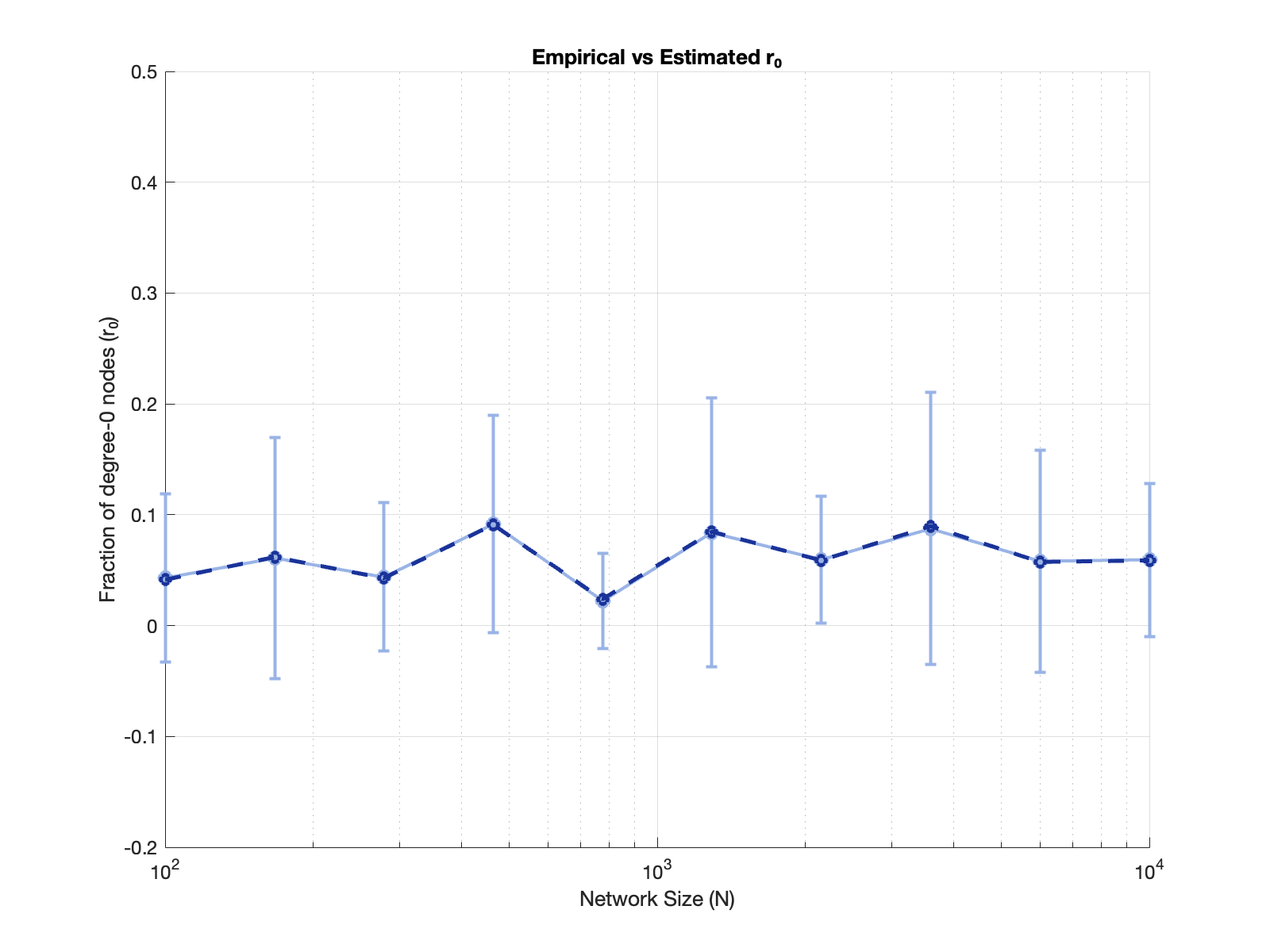}
    \includegraphics[width=0.49\linewidth, height= 0.32\linewidth]{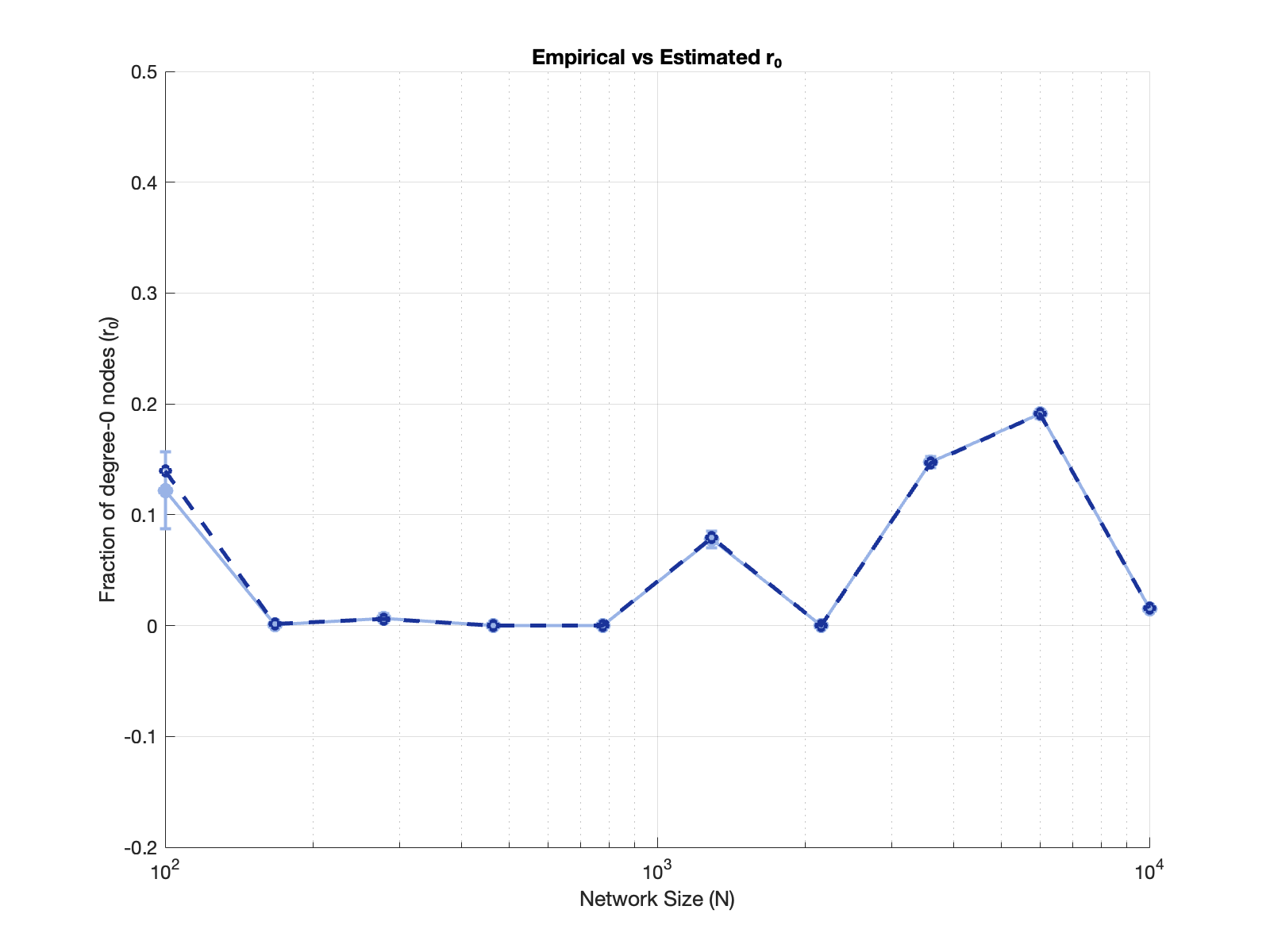}\\
    \includegraphics[width=0.49\linewidth, height= 0.32\linewidth]{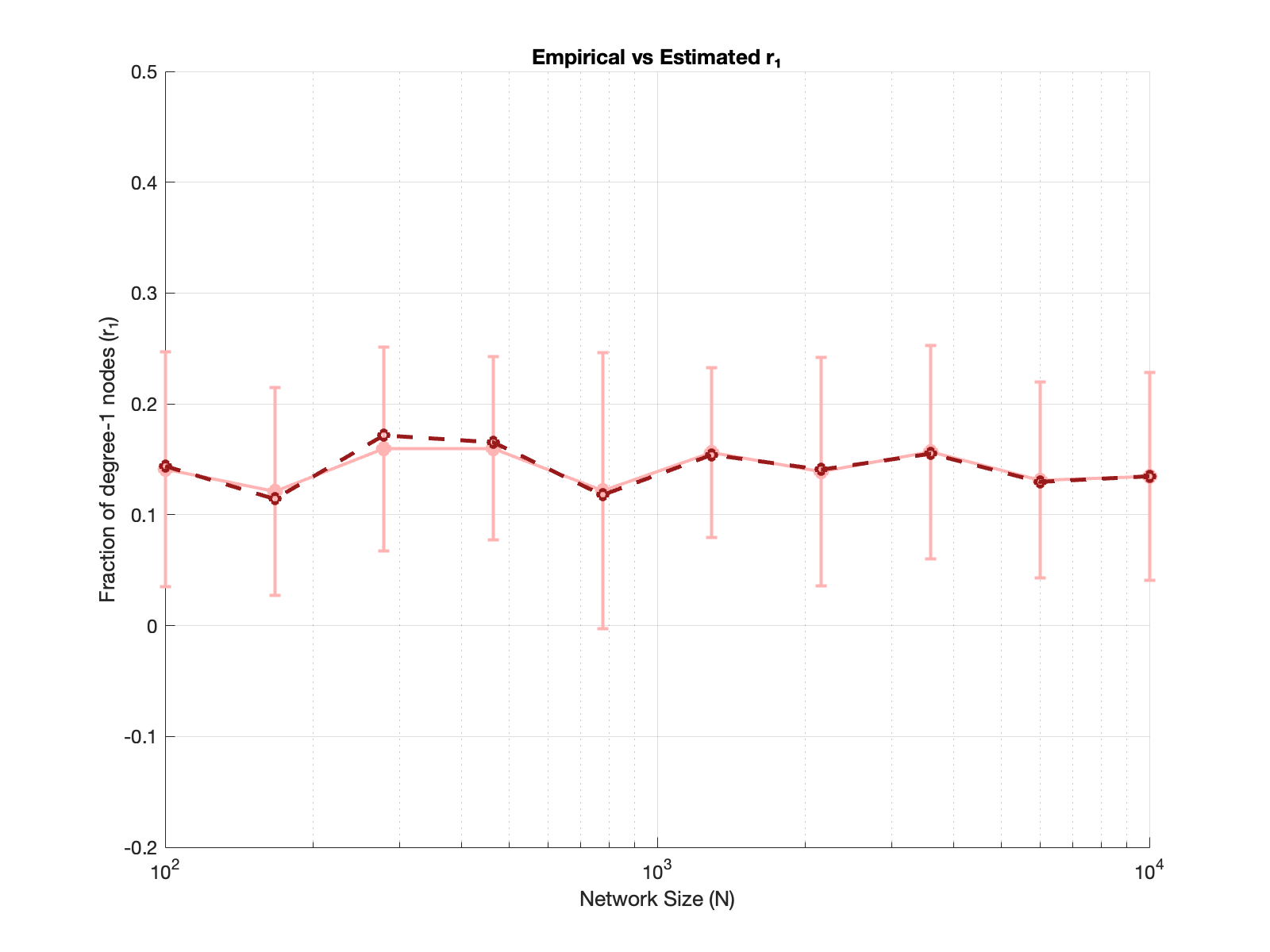}
    \includegraphics[width=0.49\linewidth, height= 0.32\linewidth]{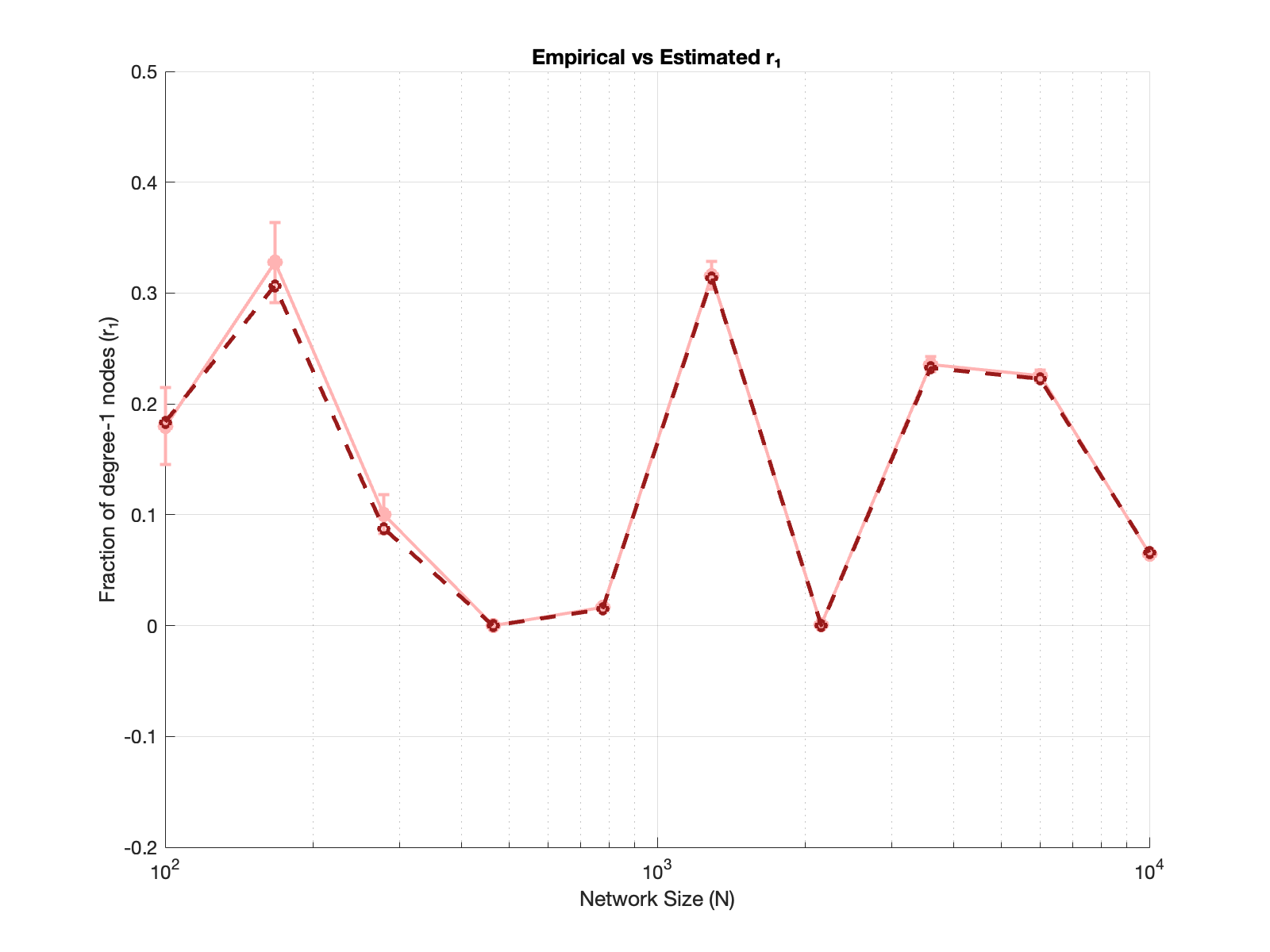}
    \caption{\textbf{Average local clustering coefficient $C$ and distance to $r_{0/1}$ versus network size $n$ (for $\alpha=0.5$).} 
    On the left, simulations are done by resampling weights every time an actual network is realized (sample of $1000$), while on the left for each $n$ weights are extracted only once, and $1000$ different adjacency matrices are realized from the same weights. On top
    we show the node-averaged local clustering coefficient ${C}$, both including (blue symbols) and excluding (purple symbols) nodes with degree $k=0,1$ (note that the latter evolves smoothly towards 1 with shrinking error bars as $n$ increases, while the former fluctuates with non-vanishing error bars, as a result of non-self-averaging).    
    The dashed blue line is $1$ minus the average of $r_{0/1}$ over realizations. Finally, in green the difference between $1-r_{0/1}$ and ${C}$ computed including nodes with $k<2$ (notice the shrinking error bars). In the middle, the light blue line with error bars is the actual value of $r_0$, while the dashed darker one is the approximation we derive in~\eqref{r0-approximation}. At the bottom, the light red line with error bars is the actual value of $r_1$, while the dashed darker one is the approximation we derive in~\eqref{r1-approximation}. }
    \label{fig:pan_05}
\end{figure}

\begin{figure}[h!]
    \centering
    \includegraphics[width=0.49\linewidth, height= 0.32\linewidth]{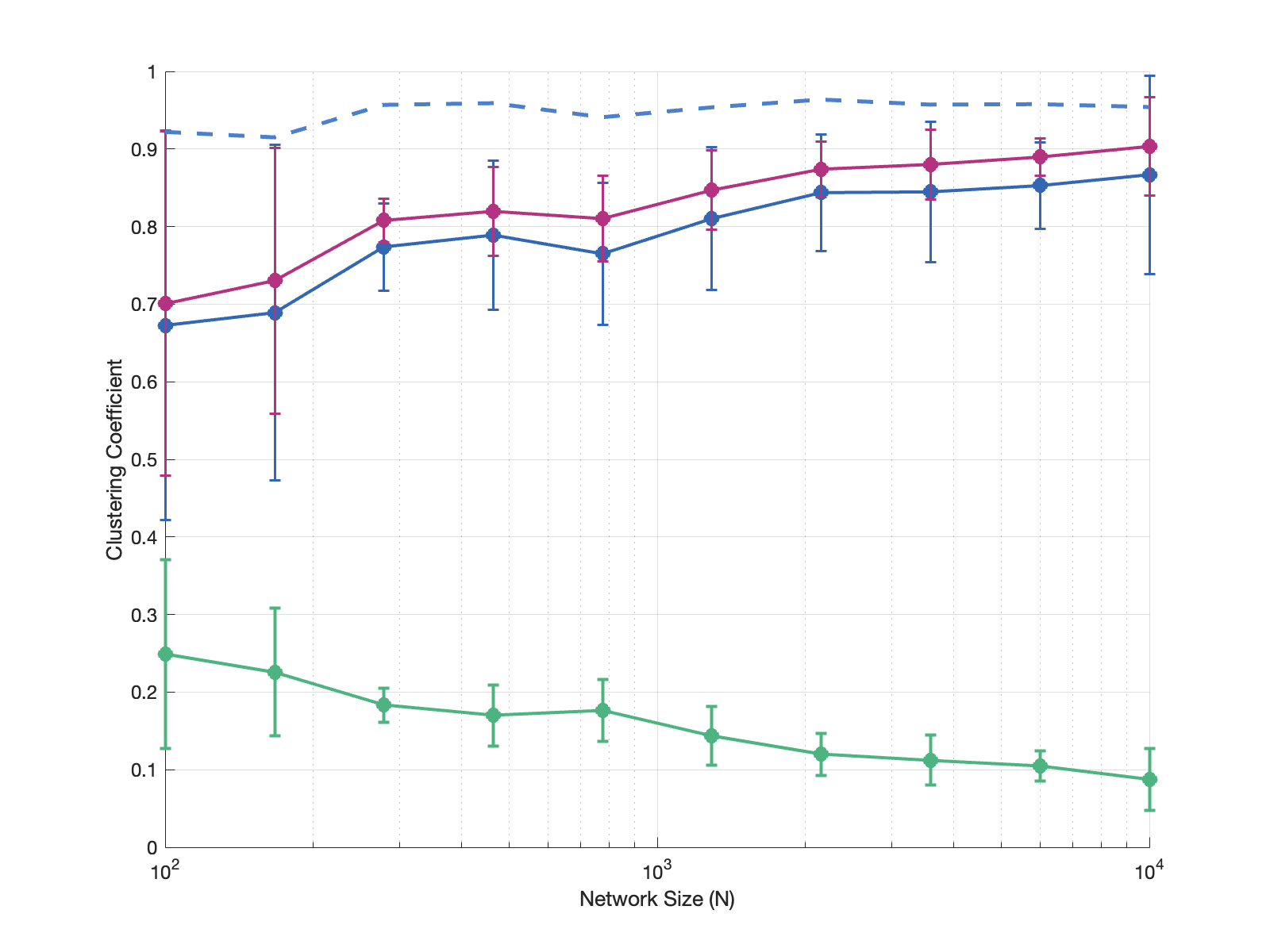}
    \includegraphics[width=0.49\linewidth, height= 0.32\linewidth]{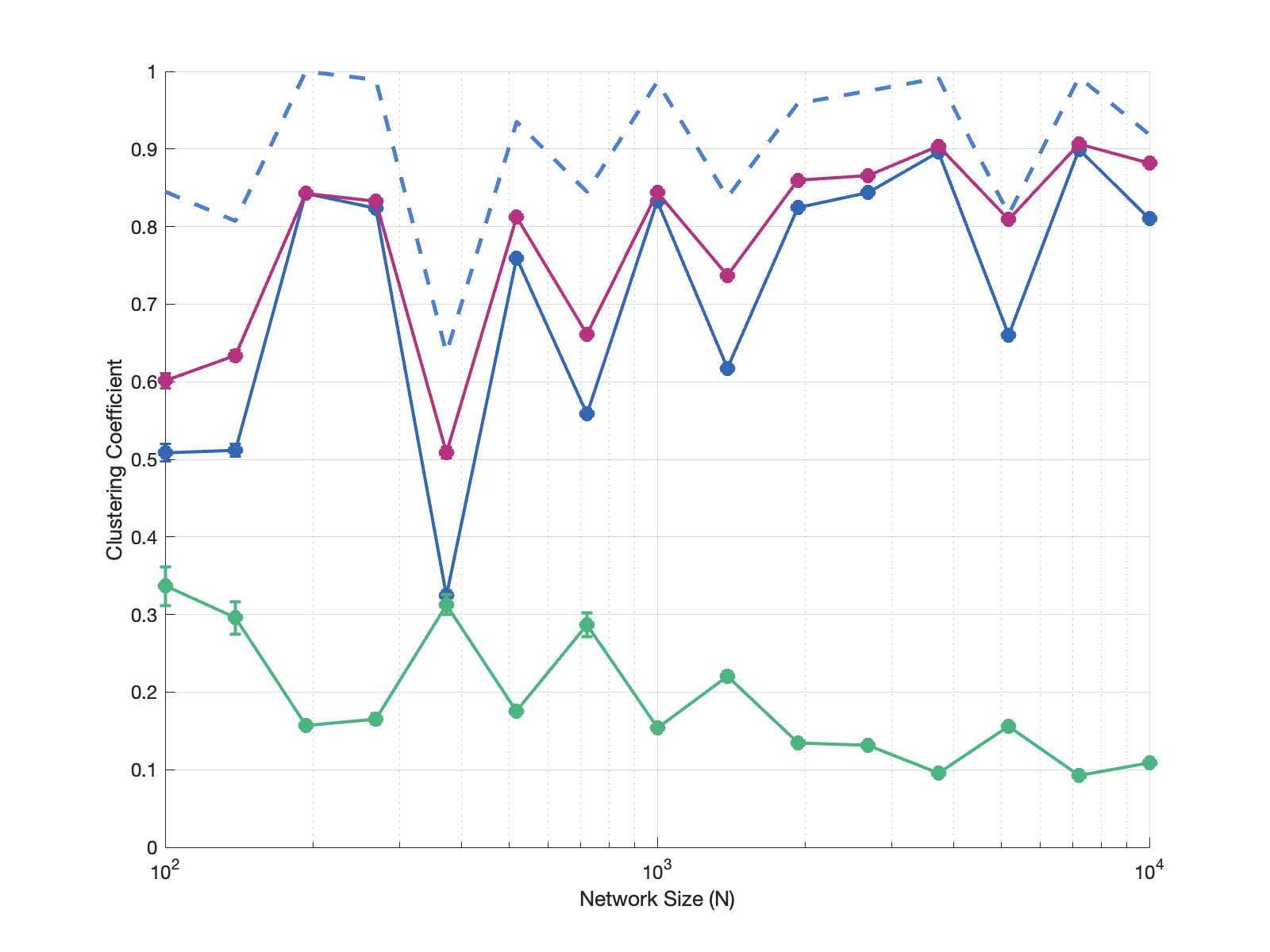}\\
    \includegraphics[width=0.49\linewidth, height= 0.32\linewidth]{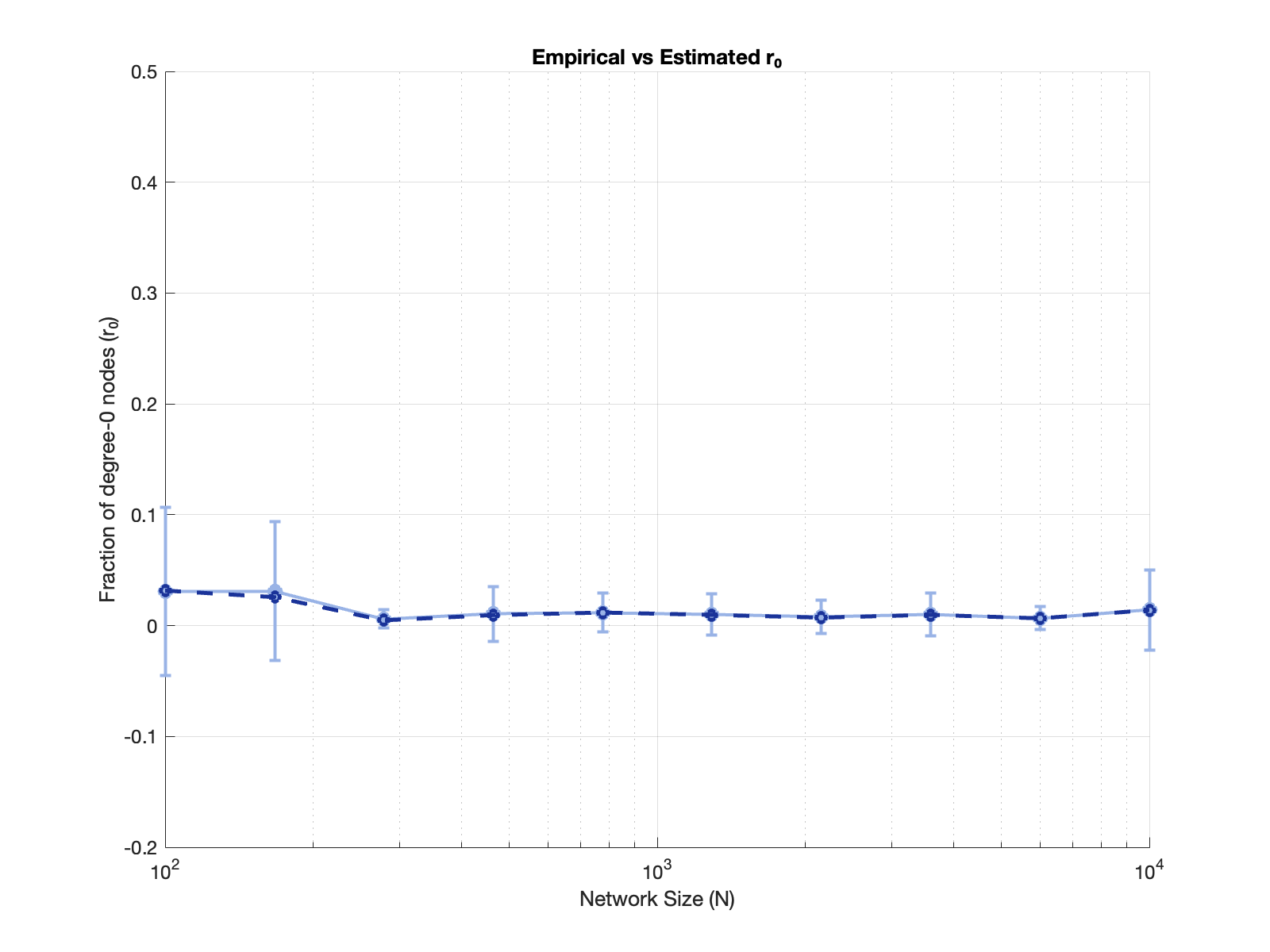}
    \includegraphics[width=0.49\linewidth, height= 0.32\linewidth]{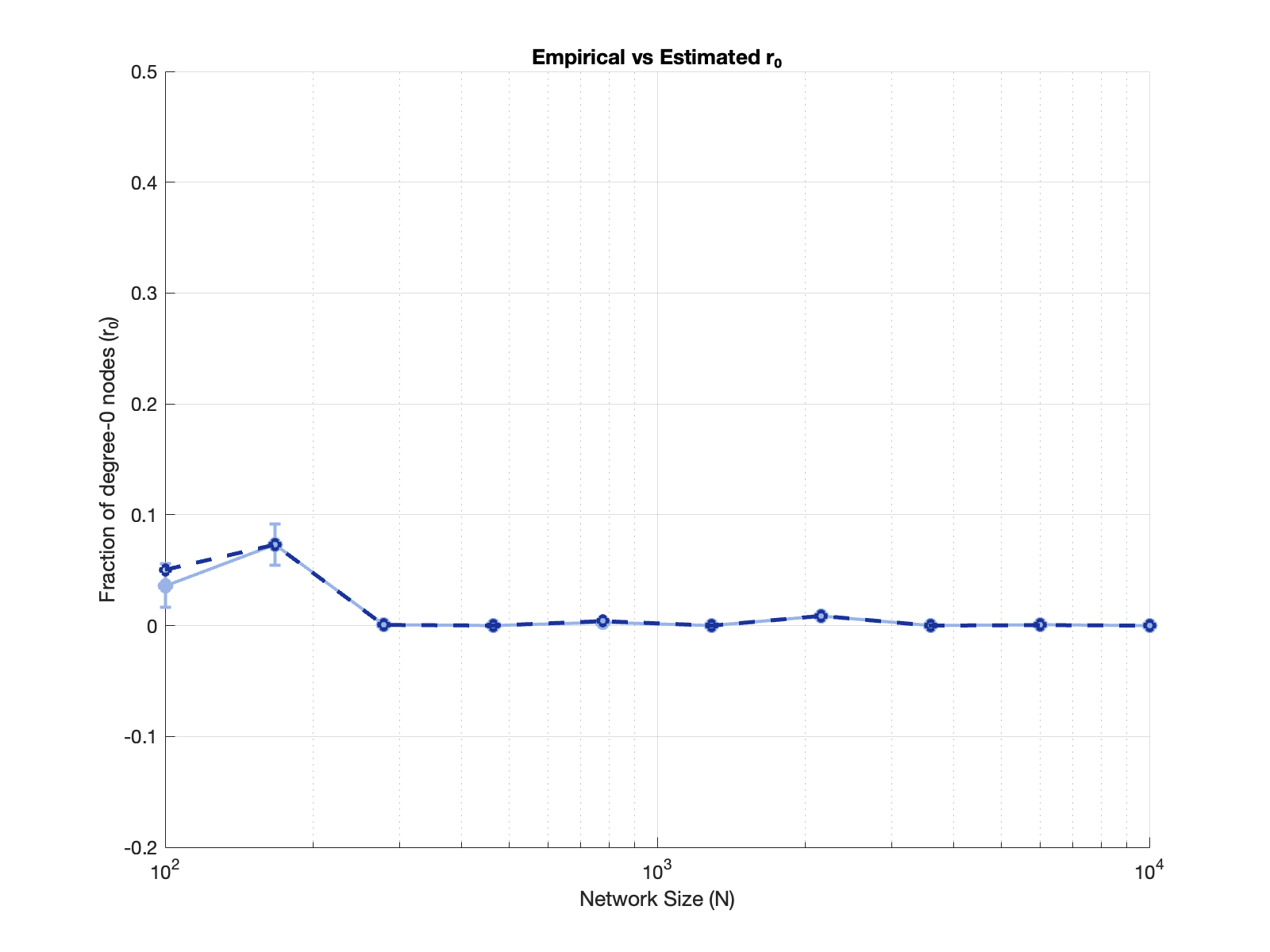}\\
    \includegraphics[width=0.49\linewidth, height= 0.32\linewidth]{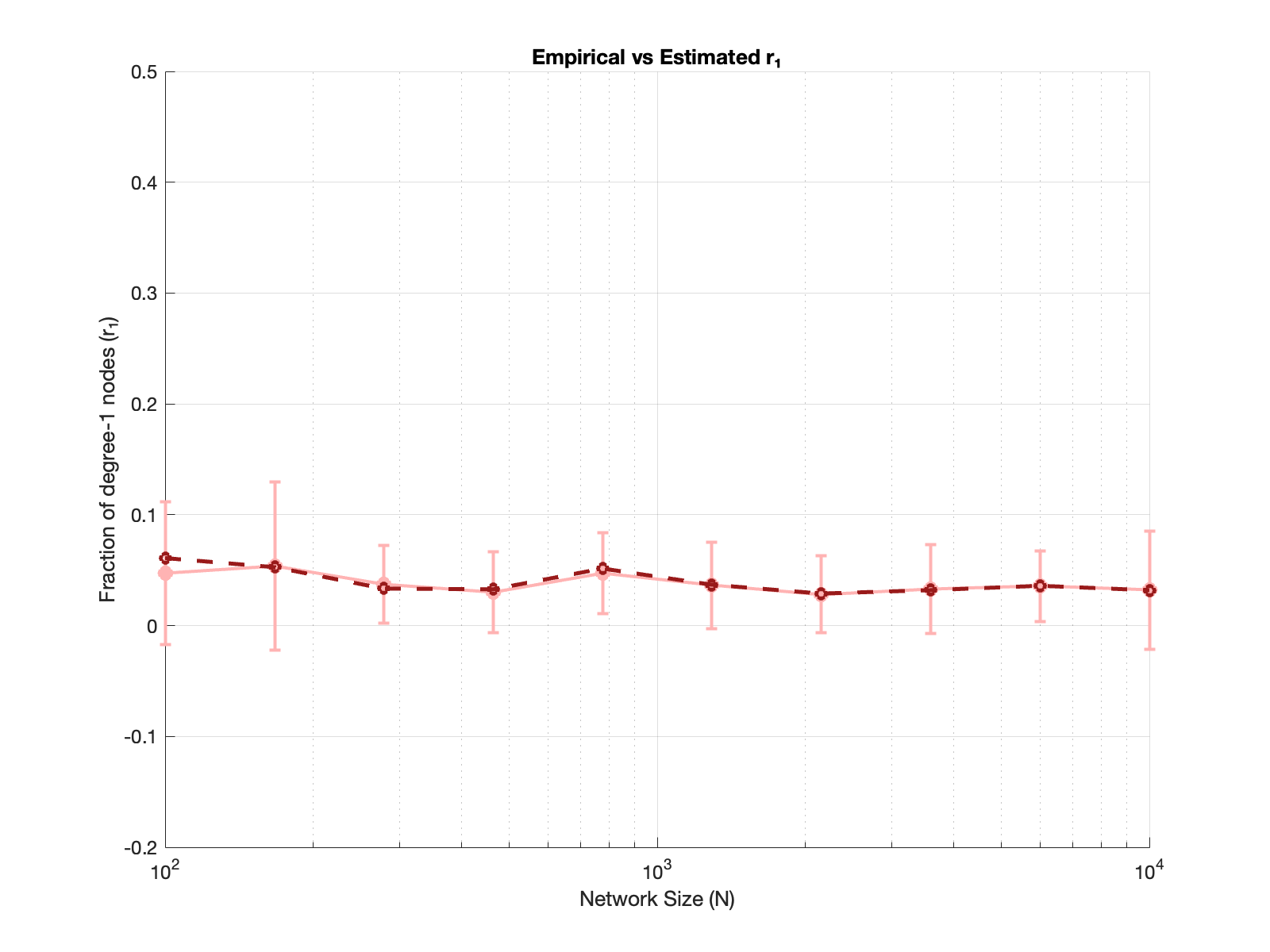}
    \includegraphics[width=0.49\linewidth, height= 0.32\linewidth]{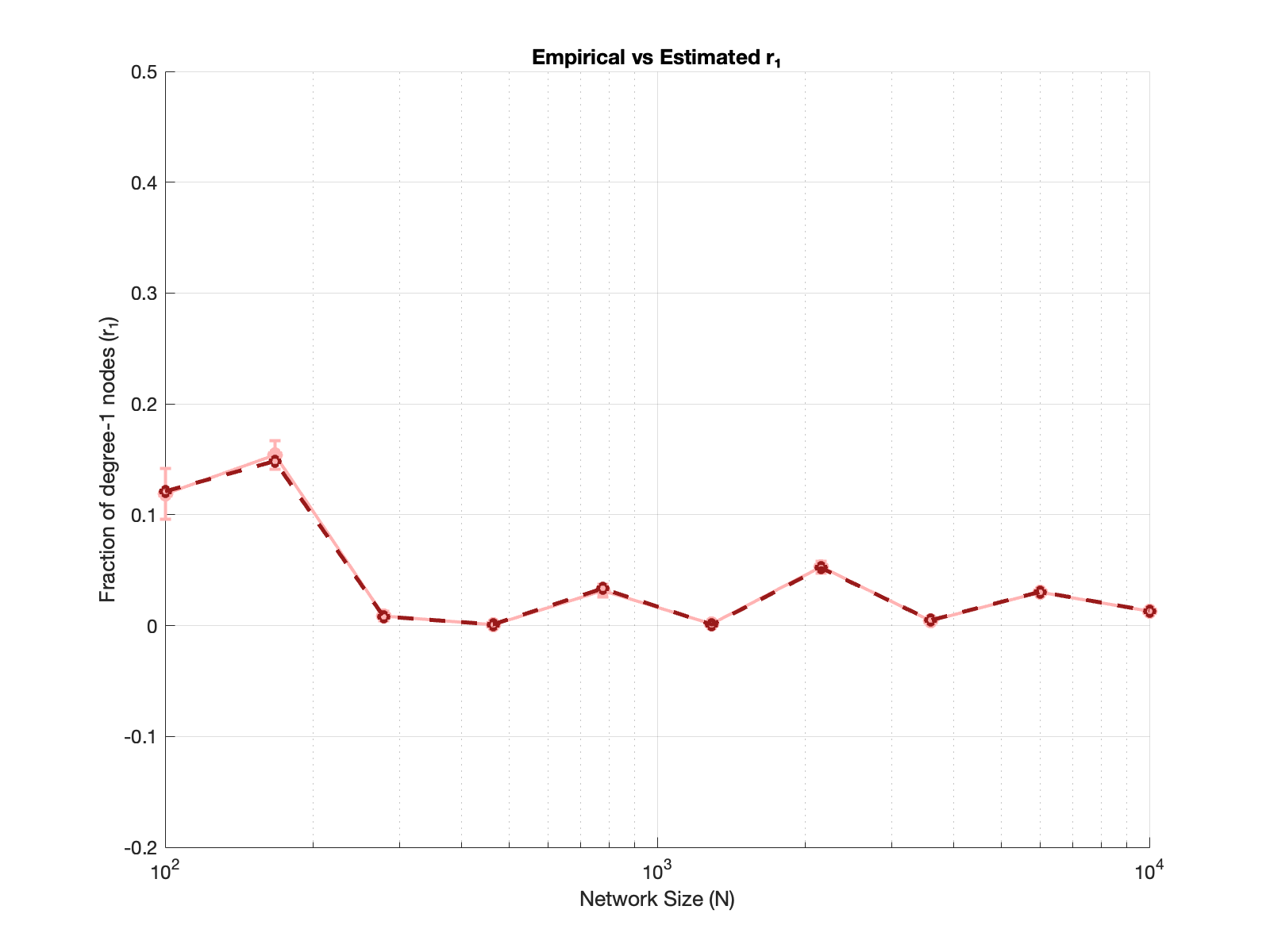}
    \caption{\textbf{Average local clustering coefficient $C$ and distance to $r_{0/1}$ versus network size $n$ (for $\alpha=0.7$).} 
    On the left, simulations are done by resampling weights every time an actual network is realized (sample of $100$), while on the left for each $n$ weights are extracted only once, and $100$ different adjacency matrices are realized from the same weights. On top
    we show the node-averaged local clustering coefficient ${C}$, both including (blue symbols) and excluding (purple symbols) nodes with degree $k=0,1$ (note that the latter evolves smoothly towards 1 with shrinking error bars as $n$ increases, while the former fluctuates with non-vanishing error bars, as a result of non-self-averaging).    
    The dashed blue line is $1$ minus the average of $r_{0/1}$ over realizations. Finally, in green the difference between $1-r_{0/1}$ and ${C}$ computed including nodes with $k<2$ (notice the shrinking error bars). In the middle, the light blue line with error bars is the actual value of $r_0$, while the dashed darker one is the approximation we derive in~\eqref{r0-approximation}. At the bottom, the light red line with error bars is the actual value of $r_1$, while the dashed darker one is the approximation we derive in~\eqref{r1-approximation}. }
    \label{fig:pan_07}
\end{figure}

\section{Lack of self-averaging in the fraction of next low-degree nodes} 
The fact that both $r_0$ and $r_1$ do not converge to a limiting deterministic value, but rather {\em in distribution} to a random variable whose variance is non-vanishing even when $n$ tends to infinity, is a peculiar aspect of the MSM. 
This property applies not only to isolated or pendant nodes, but also to all the next `low degree' nodes. 
Indeed, one can estimate $r_k$, the fraction of nodes whose degree is $k$, generalizing \eqref{r0-computation-rep} and \eqref{r1-first-step} as follows:
\begin{eqnarray}
    r_k &\equiv& \frac{1}{n} \sum_{\substack{i_{0}, i_{1}, \ldots, i_{k} \\ (i_r\ne i_s\, \forall r,s)}} p_{i_0 i_1}p_{i_0 i_2} \cdots p_{i_0 i_k} \prod_{j\colon j\neq i_0, i_1, \ldots i_k} (1-p_{i_0 j}) \\
    &=& \frac{1}{n} \sum_{\substack{i_{0}, i_{1}, \ldots, i_{k} \\ (i_r\ne i_s\, \forall r,s)}} \left(1 - \e^{-\delta_n w_{i_0} w_{i_1}}\right) \left(1 - \e^{-\delta_n w_{i_0} w_{i_2}}\right) \cdots \left(1 - \e^{-\delta_n w_{i_0} w_{i_k}}\right) \e^{-\delta_n w_{i_0} \sum_{j\colon j \neq i_{0}, i_1, \ldots i_k}w_{j}}.\nonumber
\end{eqnarray}
The key point now is that if one wants to proceed as in the last step of \eqref{r1-first-step} and replace the sum over $i_0$ divided by $n$ with the expectation over the whole sample, one should ensure that the sample without $k$ weights is still representative of a Pareto distribution. 
This happens if $k$ is small compared to $n$. In this case, we can proceed as in \eqref{r1-first-step} with the approximation
\begin{eqnarray}
    r_{k} &\approx& \sum_{\substack{i_1, i_2, \ldots, i_{k}\\
    (i_r\ne i_s\, \forall r,s)} }\EX\left[ (1-\e^{-W \delta_n w_{i_1}}) (1-\e^{-W \delta_n w_{i_2}}) \cdots (1-\e^{-W \delta_n w_{i_k}}) \e^{-W \delta_{n} \sum_{j\colon j \neq i_{0}, i_1, \ldots i_k}w_{j}}\right] \\
    &=& \sum_{\substack{i_1, i_2, \ldots, i_{k}\\
    (i_r\ne i_s\, \forall r,s)} } \EX\left[ (1-\e^{-W y^{(n)}_{i_1}}) (1-\e^{-W y^{(n)}_{i_2}}) \cdots (1-\e^{-W y^{(n)}_{i_k}}) \e^{-W  \left(S_n - \sum_{j = i_1, i_2, \ldots i_k} y^{(n)}_{j}\right)}\right].\nonumber
\end{eqnarray}
conditional on the realized values of the various $y^{(n)}_{j}$. Again, as $n \to \infty$, those random variables will converge {\em in distribution} to the Gamma variables \eqref{conv-sum-Pareto-order-statistics}, and also $r_k$ will be converging {\em in distribution} to the quantity
\begin{equation}
    r_k \convd \sum_{\substack{i_1, i_2, \ldots i_{k}\\
    (i_r\ne i_s\, \forall r,s)} } \EX\left[ (1-\e^{-W y_{i_1}}) (1-\e^{-W y_{i_2}}) \ldots (1-\e^{-W y_{i_k}}) \e^{-W  \left(\mathcal{S} - \sum_{j = i_1, i_2, \ldots i_k} y_{j}\right)}\right],
\end{equation}
thus still fluctuating macroscopically even in the large $n$ limit, as we confirm in Figures~\ref{fig:convergence_in_dist_across} and~\ref{fig:convergence_in_dist_dist}. Indeed the figures show that the dispersion of $r_k$ is significant, with the fluctuations shrinking as $k$ (and not $n$) increases, suggesting that, for larger values of the degree, self-averaging is eventually restored. We leave the thorough study of this intriguing phenomenon to future works on the MSM.

\begin{figure}
    \centering
    \includegraphics[width=0.45\linewidth]{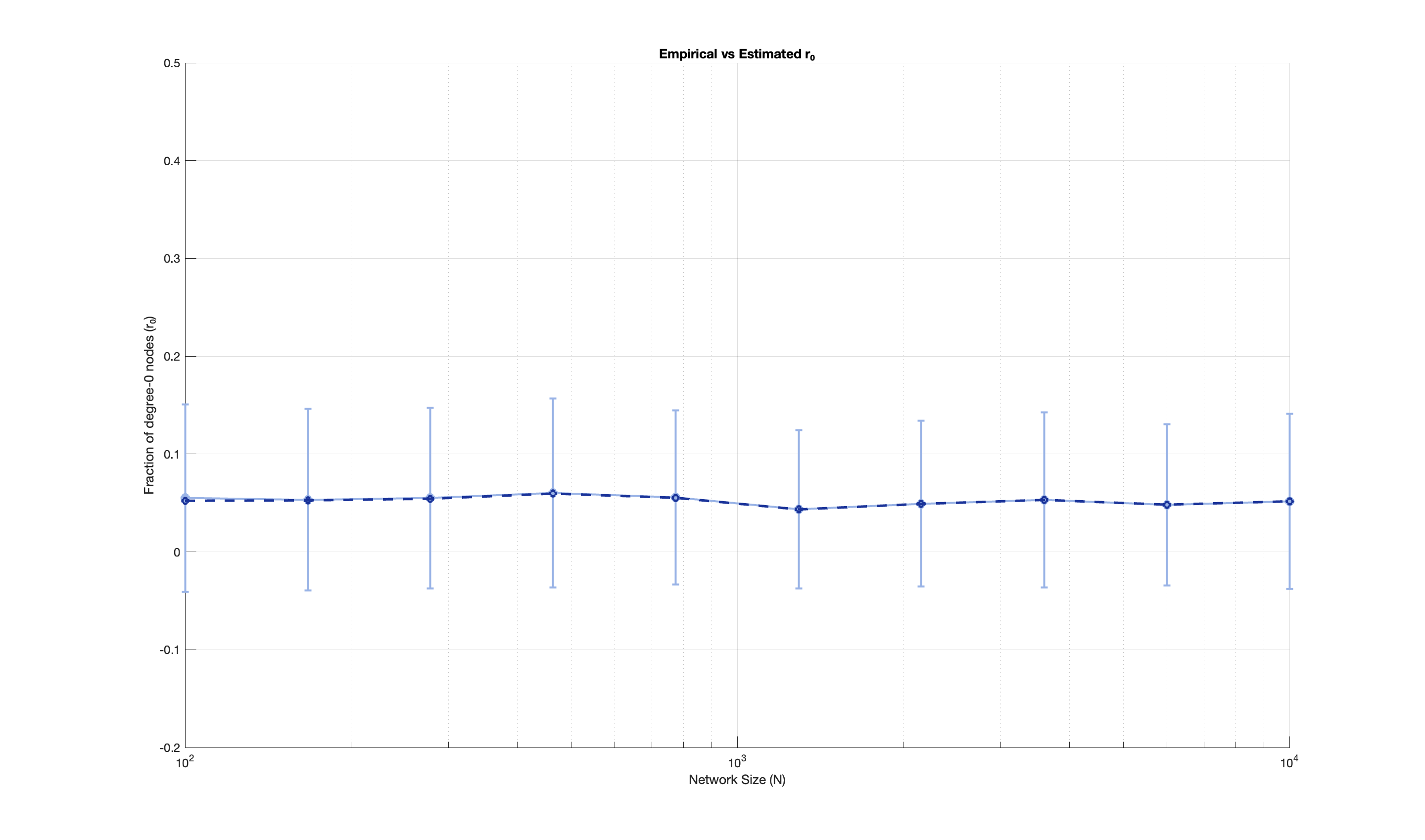}
    \includegraphics[width=0.45\linewidth]{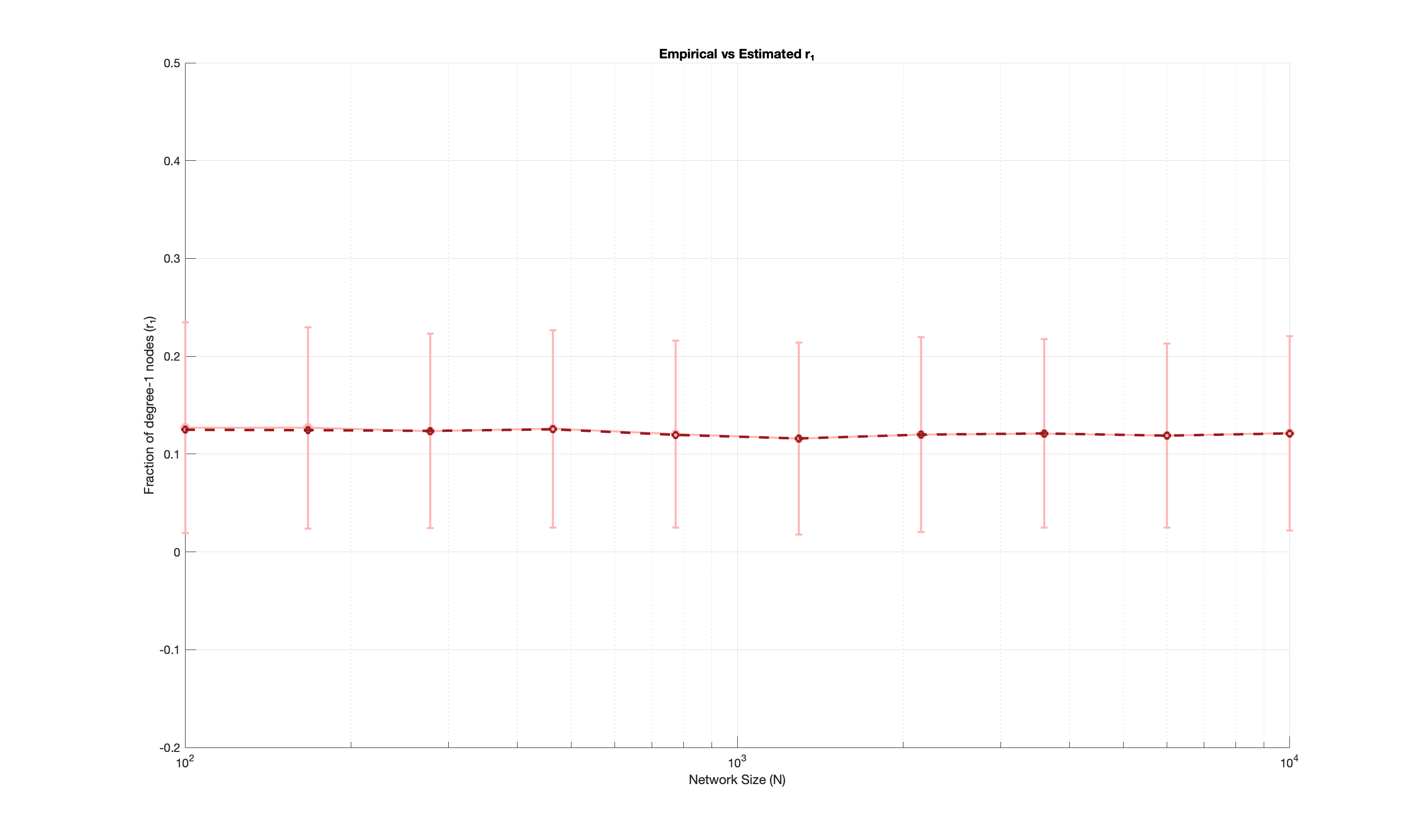}\\
    \includegraphics[width=0.45\linewidth]{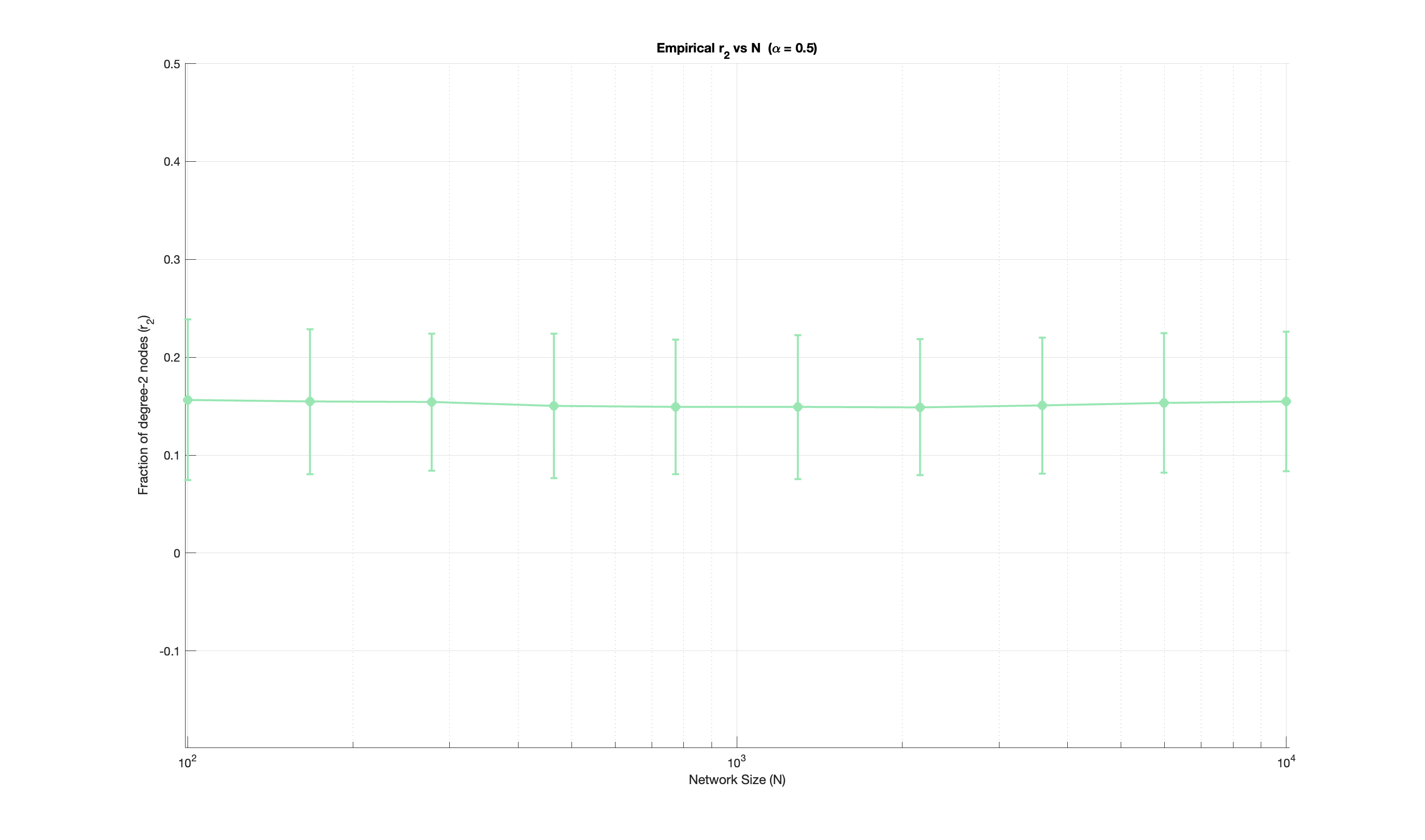}
    \includegraphics[width=0.45\linewidth]{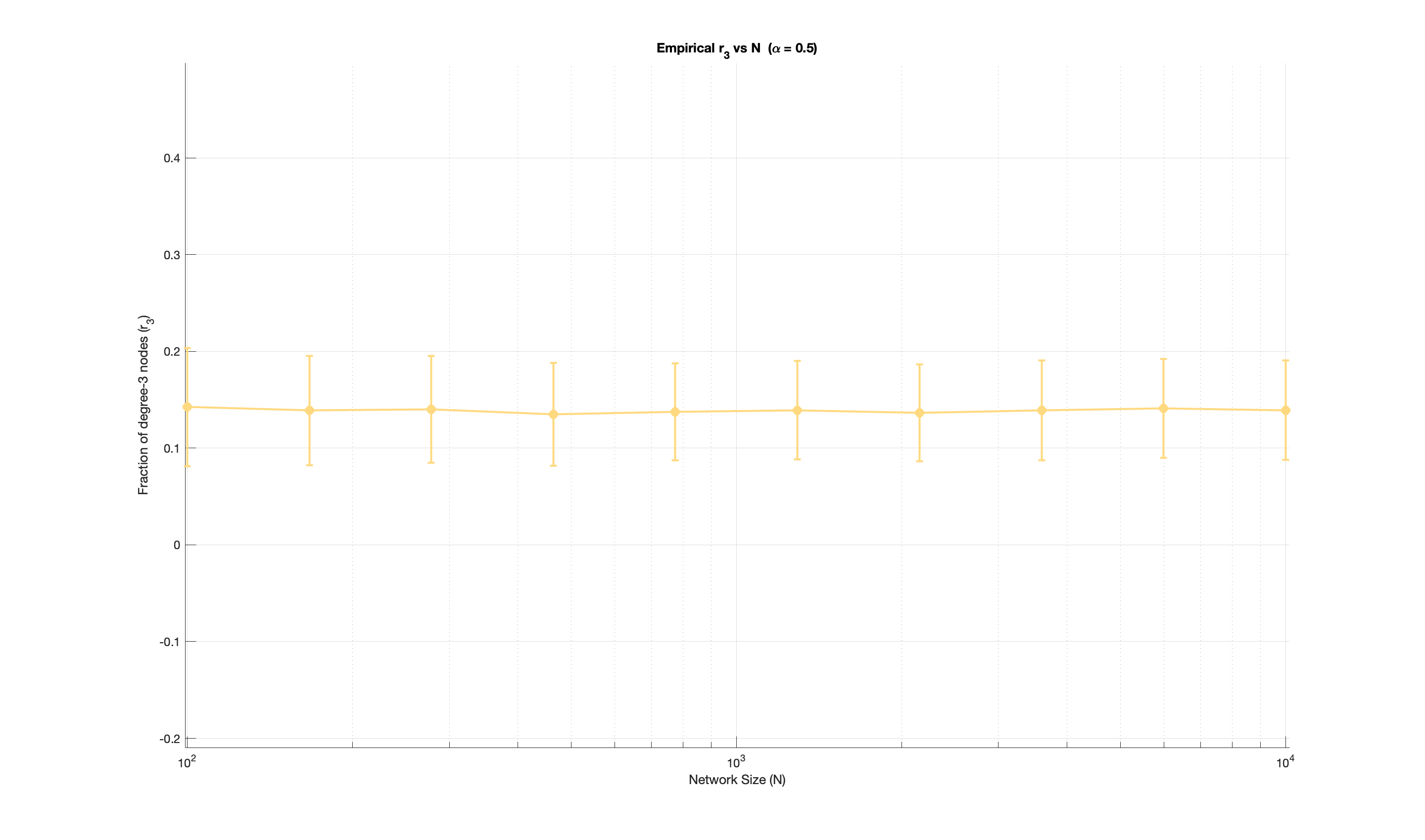}\\
    \includegraphics[width=0.45\linewidth]{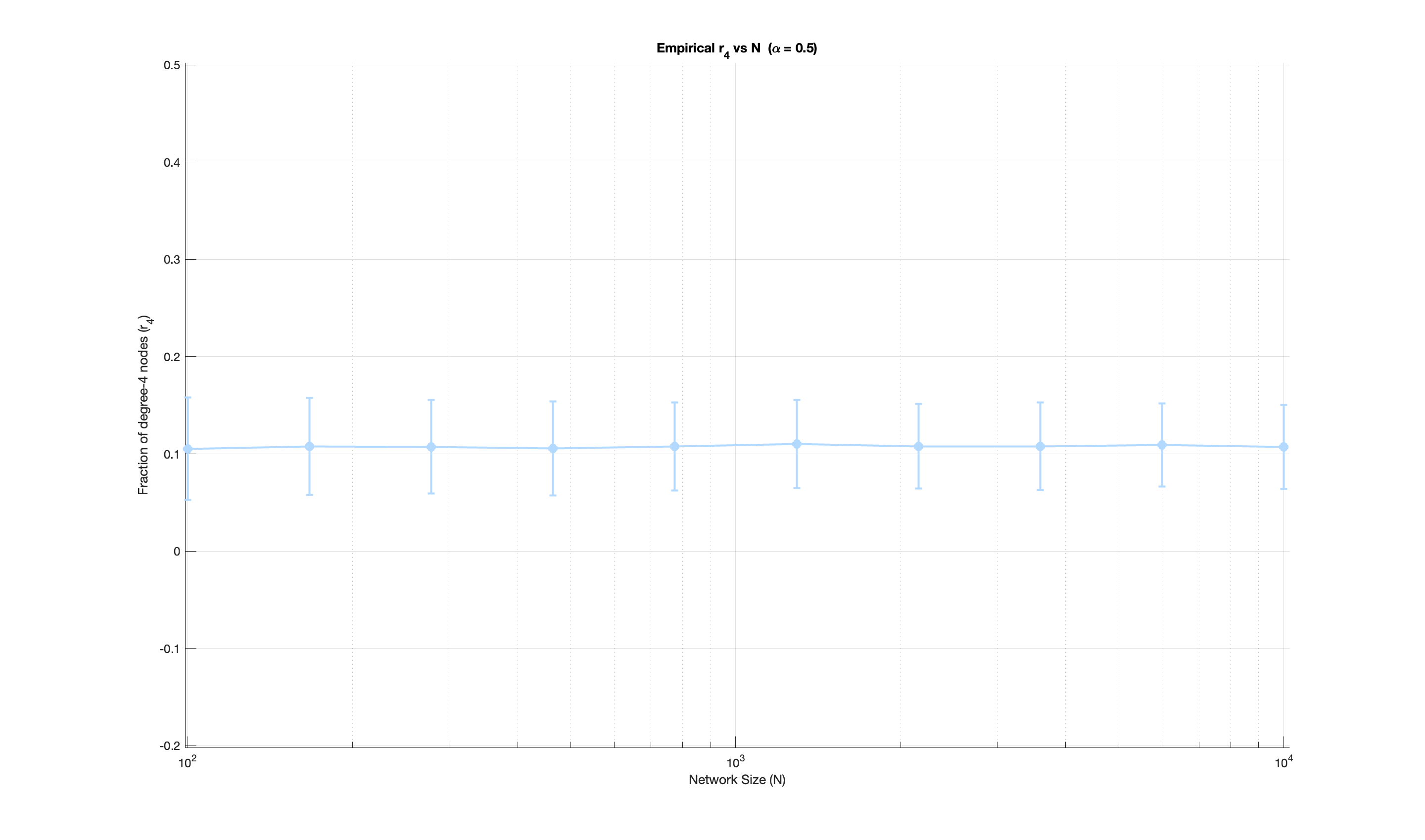}
    \includegraphics[width=0.45\linewidth]{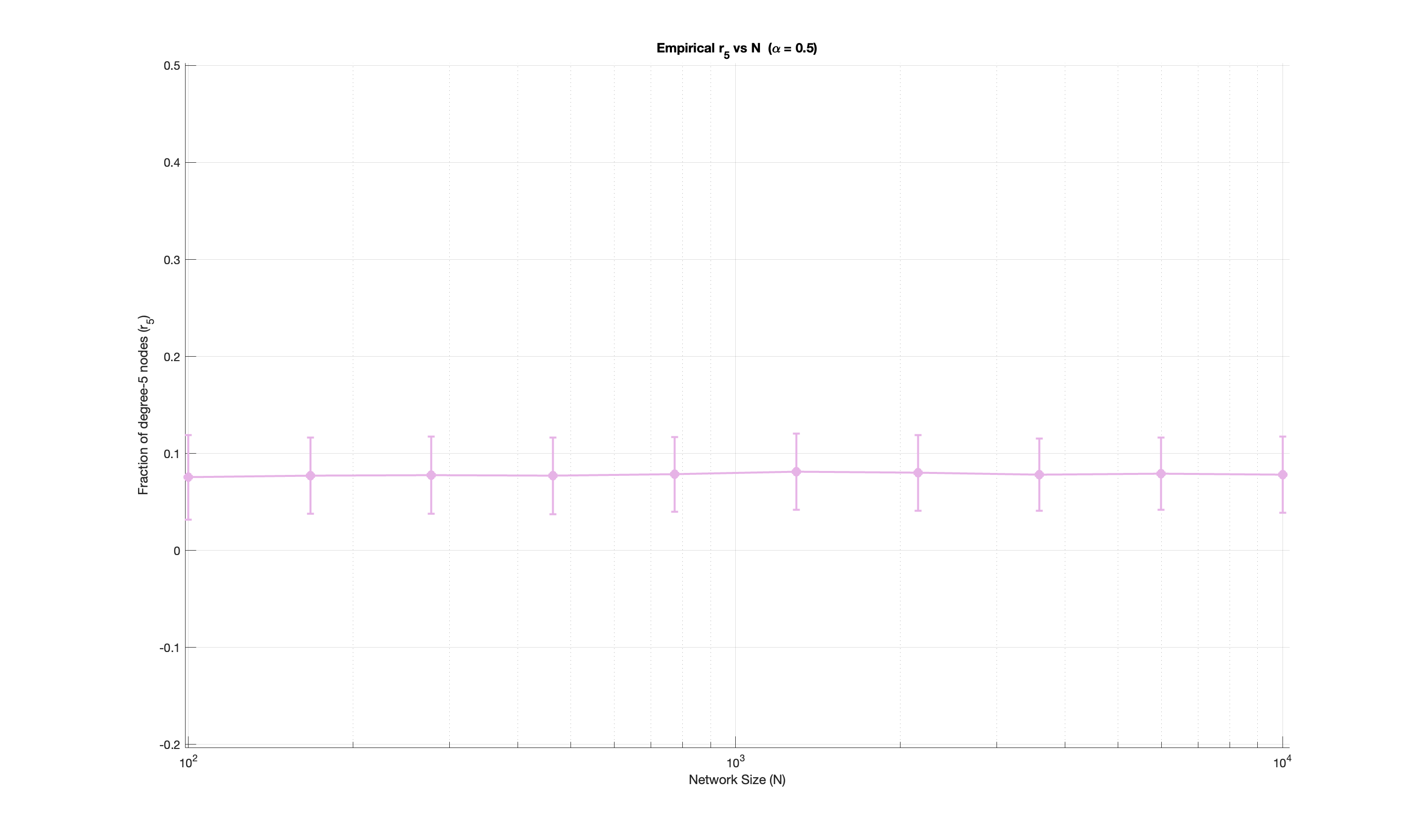}
    \caption{\textbf{Fractions of low-degree (up to $k=5$) nodes versus network size, with $\alpha = 0.5$. } From left to right and top to bottom, we plot the fractions of nodes with a specific low degree (respectively $k =0, 1, 2, 3, 4, 5$) across a sample of $1000$ different realizations for each value of $n$. For the first two figures we include the theoretical estimate we explicitly computed in ~\eqref{r0-approximation}~\eqref{r1-approximation}. We notice that the average value across different values of $n$ is roughly constant, while the dispersion of values around that average does not shrink as $n$ increases, as one would expect by assuming the self averaging of the quantities. We also notice that the error bars shrink as $k$ increases, suggesting that, for larger-degree nodes, self averaging eventually occurs.}
    \label{fig:convergence_in_dist_across}
\end{figure}

\begin{figure}
    \centering
    \includegraphics[width=0.45\linewidth]{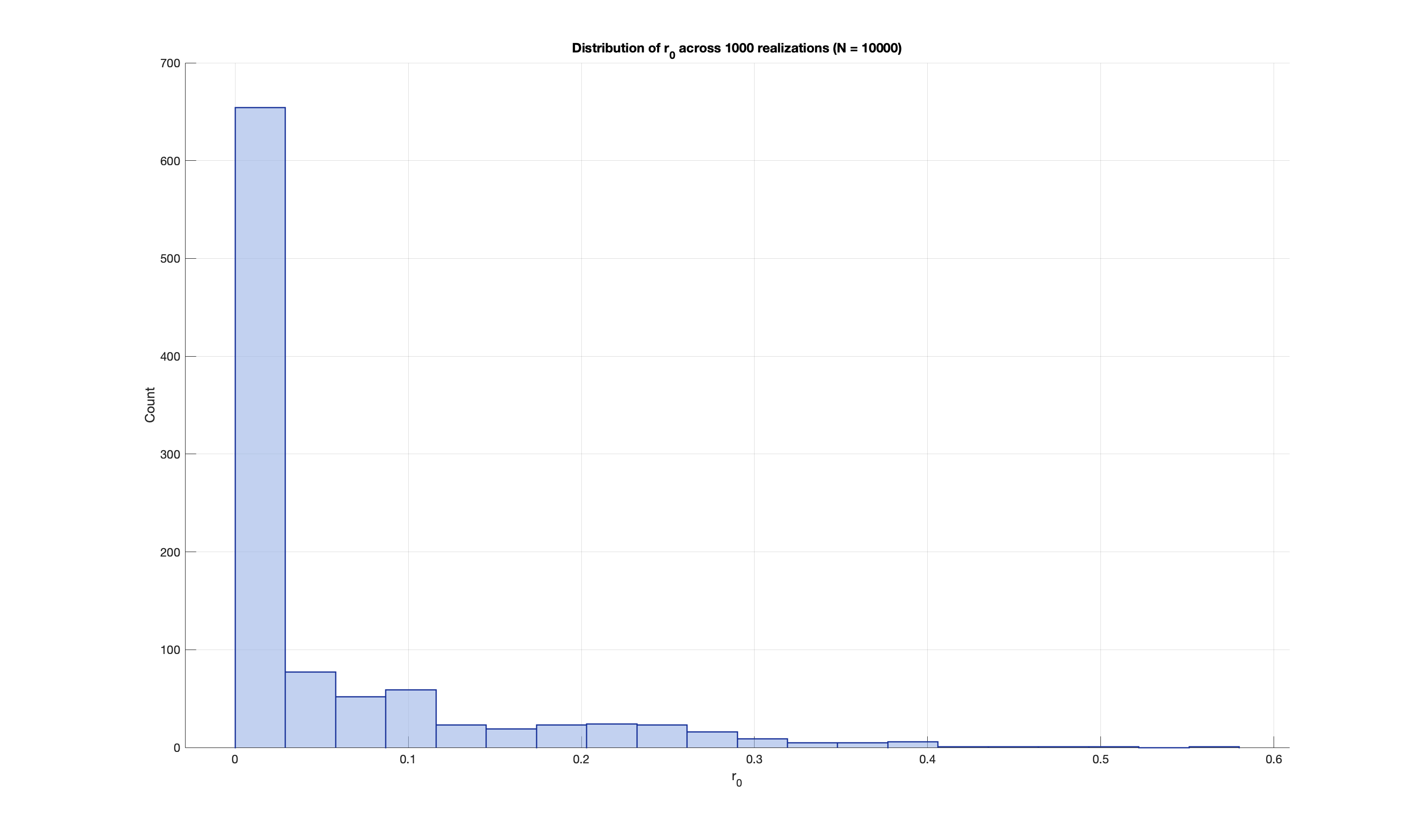}
    \includegraphics[width=0.45\linewidth]{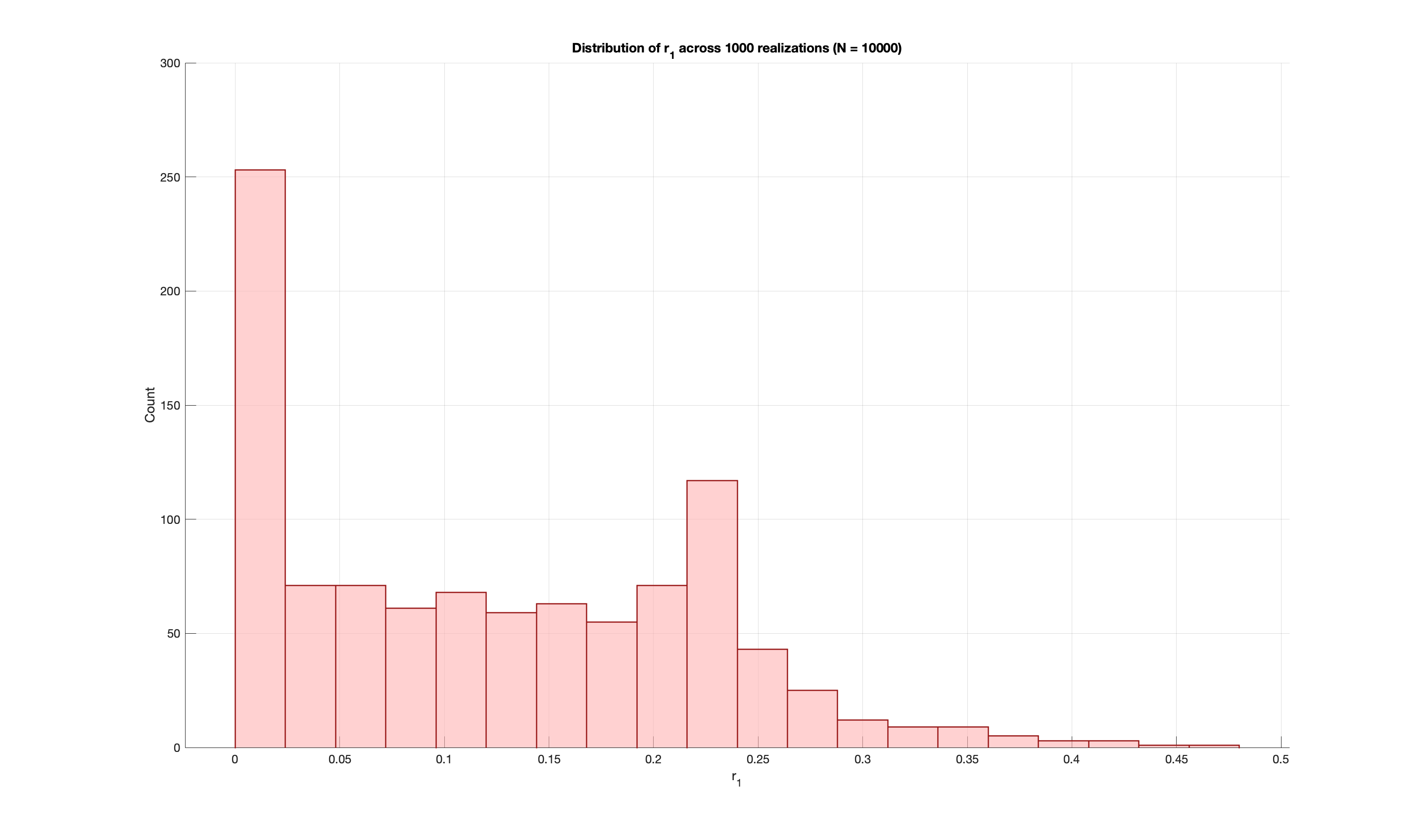}\\
    \includegraphics[width=0.45\linewidth]{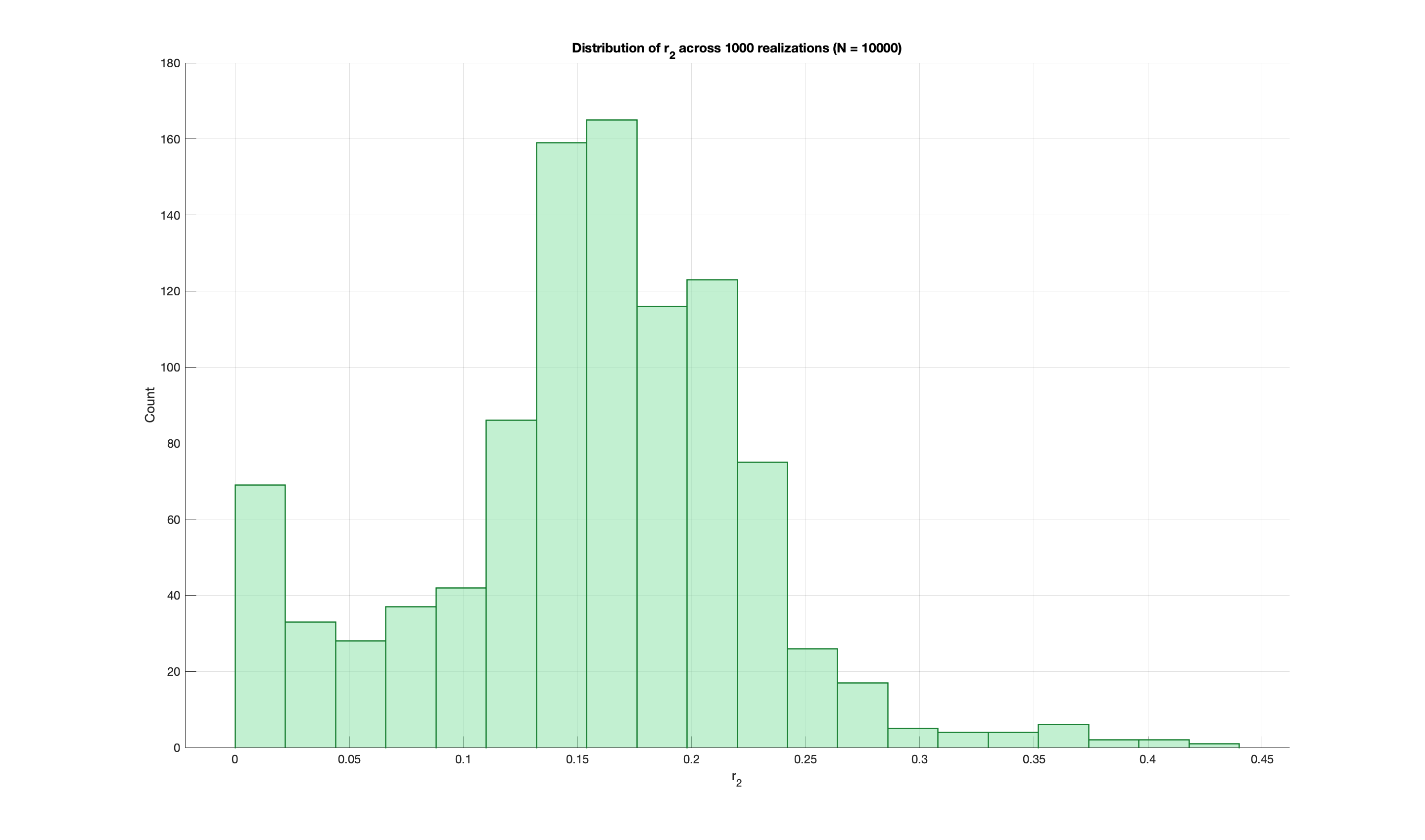}
    \includegraphics[width=0.45\linewidth]{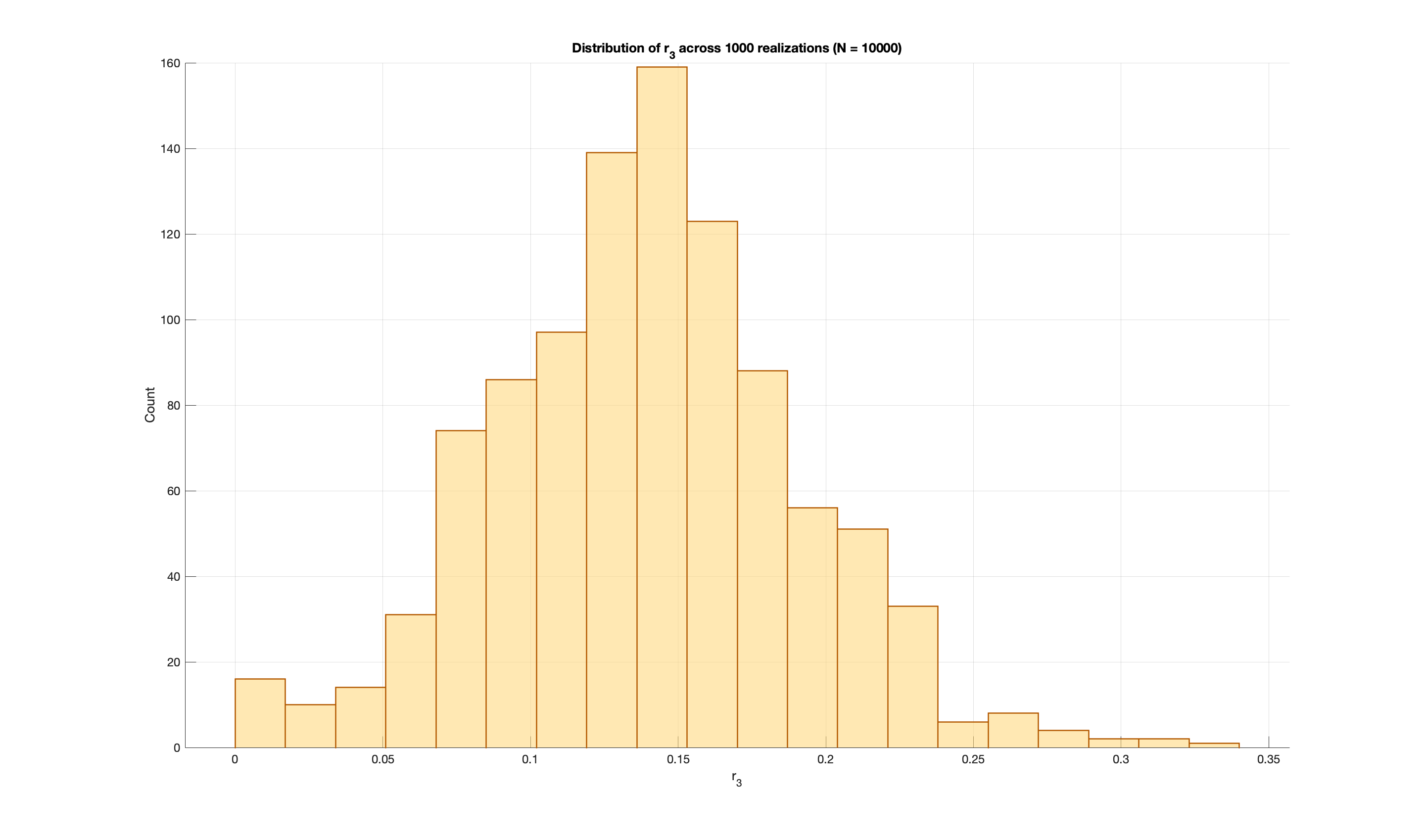}\\
    \includegraphics[width=0.45\linewidth]{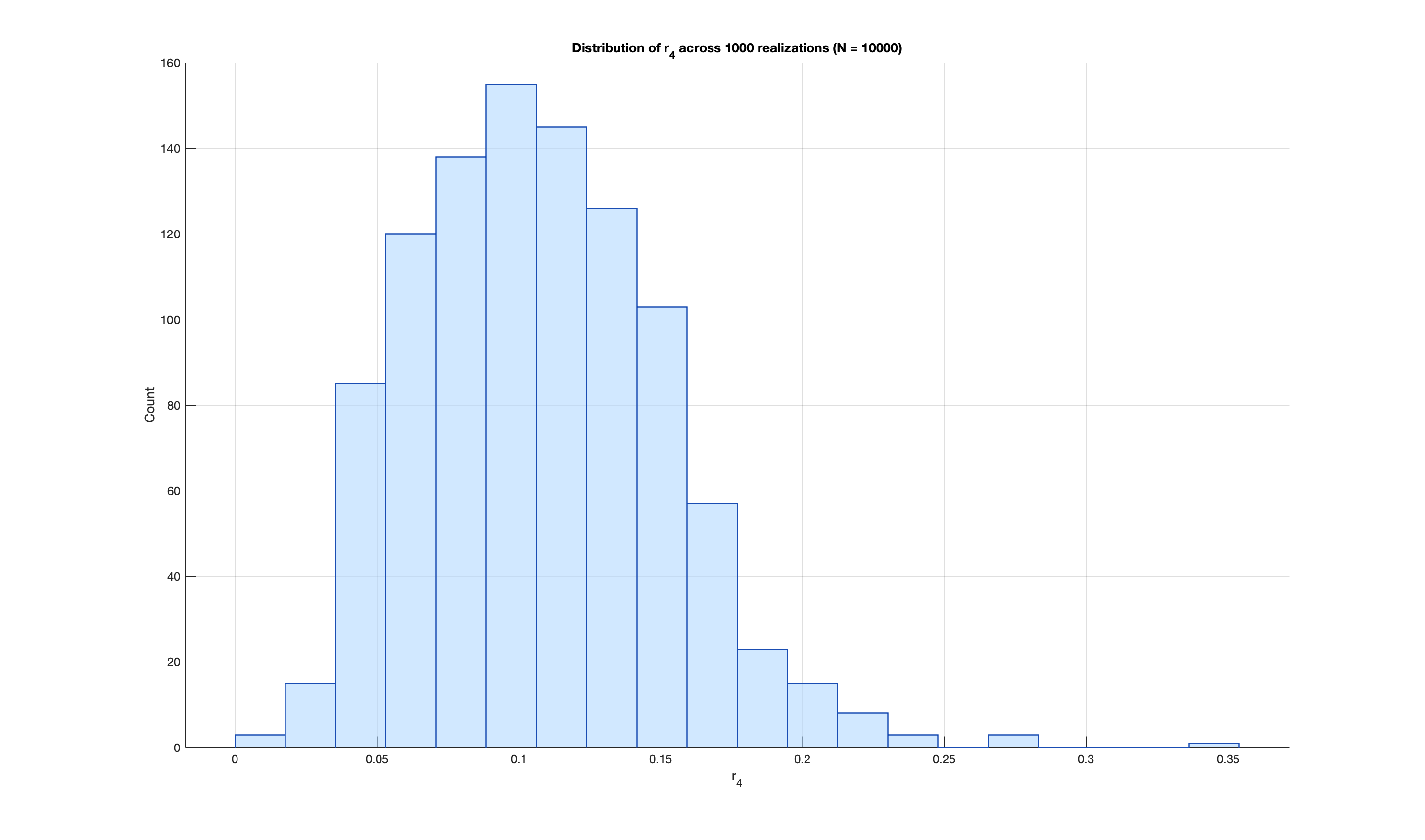}
    \includegraphics[width=0.45\linewidth]{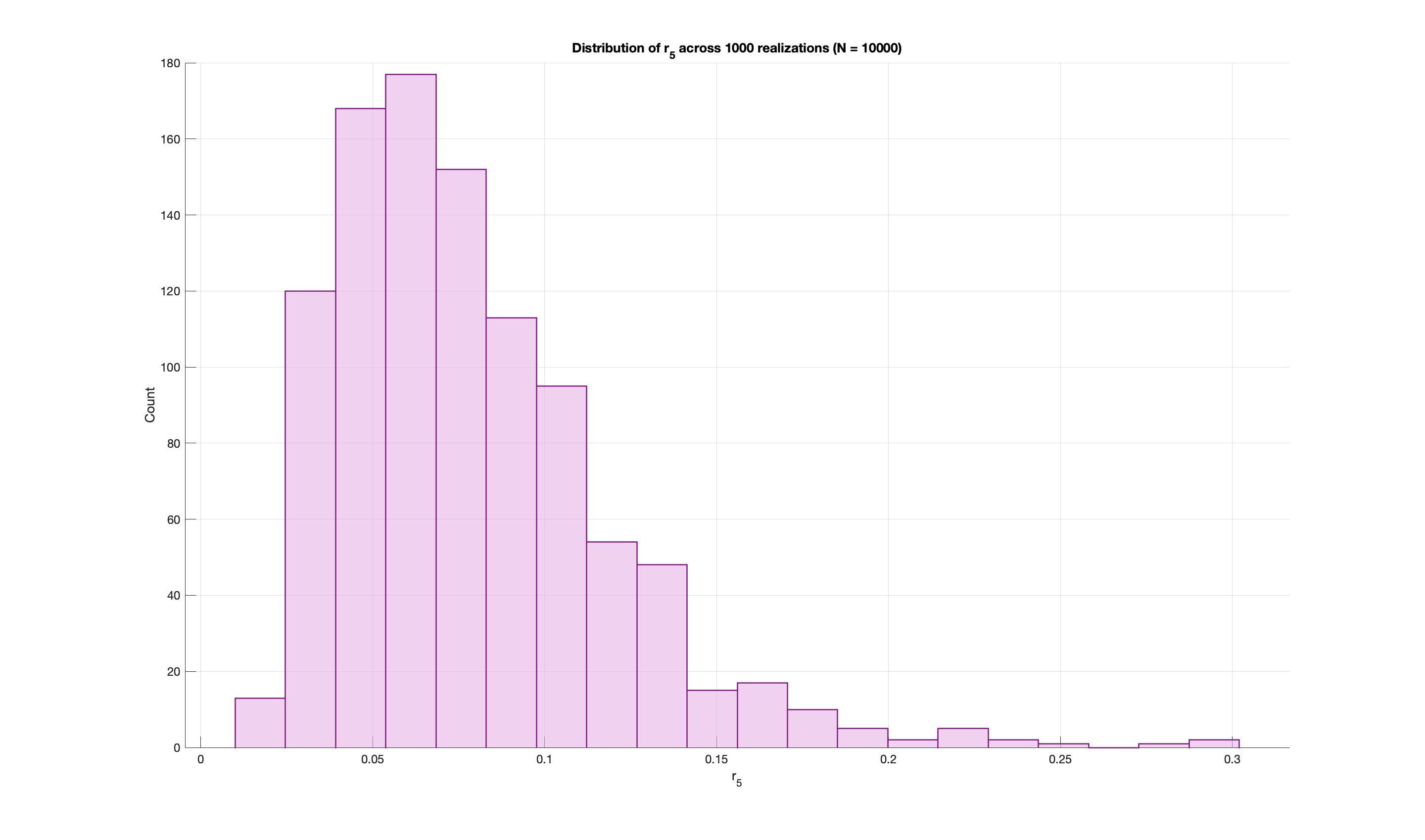}
    \caption{\textbf{Distribution of the fraction of low-degree (up to $k=5$) nodes, for $\alpha =0.5$, $n=1000$ (over a sample of $1000$ realizations).} From left to right and top to bottom, we plot the distribution of the fractions of nodes with low degree (respectively $k =0, 1, 2, 3, 4, 5$), across a sample of $1000$ different realizations for  $n=10^4$. The sample variance $\sigma_k$ is respectively $\sigma_0 = 0.09$, $\sigma_1 = 0.1$, $\sigma_2 = 0.07$, $\sigma_3 = 0.05$, $\sigma_4 = 0.04$, $\sigma_5 = 0.03$, indicating that the variance shrinks for larger values of the degree $k$.}
\label{fig:convergence_in_dist_dist}
\end{figure}

\section{Extension to stable weights}
In this section we briefly discuss how our results can be extended to the case in which the weights are sampled from a genuine $\alpha$-stable distribution $\tilde{\rho}_\alpha(w)$ with positive support and diverging mean. Throughout, we follow the formalism introduced in~\cite{Nola20}, and in particular we adopt the first parametrization defined there, corresponding to $k=0$.

For stable laws, our arguments can be adapted by replacing the exact expression for the Pareto weight density,
$\rho_{\alpha}(w)=\alpha w^{-(\alpha+1)}$ for $w\geq 1$ in~\eqref{density-Pareto}, with an asymptotic relation. Specifically, Theorem~1.2 of~\cite{Nola20} shows that the density of an $\alpha$-stable law satisfies
\begin{equation}
\label{eq:density-stable}
    \tilde{\rho}_\alpha(w)=(1+o(1))\,\frac{\alpha \gamma^{\alpha}(1+\beta)c_\alpha}{w^{\alpha+1}}, 
    \qquad w\geq 1,
\end{equation}
where the tail parameter satisfies $\alpha\in(0,1)$, $c_\alpha=\frac{\Gamma(\alpha)}{\pi}\sin\!\left(\frac{\pi\alpha}{2}\right)$, $\gamma$ denotes the scale parameter, and $\beta$ is the skewness parameter of the stable distribution.

In order to work with a distribution supported on the positive real axis, we choose $\beta=1$ and set the infimum of the support to zero. This choice fixes the location parameter to be $\delta=\gamma\tan\!\left(\frac{\pi\alpha}{2}\right)$. The only remaining free parameter is therefore the scale $\gamma$, which we determine by requiring that the tails of the stable and Pareto distributions coincide asymptotically. Imposing this matching condition yields
\begin{equation}
    \gamma = \left[\frac{\pi}{2\,\Gamma(\alpha)\sin\!\left(\frac{\pi\alpha}{2}\right)}\right]^{1/\alpha}\label{eq:tails}.
\end{equation}

With this choice of parameters, all the results we derived for Pareto-distributed weights extend directly to the case of stable sampling. In particular, since the integrals appearing in Section~\ref{app:firstMoments} are dominated by contributions from large values of $w$, the asymptotic estimates derived in Section~\ref{app:asymptotics} remain valid. Similarly, the arguments used in Section~\ref{app:rzerouno} to compute the fraction of disconnected nodes and nodes of degree one carry over almost verbatim to the stable case. The only modification is the appearance of a constant $c$ in~\eqref{conv-sum-Pareto-order-statistics}, which is now directly related to the scale parameter of the underlying stable law.

Indeed, in the convergence statement~\eqref{conv-sum-Pareto-order-statistics}, the partial sum $S_n$ has a strictly $\alpha$-stable distribution for every $n$. Interpreting $S_n$ as the value at time $1$ of a stable Lévy process, the sum $c\sum_{j\geq 1}\Gamma_j^{-1/\alpha}$ can be viewed as the collection of jump sizes of the process, ordered by decreasing magnitude. Correspondingly, the rescaled weights $n^{-1/\alpha}W_i$ represent the increments of the process over the time intervals $[(i-1)/n,i/n]$, and the variables $(y_k^{\sss(n)})_{k=1}^n$ in~\eqref{ordered-resceled-weights} can be interpreted as the ordered increments over these intervals. Since, with high probability and for large $n$, each interval contains at most one large jump, it follows that~\eqref{conv-sum-Pareto-order-statistics} continues to hold in this setting, with the first coordinate (corresponding to the sum) having exactly the same distribution as in the Pareto case.

Finally, in Figures~\ref{fig:stable_cc},~\ref{fig:pan_03_stable},~\ref{fig:pan_05_stable}, and~\ref{fig:pan_07_stable}, we provide numerical evidence confirming that the results presented in the main text remain valid when the weights are drawn from an $\alpha$-stable distribution, in agreement with the arguments given above.

\begin{figure}[h]
    \centering
    \includegraphics[width=0.32\linewidth, height= 0.2\linewidth]{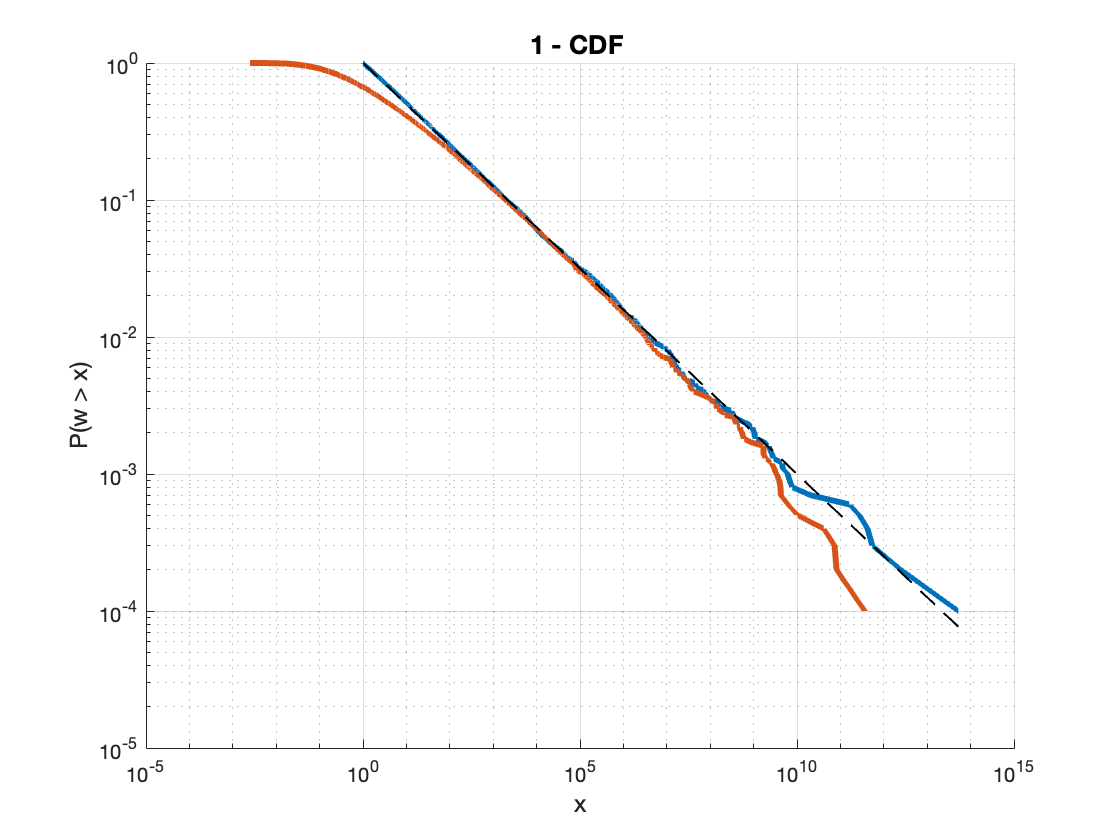}
    \includegraphics[width=0.32\linewidth, height= 0.2\linewidth]{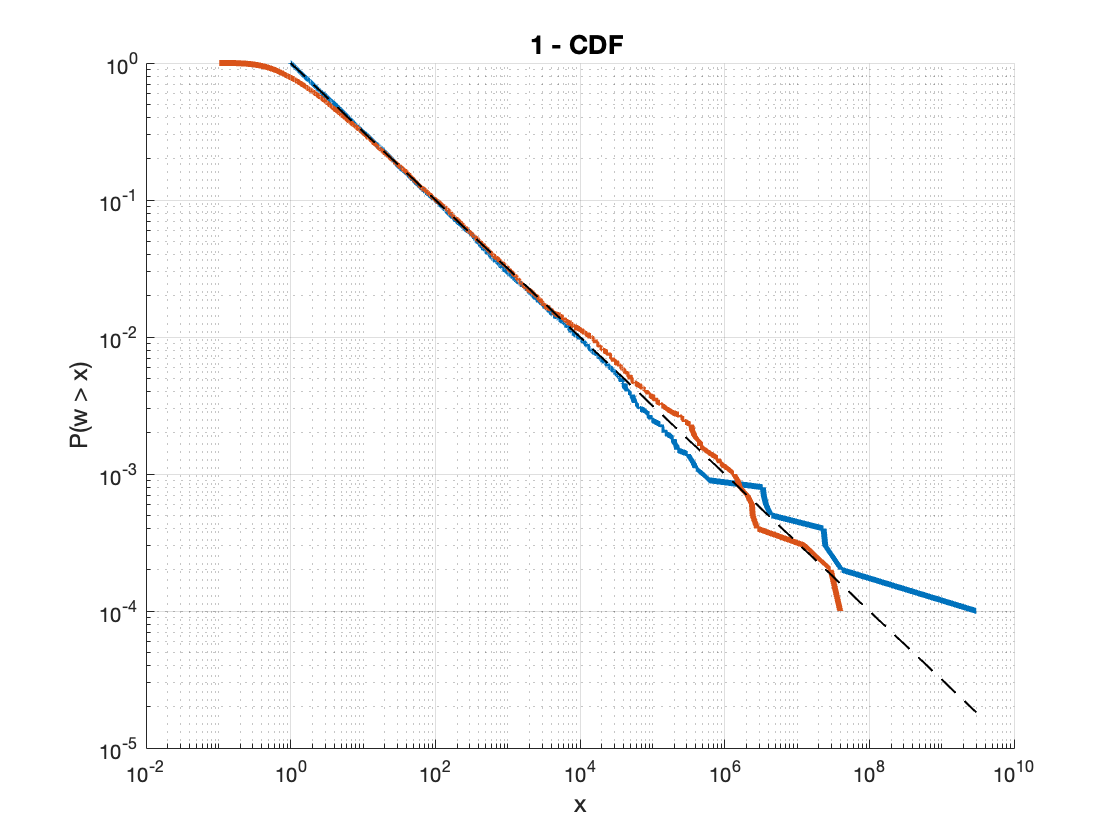}
    \includegraphics[width=0.32\linewidth, height= 0.2\linewidth]{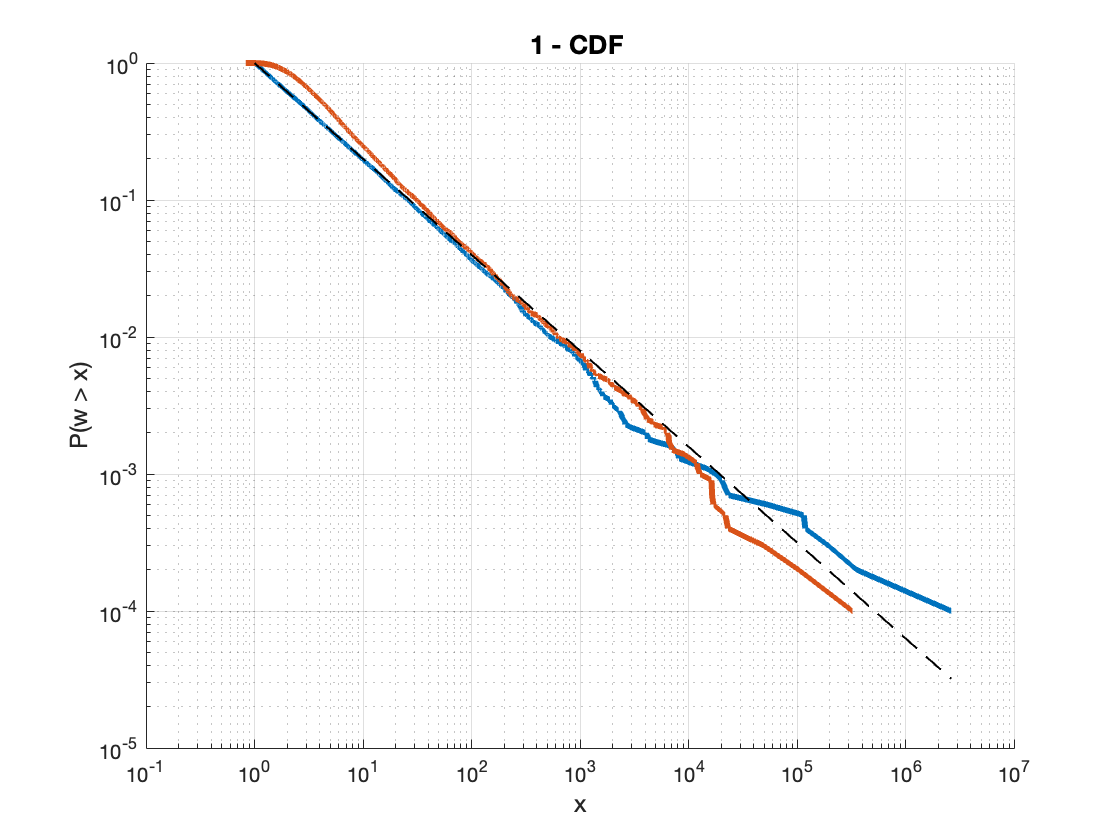}
    \caption{\textbf{Comparison between pure Pareto and Stable distributions in the tail behaviour}. Using the parameter of~\eqref{eq:tails}, we can match the tails of the two distributions. From left to right $\alpha =0.3,0.5,0.7$, sample of $10^4$.}
    \label{fig:placeholder}
\end{figure}

\begin{figure}[h]
    \includegraphics[width=0.32\linewidth, height= 0.2\linewidth]{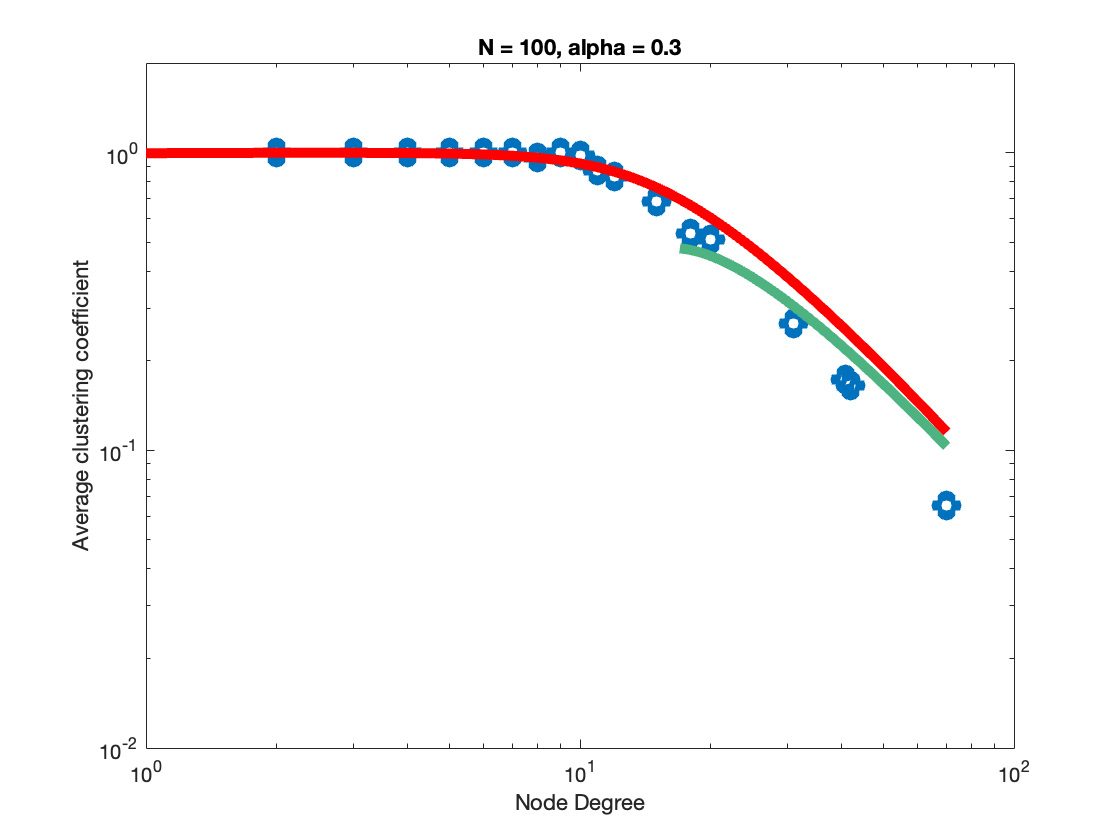} 
    \includegraphics[width=0.32\linewidth, height= 0.2\linewidth]{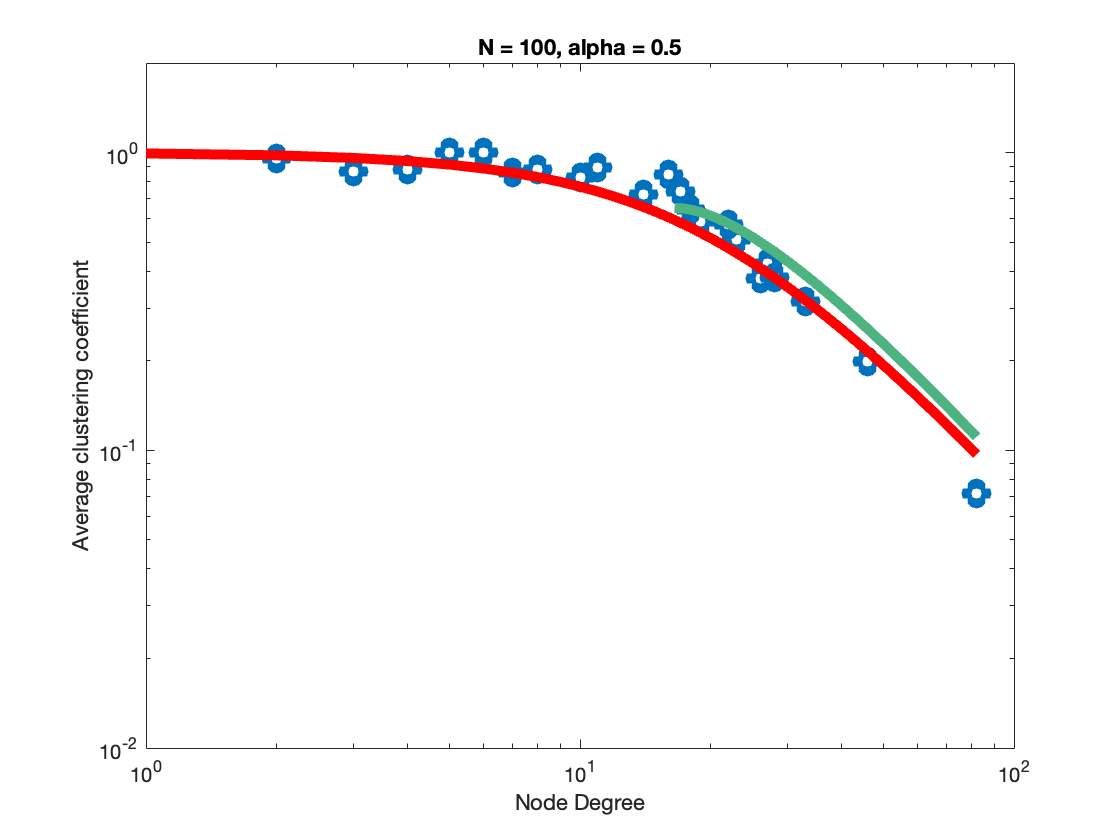}
    \includegraphics[width=0.32\linewidth, height= 0.2\linewidth]{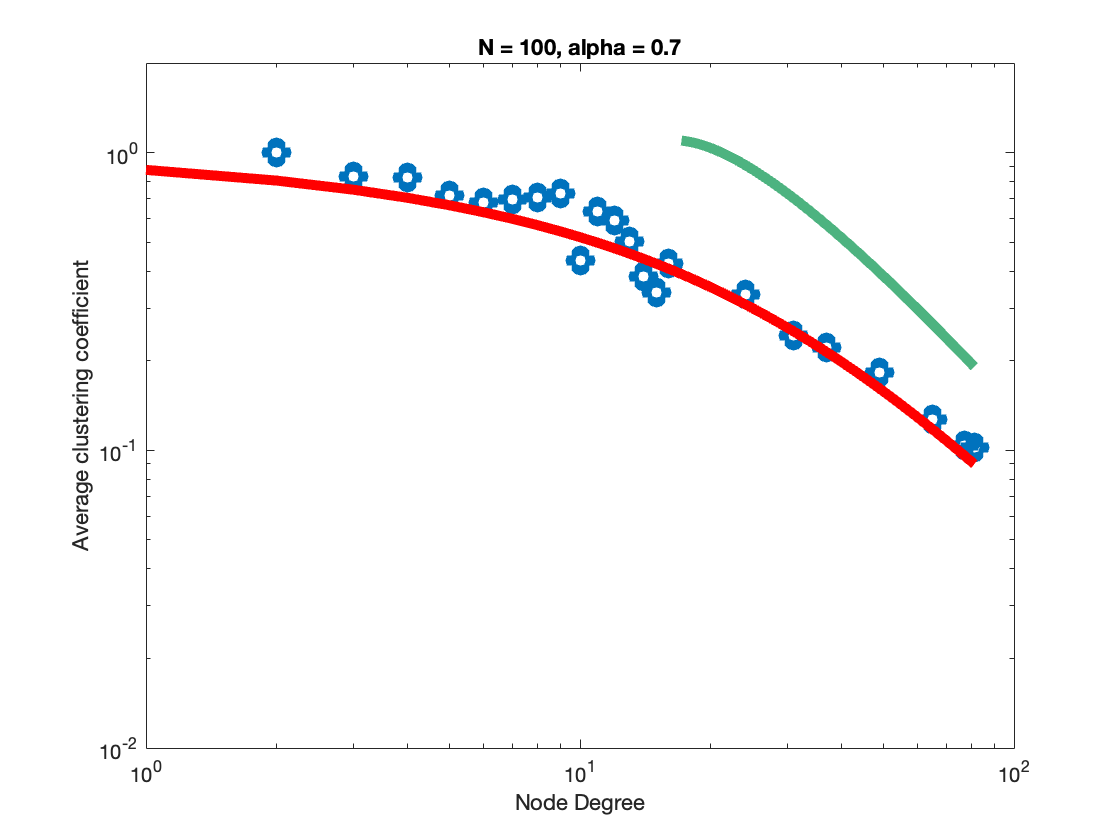}\\
    \includegraphics[width=0.32\linewidth, height= 0.2\linewidth]{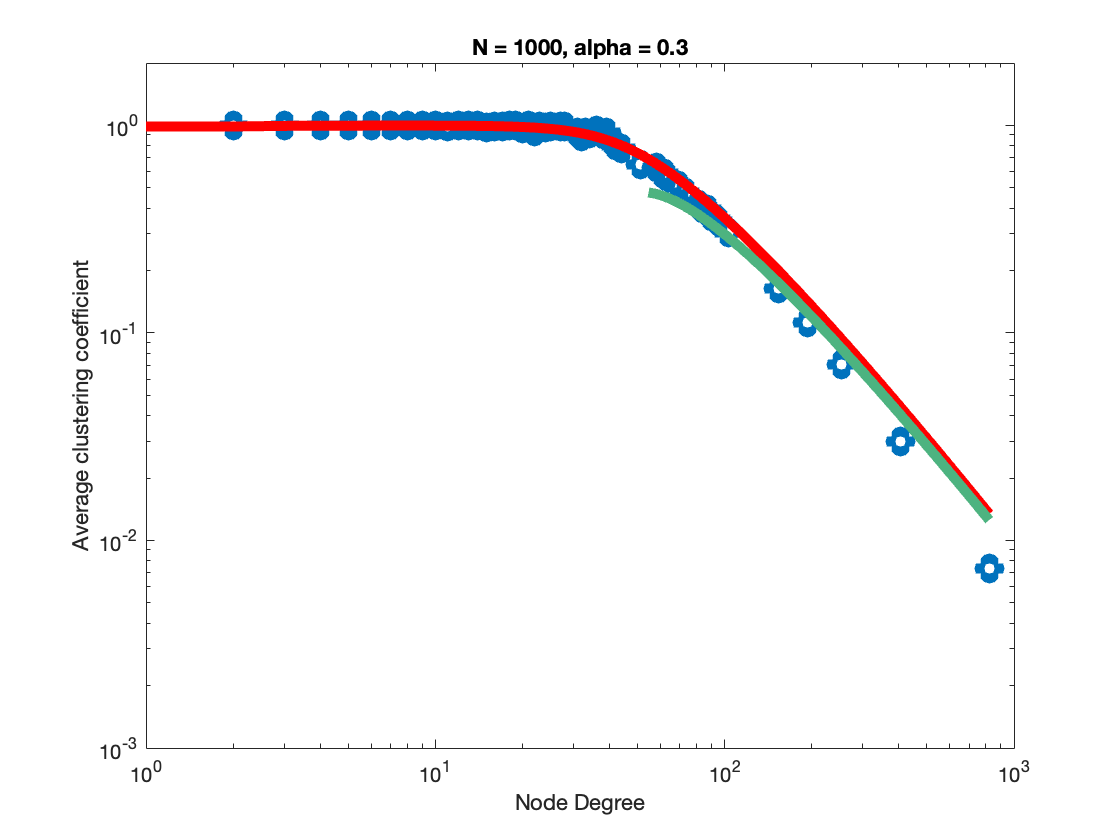} 
    \includegraphics[width=0.32\linewidth, height= 0.2\linewidth]{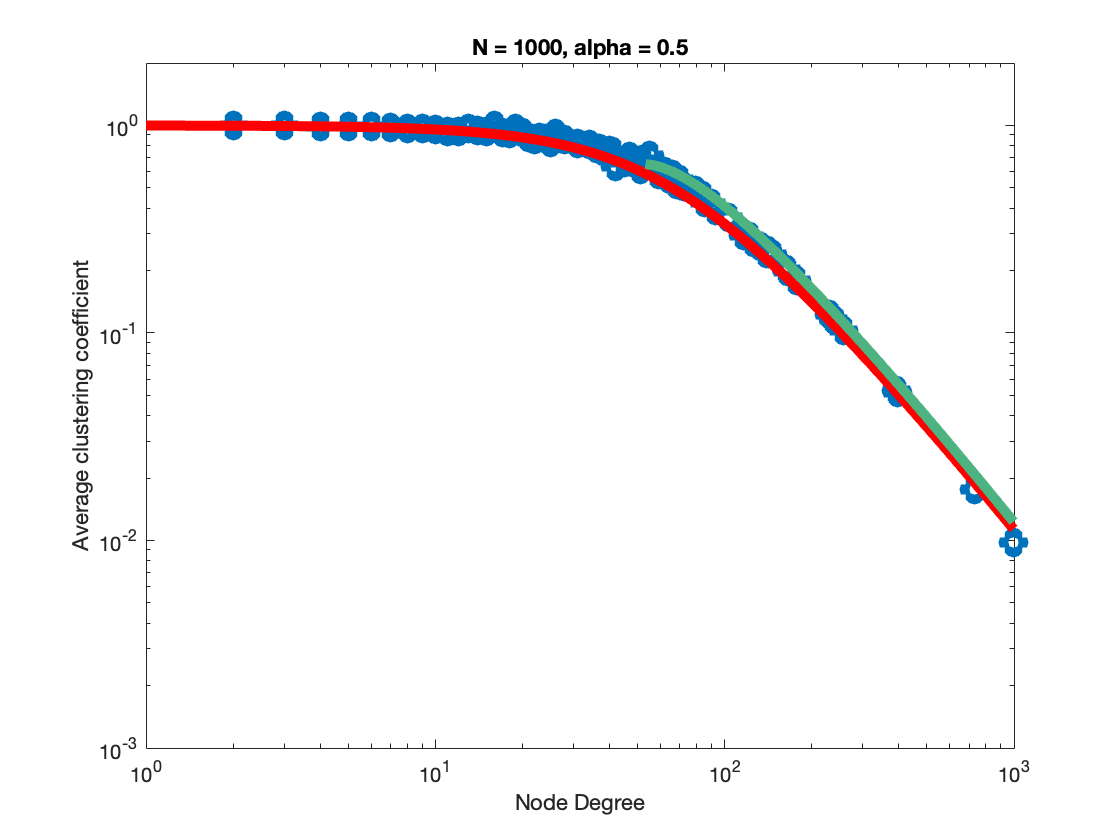}
    \includegraphics[width=0.32\linewidth, height= 0.2\linewidth]{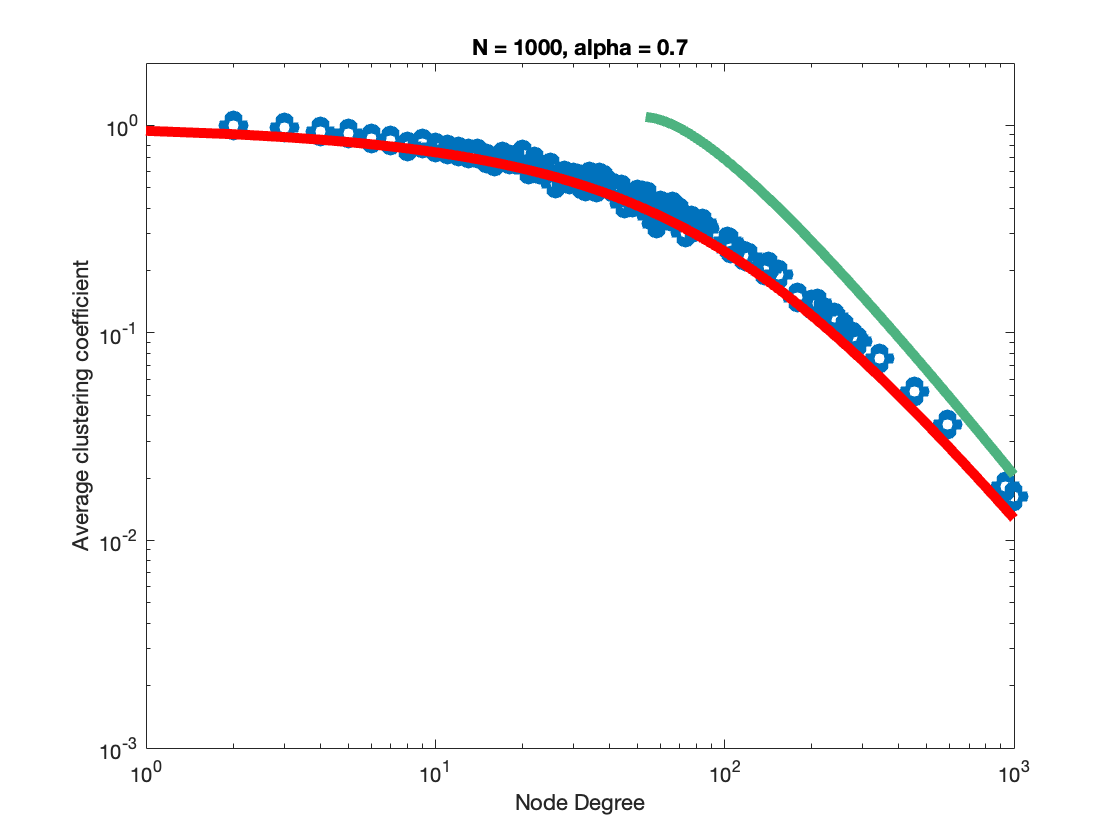}\\
    \includegraphics[width=0.32\linewidth, height= 0.2\linewidth]{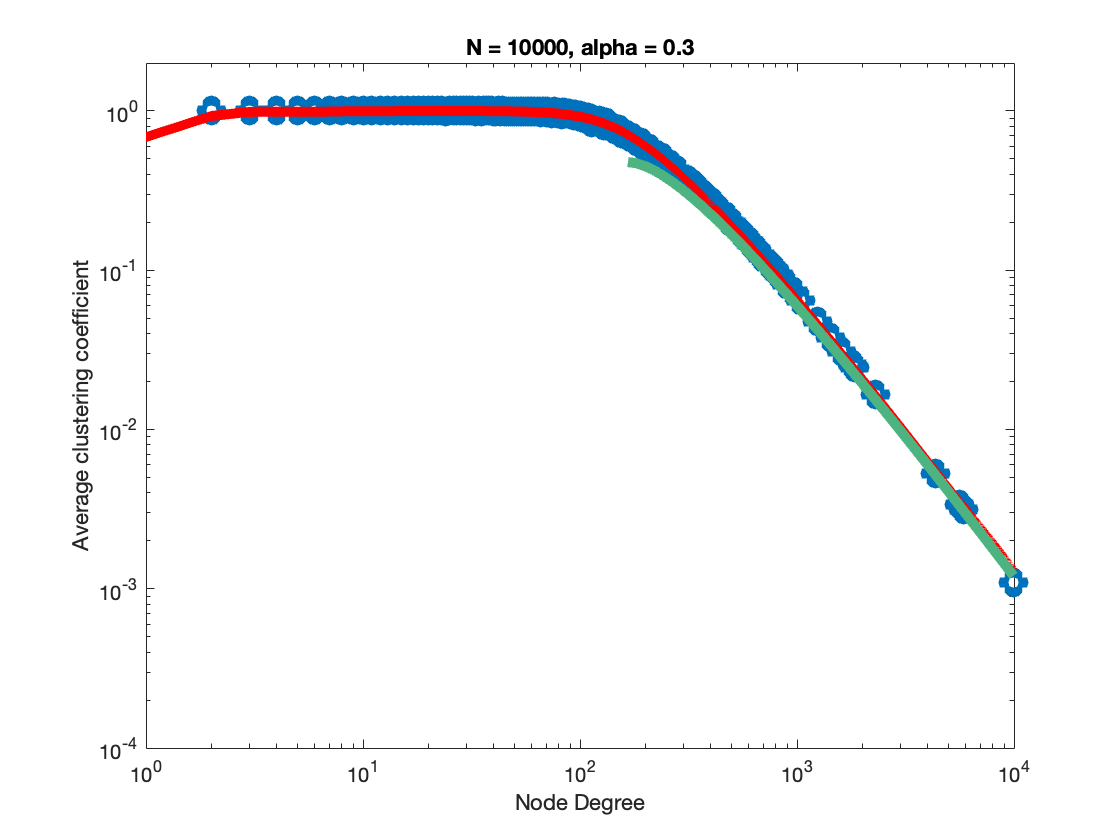} \includegraphics[width=0.32\linewidth, height= 0.2\linewidth]{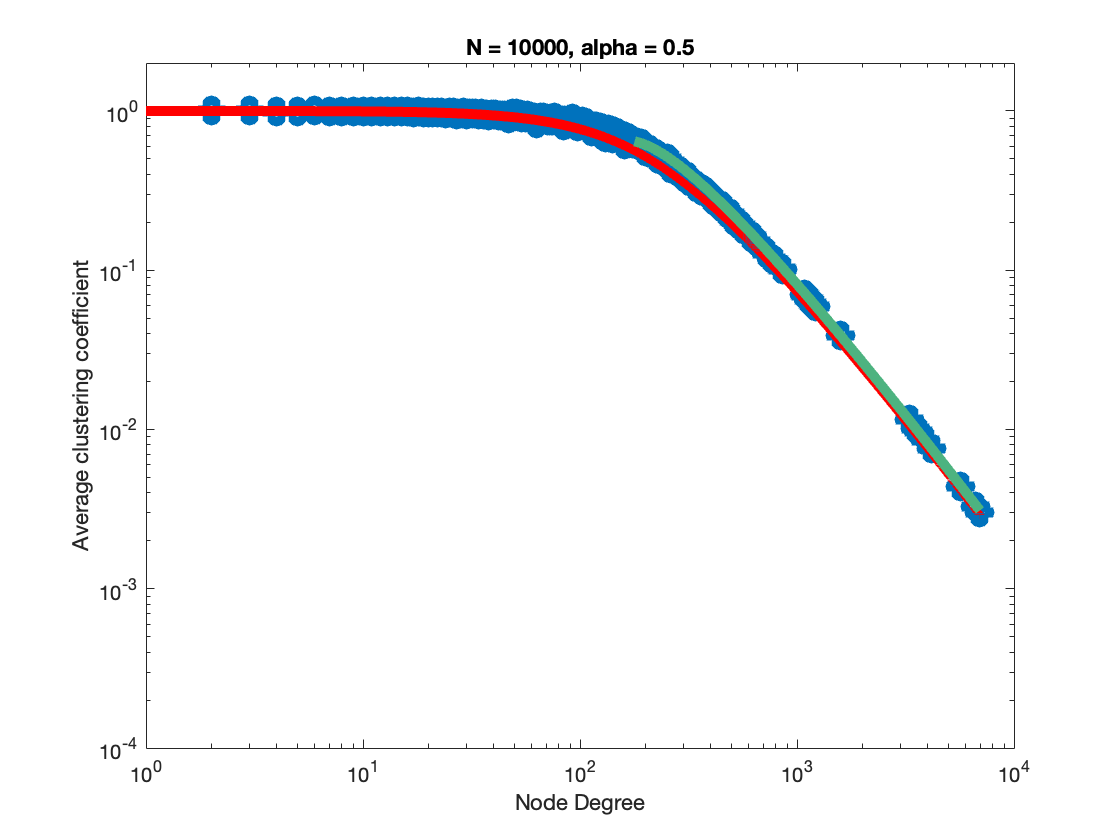}
    \includegraphics[width=0.32\linewidth, height= 0.2\linewidth]{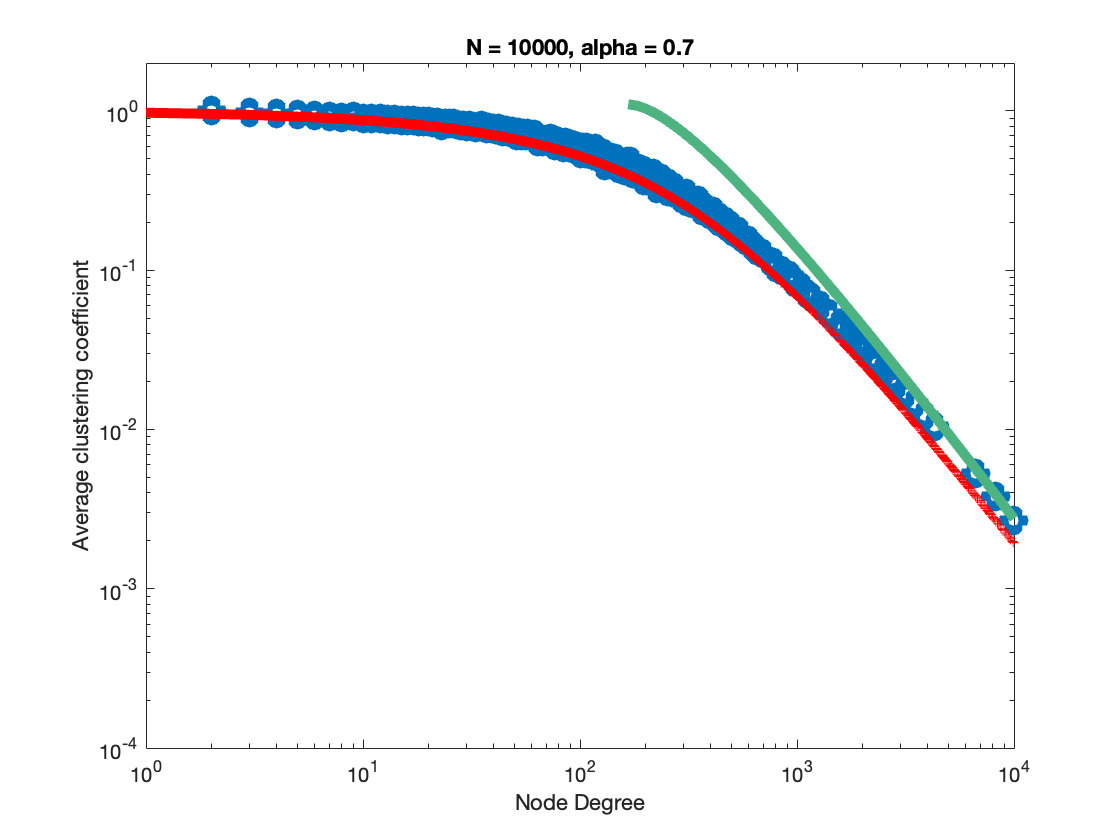}
    \caption{{\bf Clustering functions for different values of $n$ and $\alpha$ for the MSM~\eqref{eq:pizza} with Stable weights~\eqref{eq:density-stable}.} Blue circles: empirical clustering function~\eqref{eq:cc_empirical} versus the degree $k$ computed on actual realized graphs (obtained by sampling the weights once, and sampling the graph once conditionally on the realized weights). Red curves: our analytical expression~\eqref{eq:analytical_clustering} for the annealed clustering function evaluated on the actual degree $k=a \sqrt{n}$. 
    Green curves: our asymptotic calculation~\eqref{eq:CC_hub} valid for diverging reduced degrees. We evaluate and plot all functions only for degrees larger than $1$, to avoid ambiguities in the definition of the clustering coefficient for $k<2$.
    From top to bottom: $n=10^2,10^3,10^4$. 
    From left to right: $\alpha=0.3,0.5,0.7$.}  \label{fig:stable_cc}
\end{figure}

\begin{figure}[h!]
    \centering
    \includegraphics[width=0.49\linewidth, height= 0.32\linewidth]{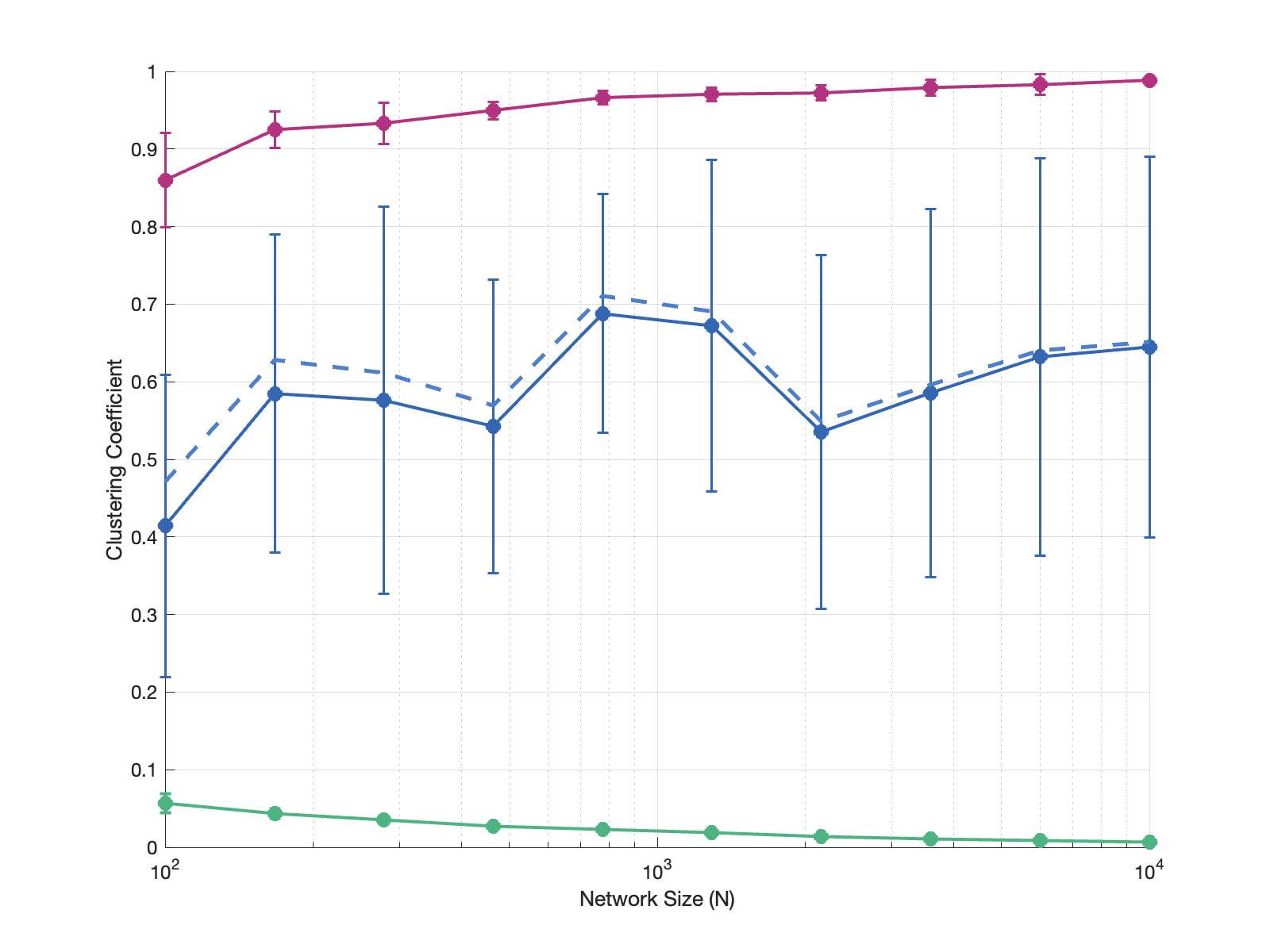}
    \includegraphics[width=0.49\linewidth, height= 0.32\linewidth]{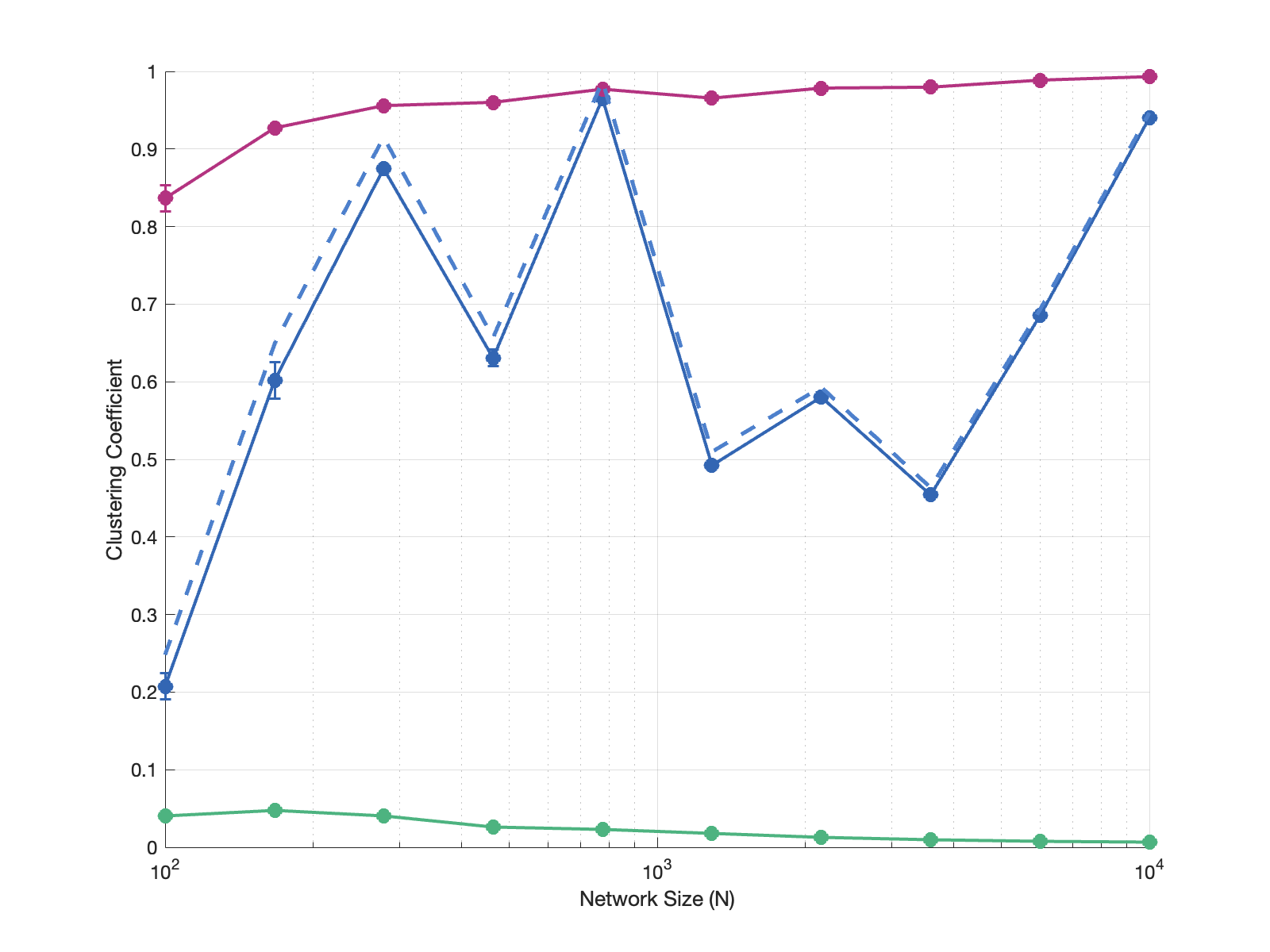}\\
    \includegraphics[width=0.49\linewidth, height= 0.32\linewidth]{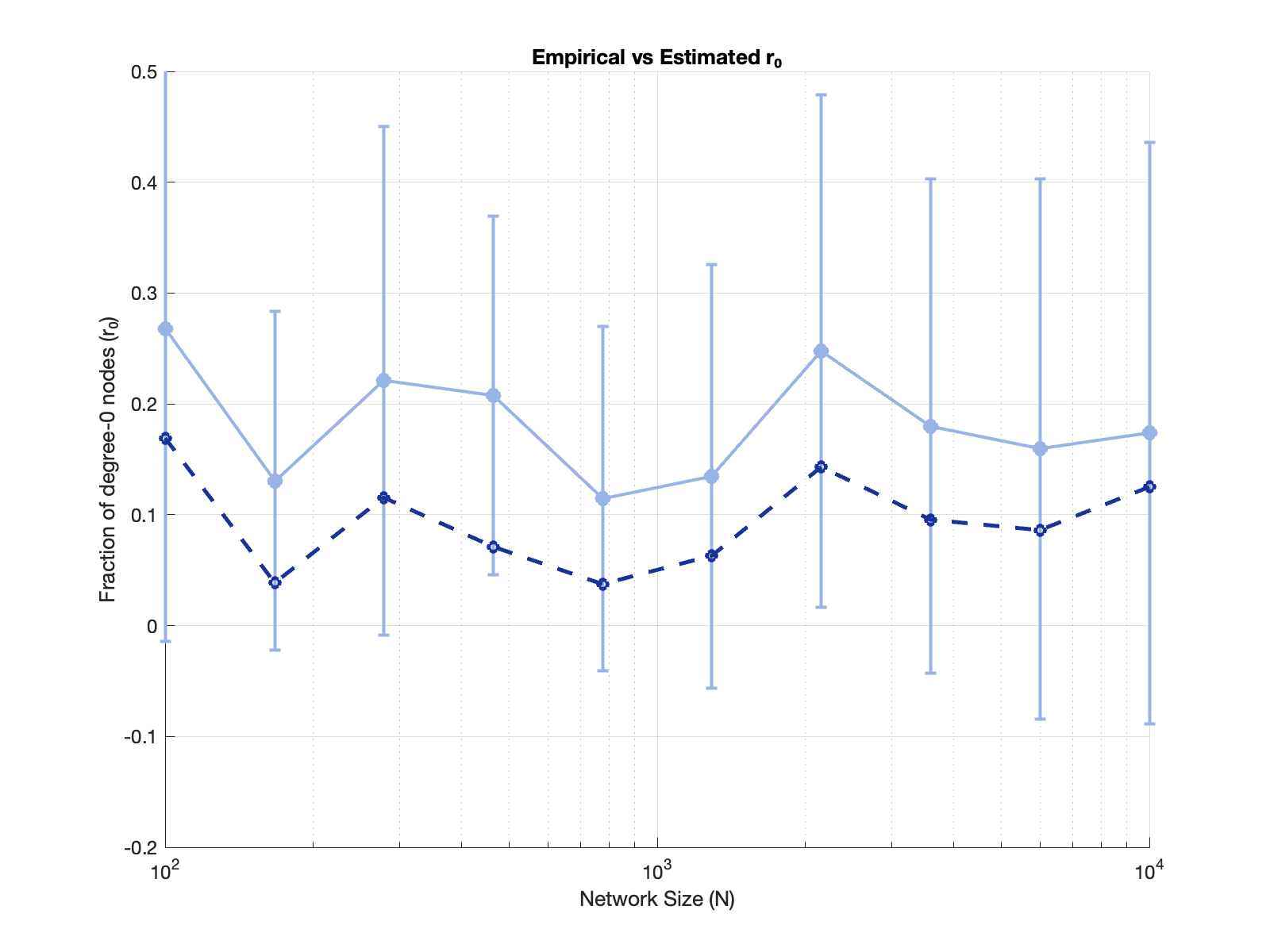}
    \includegraphics[width=0.49\linewidth, height= 0.32\linewidth]{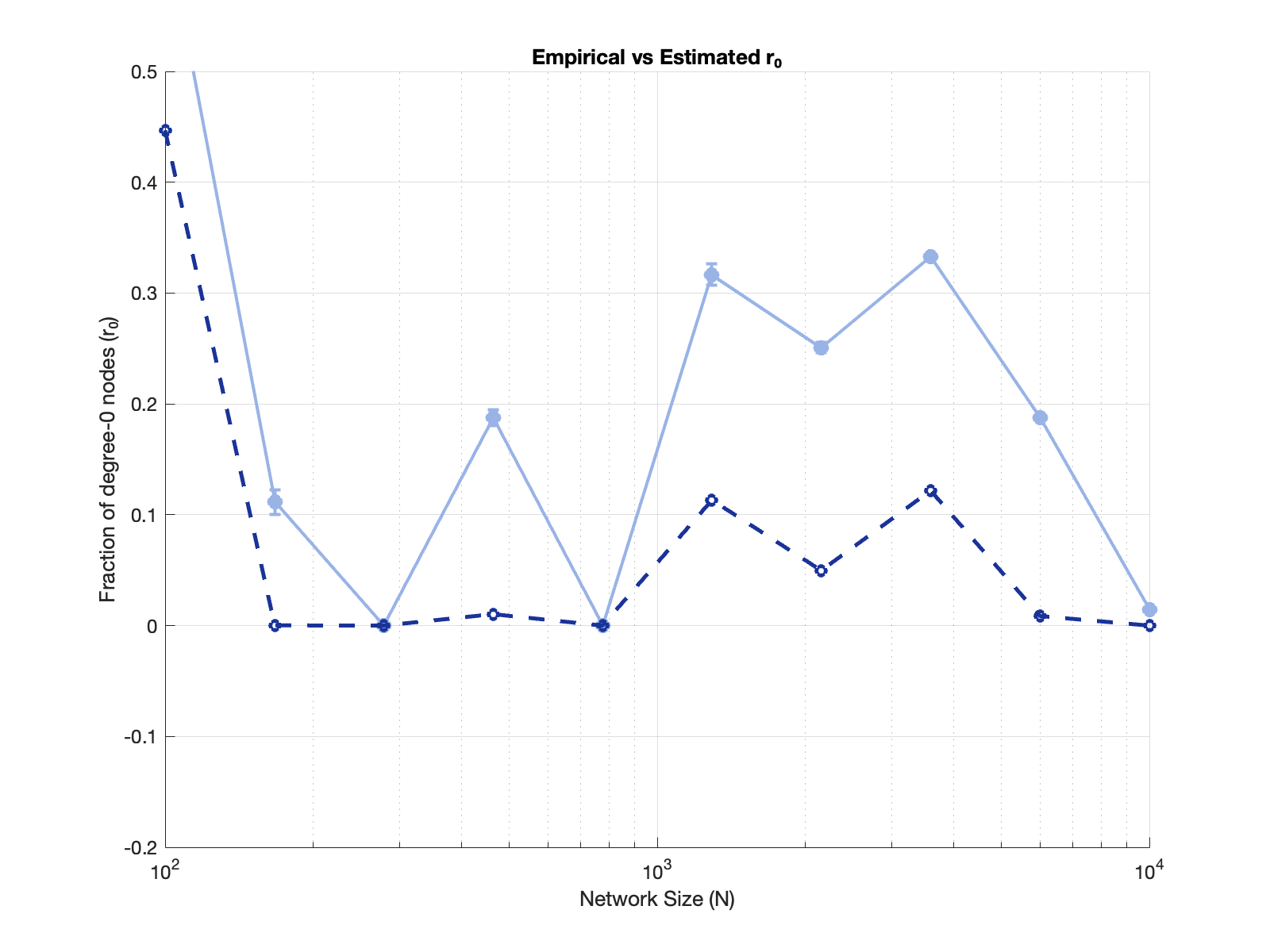}\\
    \includegraphics[width=0.49\linewidth, height= 0.32\linewidth]{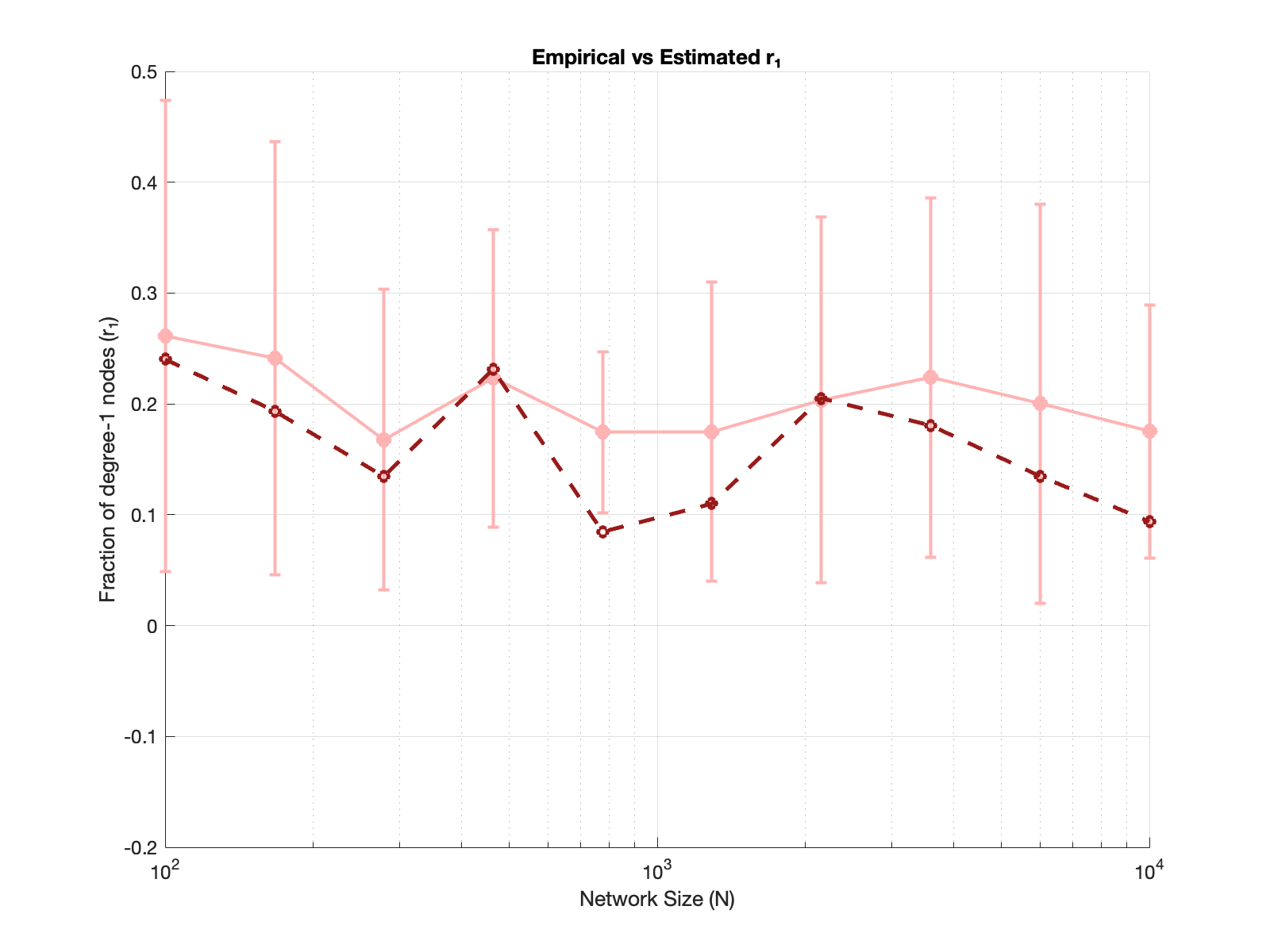}
    \includegraphics[width=0.49\linewidth, height= 0.32\linewidth]{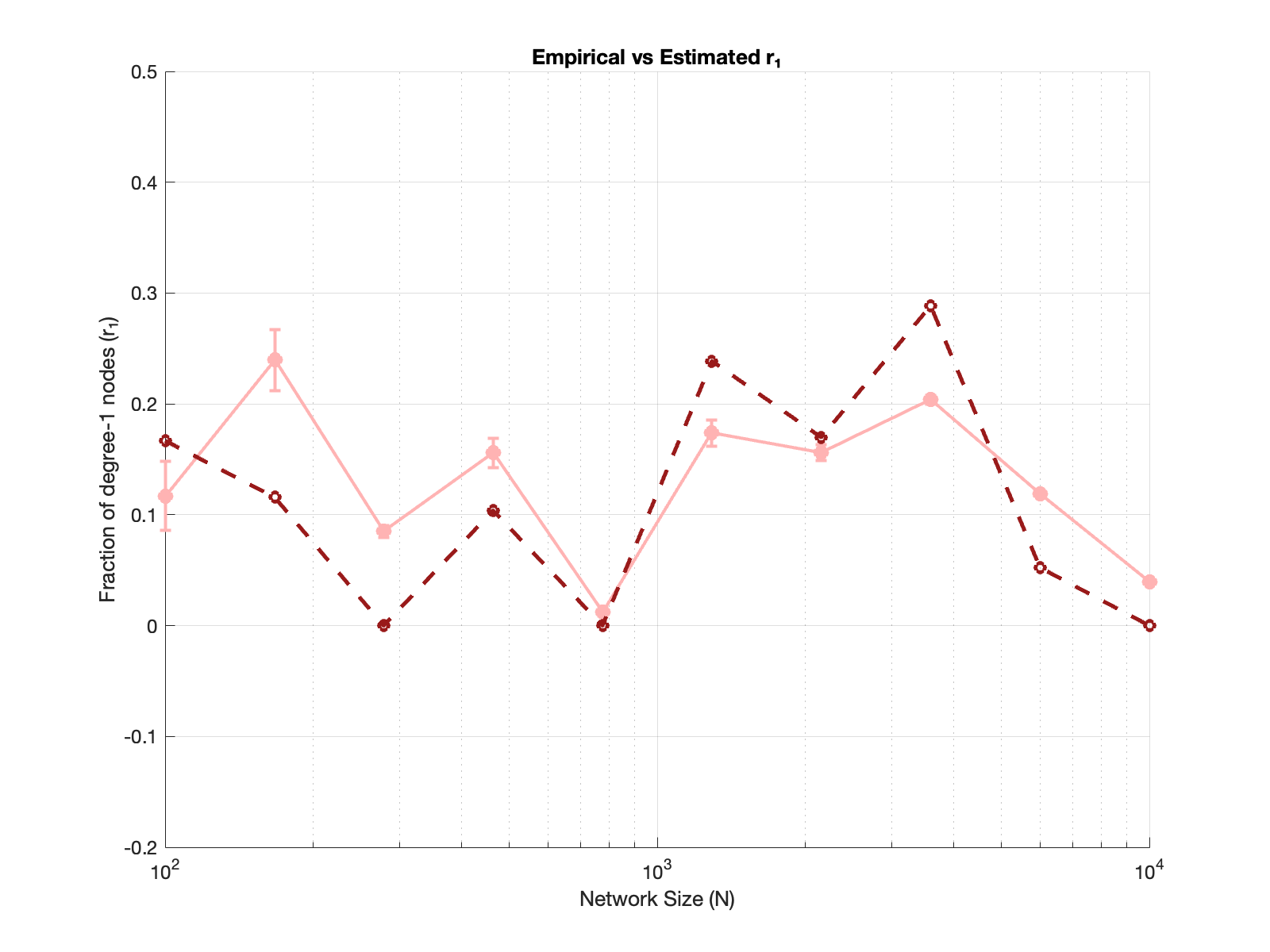}
    \caption{\textbf{Average local clustering coefficient $C$ and distance to $r_{0/1}$ versus network size $n$, with weights drawn from a Stable distribution (for $\alpha=0.3$).} 
    On the left, simulations are done by resampling weights every time an actual network is realized (sample of $10$), while on the left for each $n$ weights are extracted only once, and $10$ different adjacency matrices are realized from the same weights. On top
    we show the node-averaged local clustering coefficient ${C}$, both including (blue symbols) and excluding (purple symbols) nodes with degree $k=0,1$ (note that the latter evolves smoothly towards 1 with shrinking error bars as $n$ increases, while the former fluctuates with non-vanishing error bars, as a result of non-self-averaging).    
    The dashed blue line is $1$ minus the average of $r_{0/1}$ over realizations. Finally, in green the difference between $1-r_{0/1}$ and ${C}$ computed including nodes with $k<2$ (notice the shrinking error bars). In the middle, the light blue line with error bars is the actual value of $r_0$, while the dashed darker one is the approximation we derive in~\eqref{r0-approximation}. At the bottom, the light red line with error bars is the actual value of $r_1$, while the dashed darker one is the approximation we derive in~\eqref{r1-approximation}. }
    \label{fig:pan_03_stable}
\end{figure}

\begin{figure}[h!]
    \centering
    \includegraphics[width=0.49\linewidth, height= 0.32\linewidth]{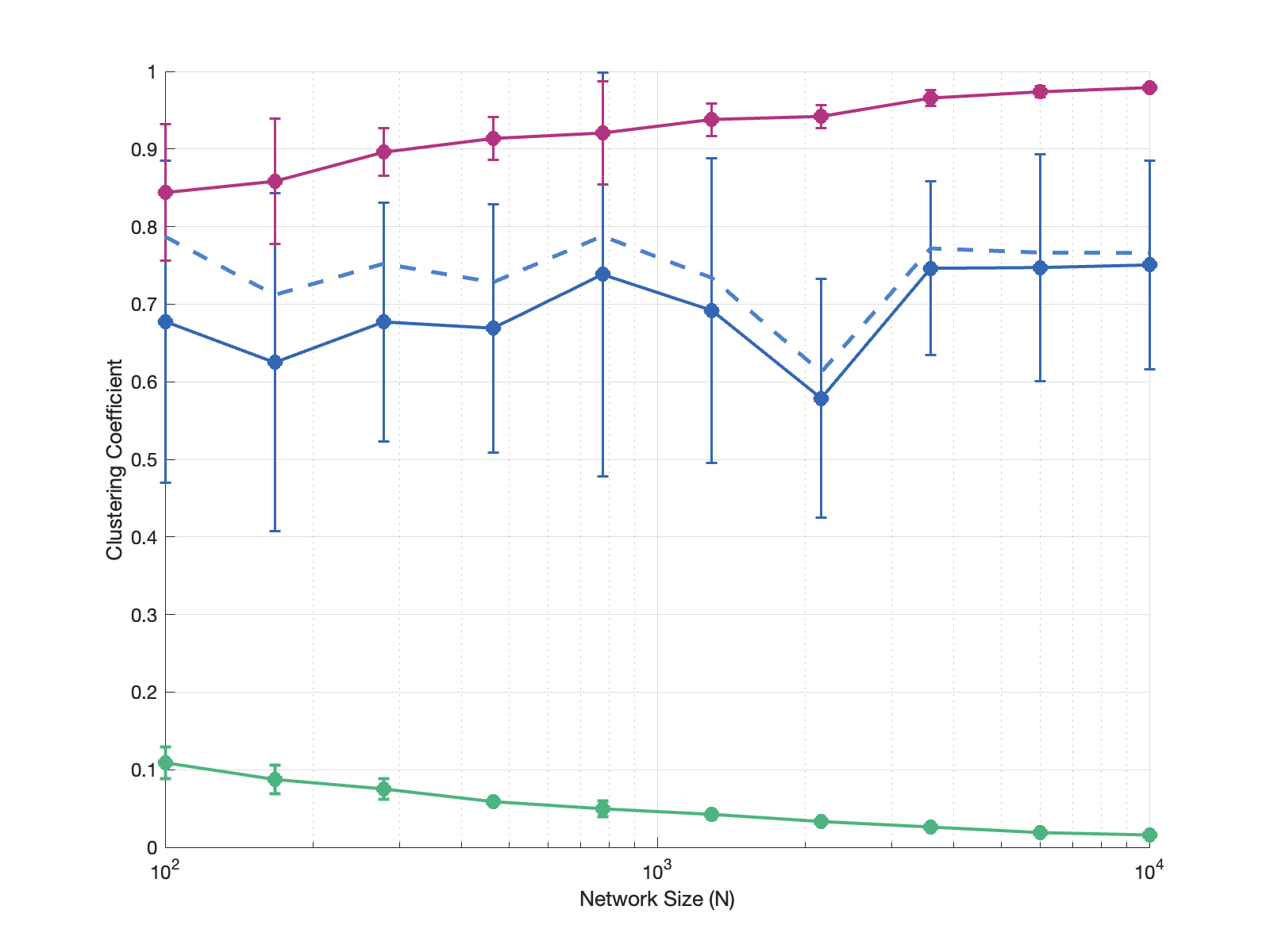}
    \includegraphics[width=0.49\linewidth, height= 0.32\linewidth]{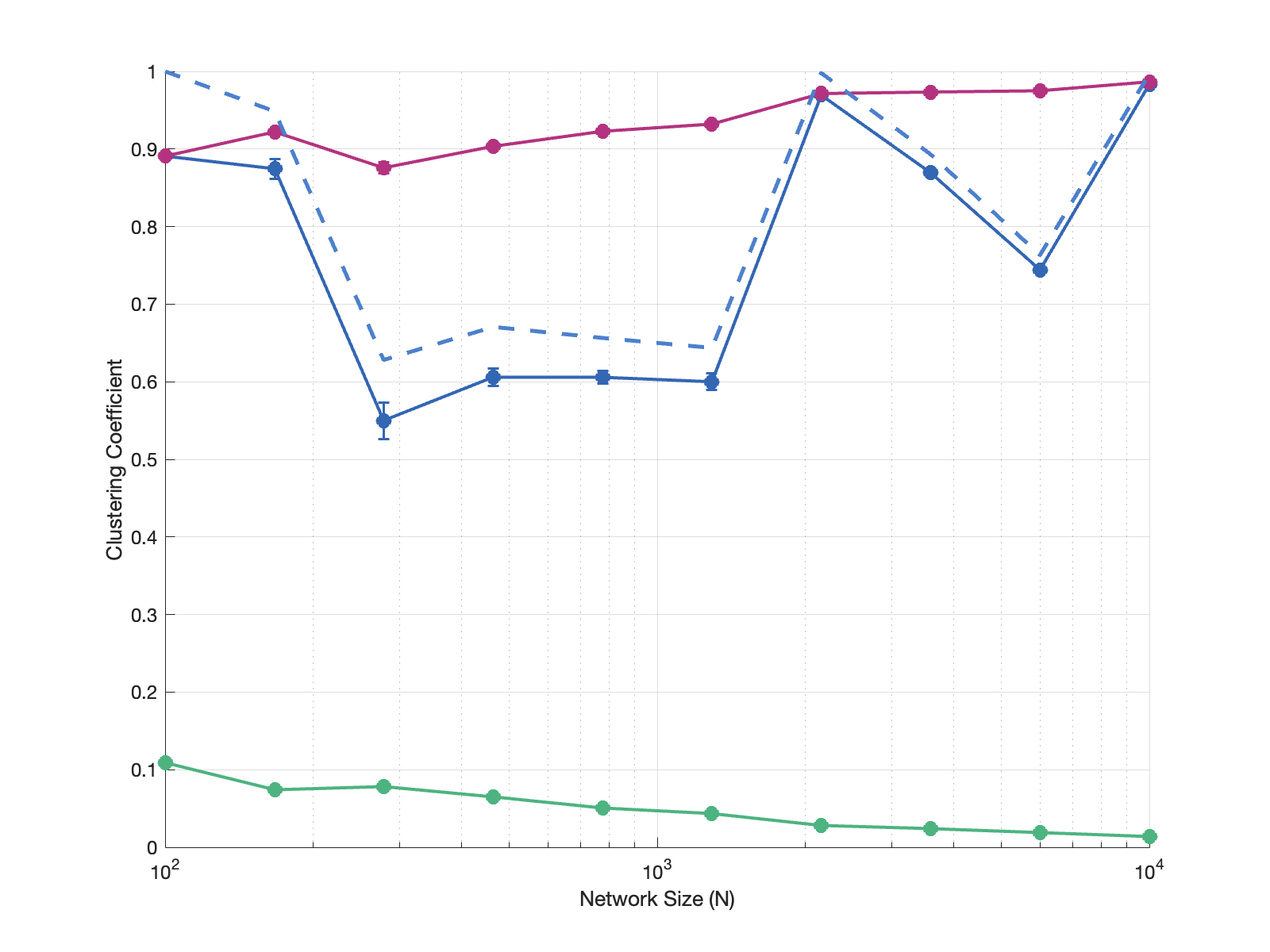}\\
    \includegraphics[width=0.49\linewidth, height= 0.32\linewidth]{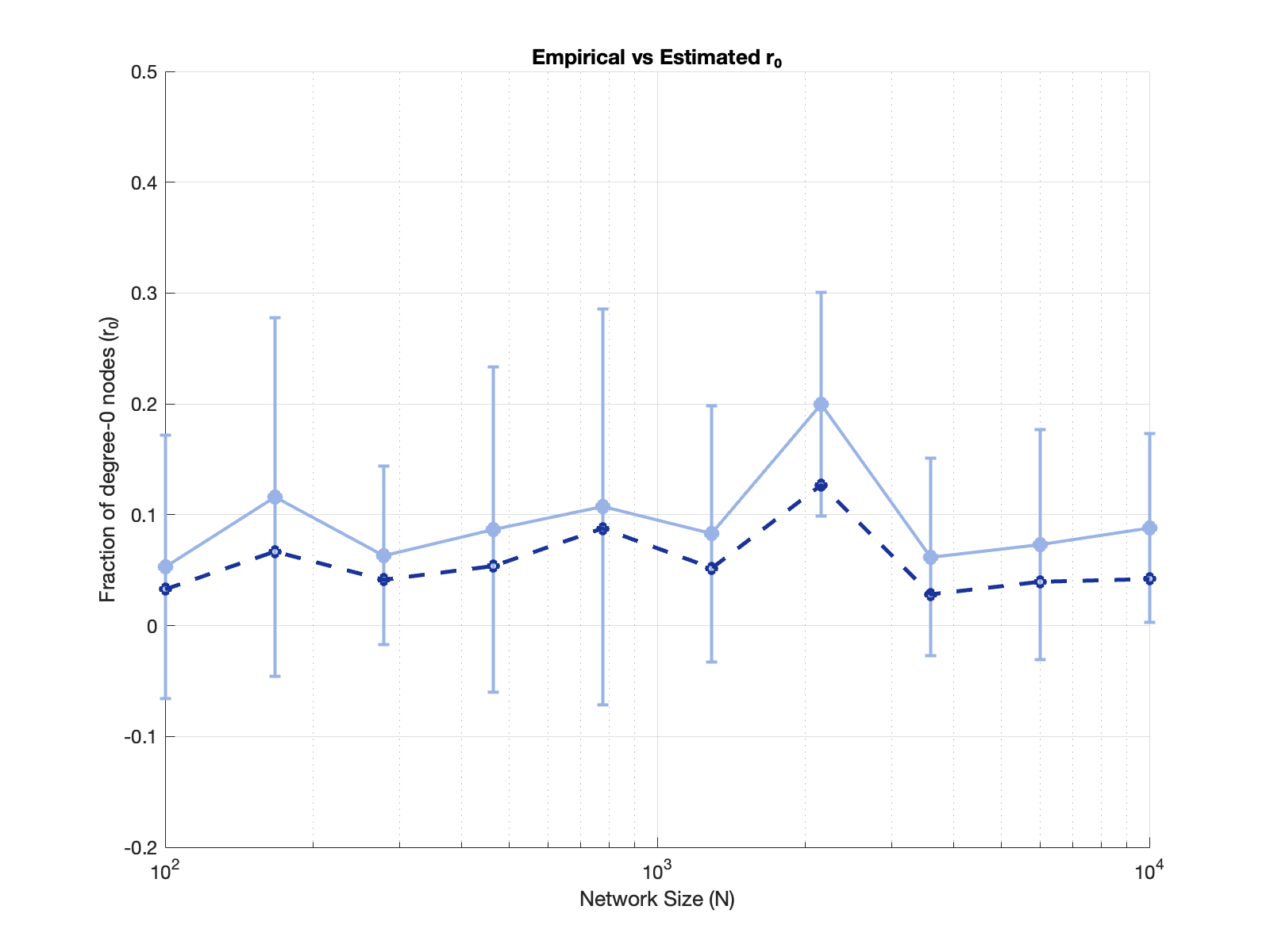}
    \includegraphics[width=0.49\linewidth, height= 0.32\linewidth]{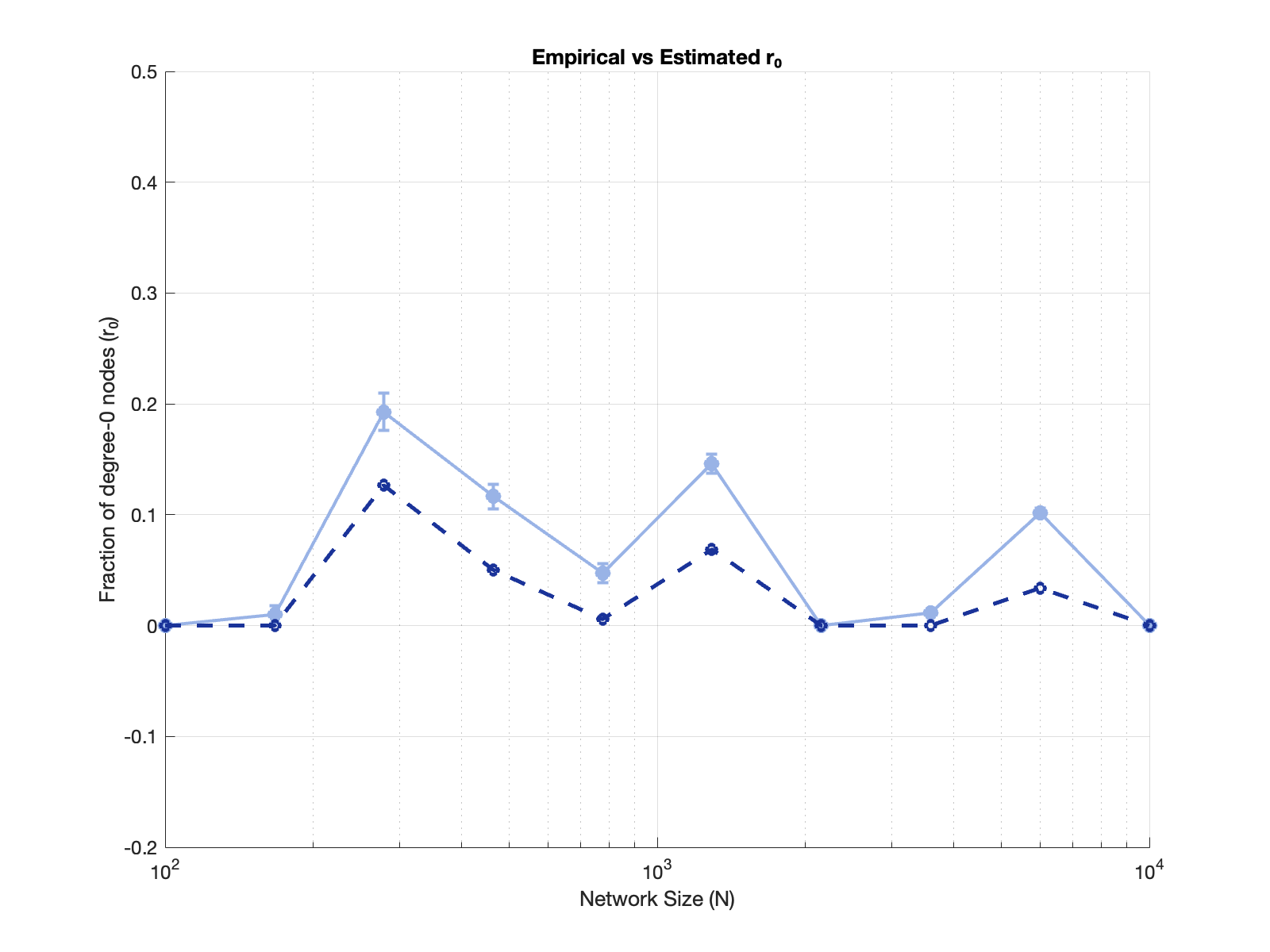}\\
    \includegraphics[width=0.49\linewidth, height= 0.32\linewidth]{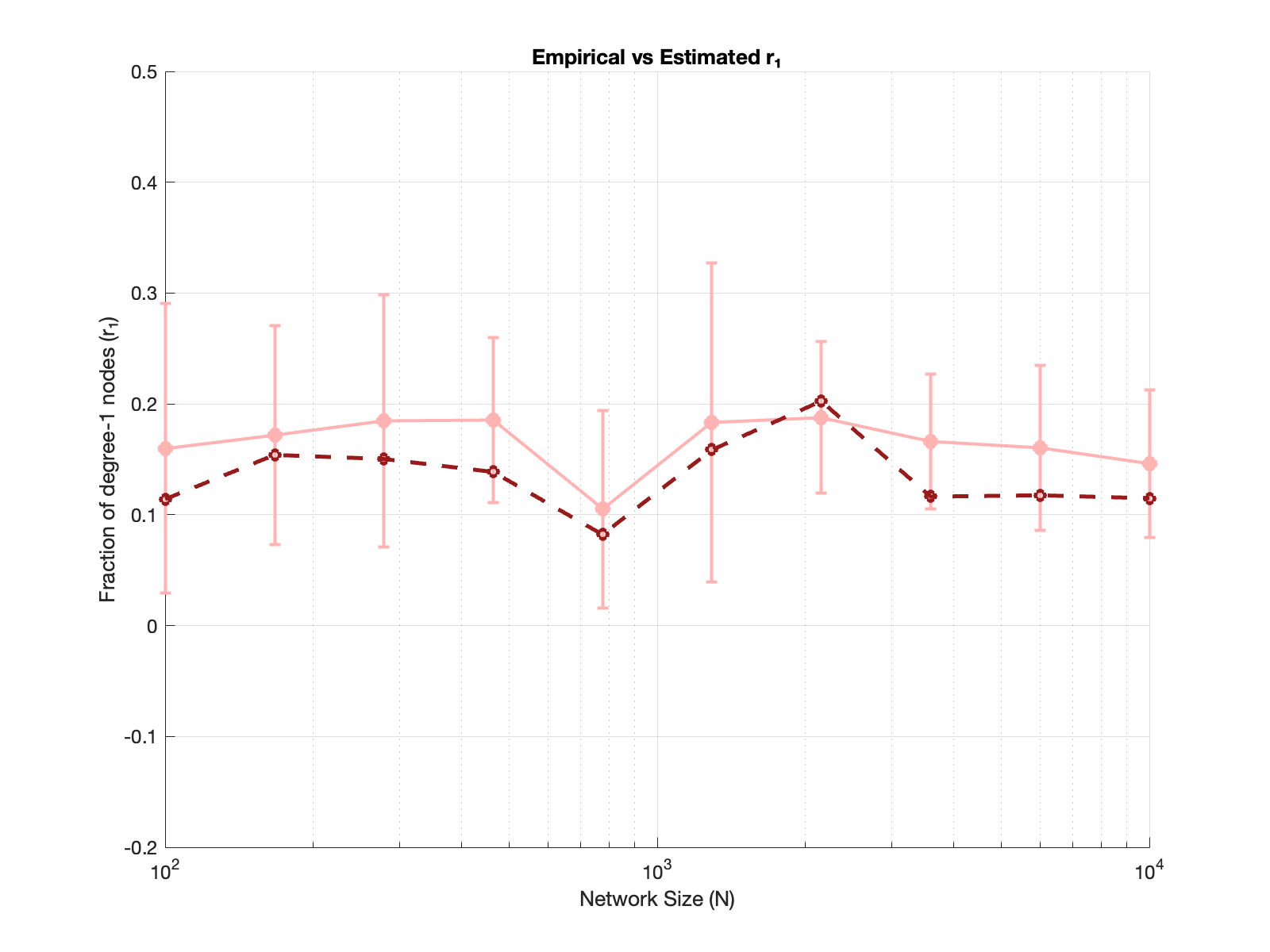}
    \includegraphics[width=0.49\linewidth, height= 0.32\linewidth]{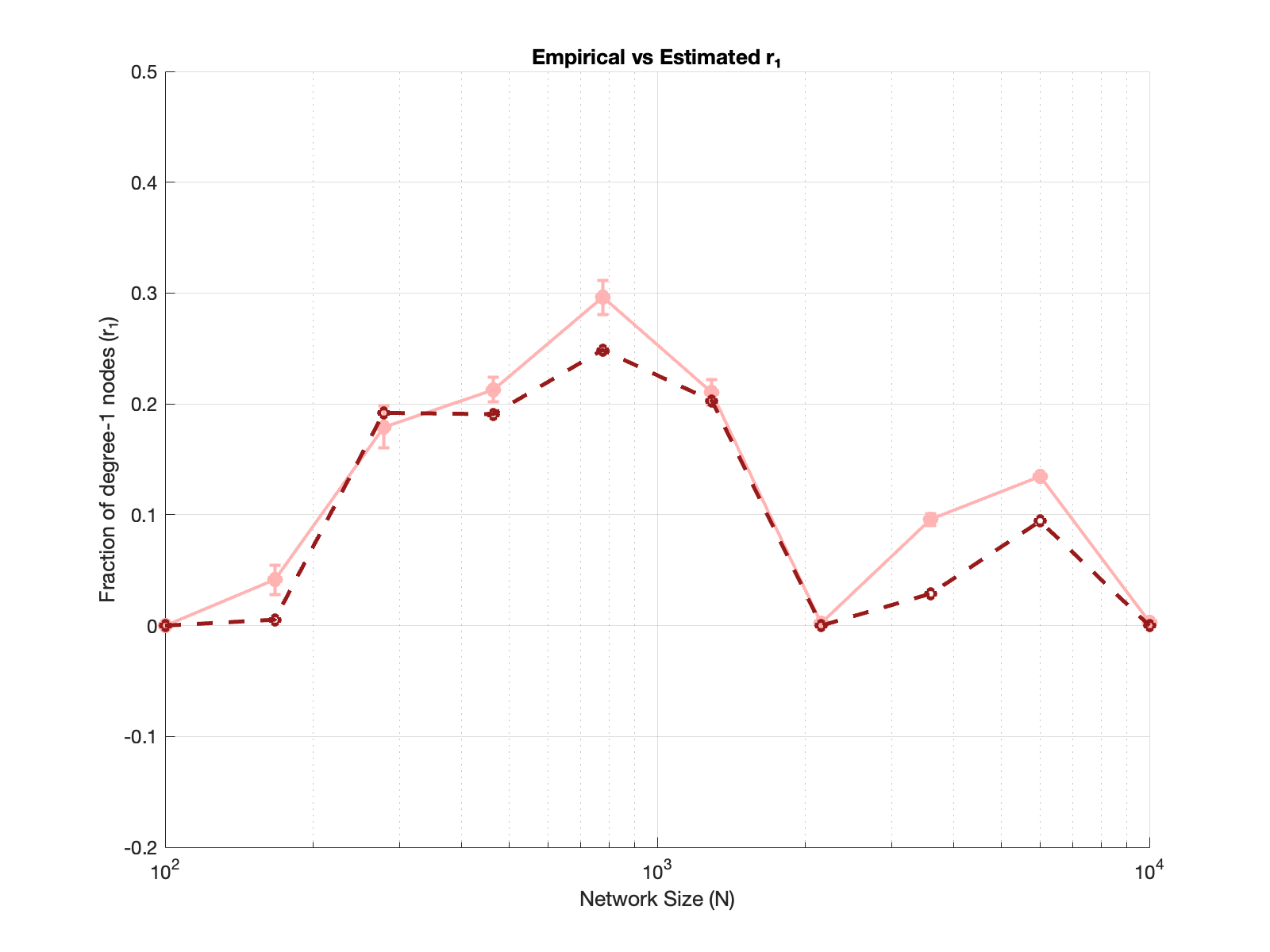}
    \caption{\textbf{Average local clustering coefficient $C$ and distance to $r_{0/1}$ versus network size $n$, with weights drawn from a Stable distribution (for $\alpha=0.5$).} 
    On the left, simulations are done by resampling weights every time an actual network is realized (sample of $10$), while on the left for each $n$ weights are extracted only once, and $10$ different adjacency matrices are realized from the same weights. On top
    we show the node-averaged local clustering coefficient ${C}$, both including (blue symbols) and excluding (purple symbols) nodes with degree $k=0,1$ (note that the latter evolves smoothly towards 1 with shrinking error bars as $n$ increases, while the former fluctuates with non-vanishing error bars, as a result of non-self-averaging).    
    The dashed blue line is $1$ minus the average of $r_{0/1}$ over realizations. Finally, in green the difference between $1-r_{0/1}$ and ${C}$ computed including nodes with $k<2$ (notice the shrinking error bars). In the middle, the light blue line with error bars is the actual value of $r_0$, while the dashed darker one is the approximation we derive in~\eqref{r0-approximation}. At the bottom, the light red line with error bars is the actual value of $r_1$, while the dashed darker one is the approximation we derive in~\eqref{r1-approximation}. }
    \label{fig:pan_05_stable}
\end{figure}

\begin{figure}[h!]
    \centering
    \includegraphics[width=0.49\linewidth, height= 0.32\linewidth]{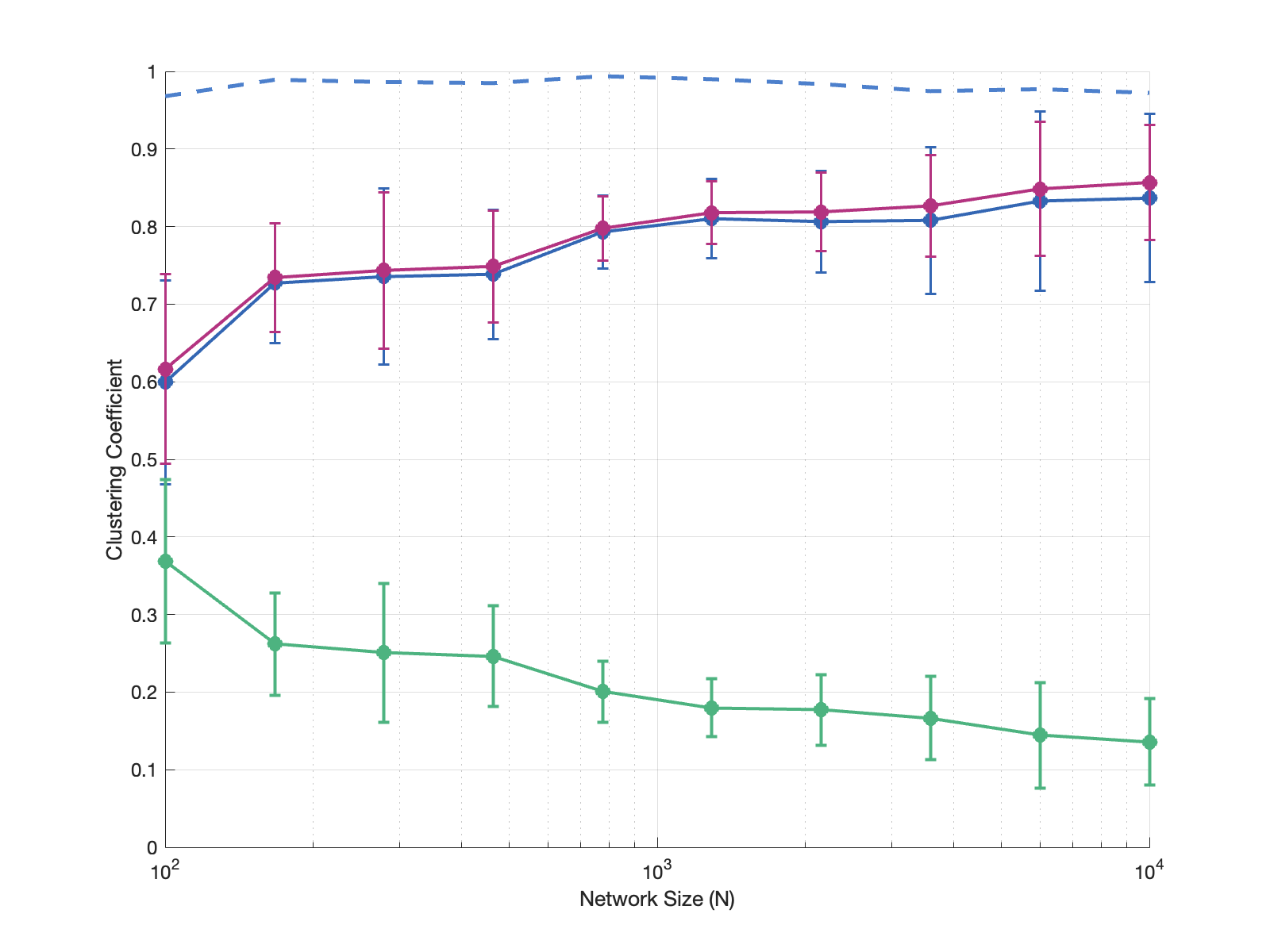}
    \includegraphics[width=0.49\linewidth, height= 0.32\linewidth]{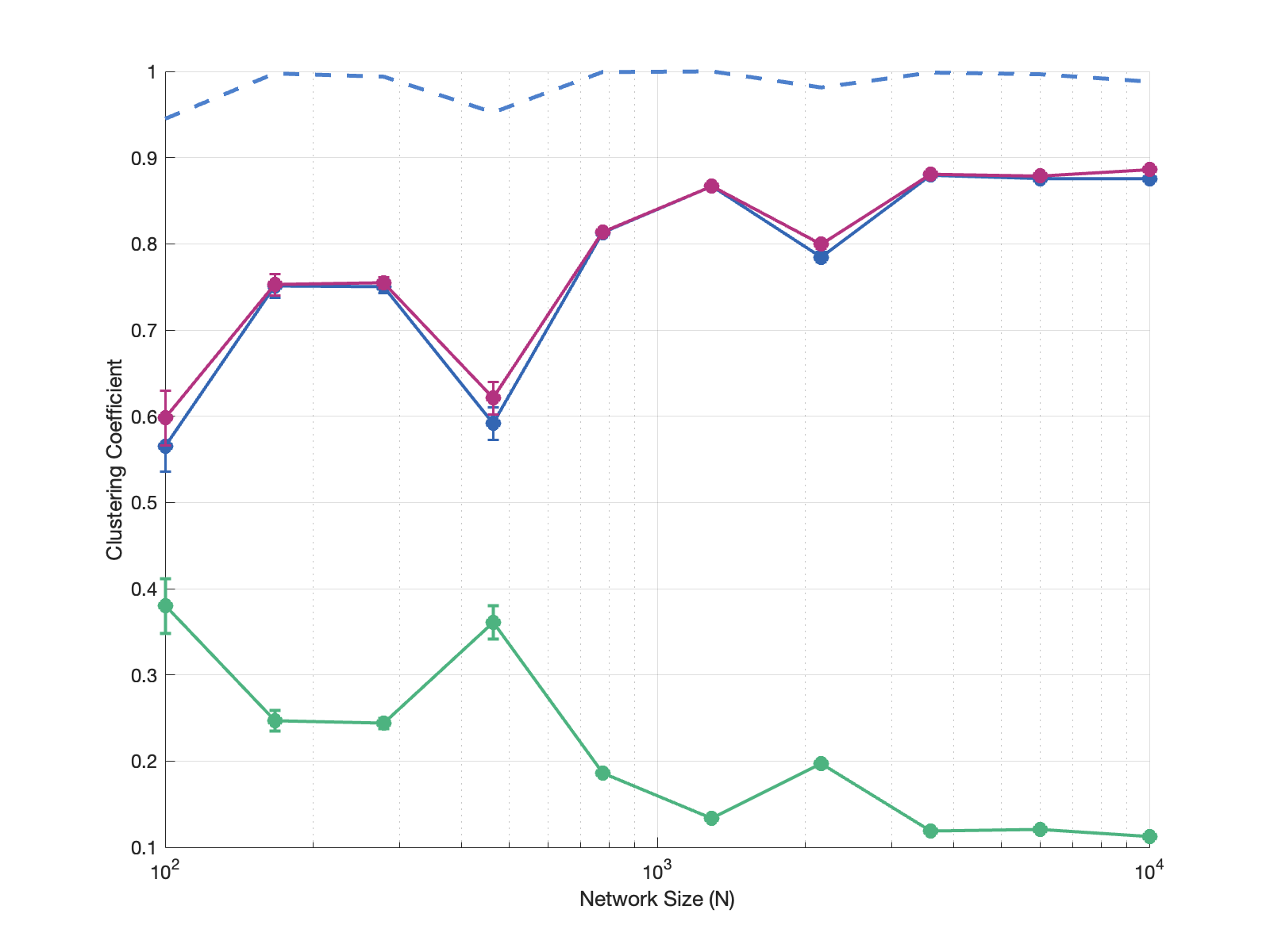}\\
    \includegraphics[width=0.49\linewidth, height= 0.32\linewidth]{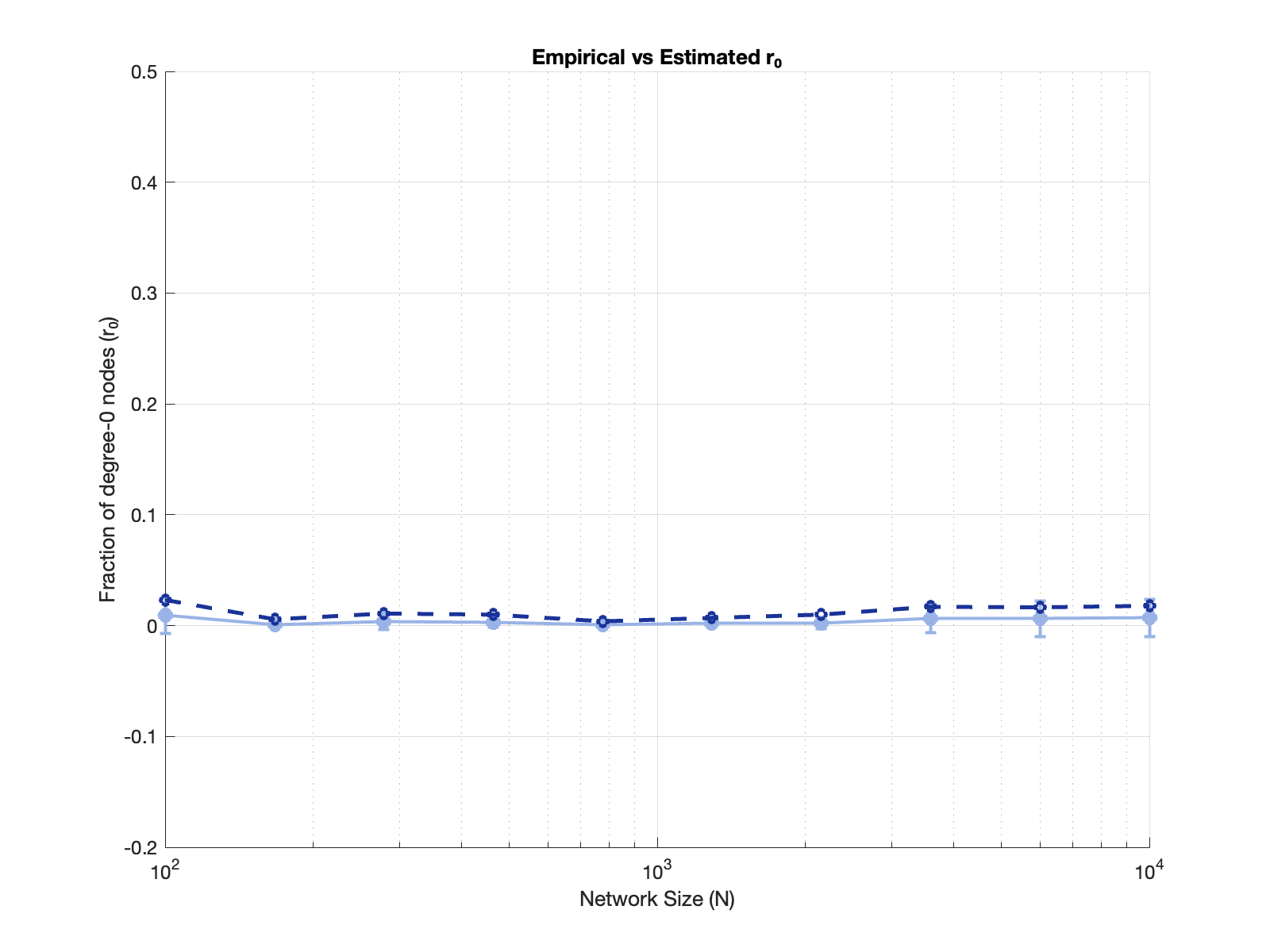}
    \includegraphics[width=0.49\linewidth, height= 0.32\linewidth]{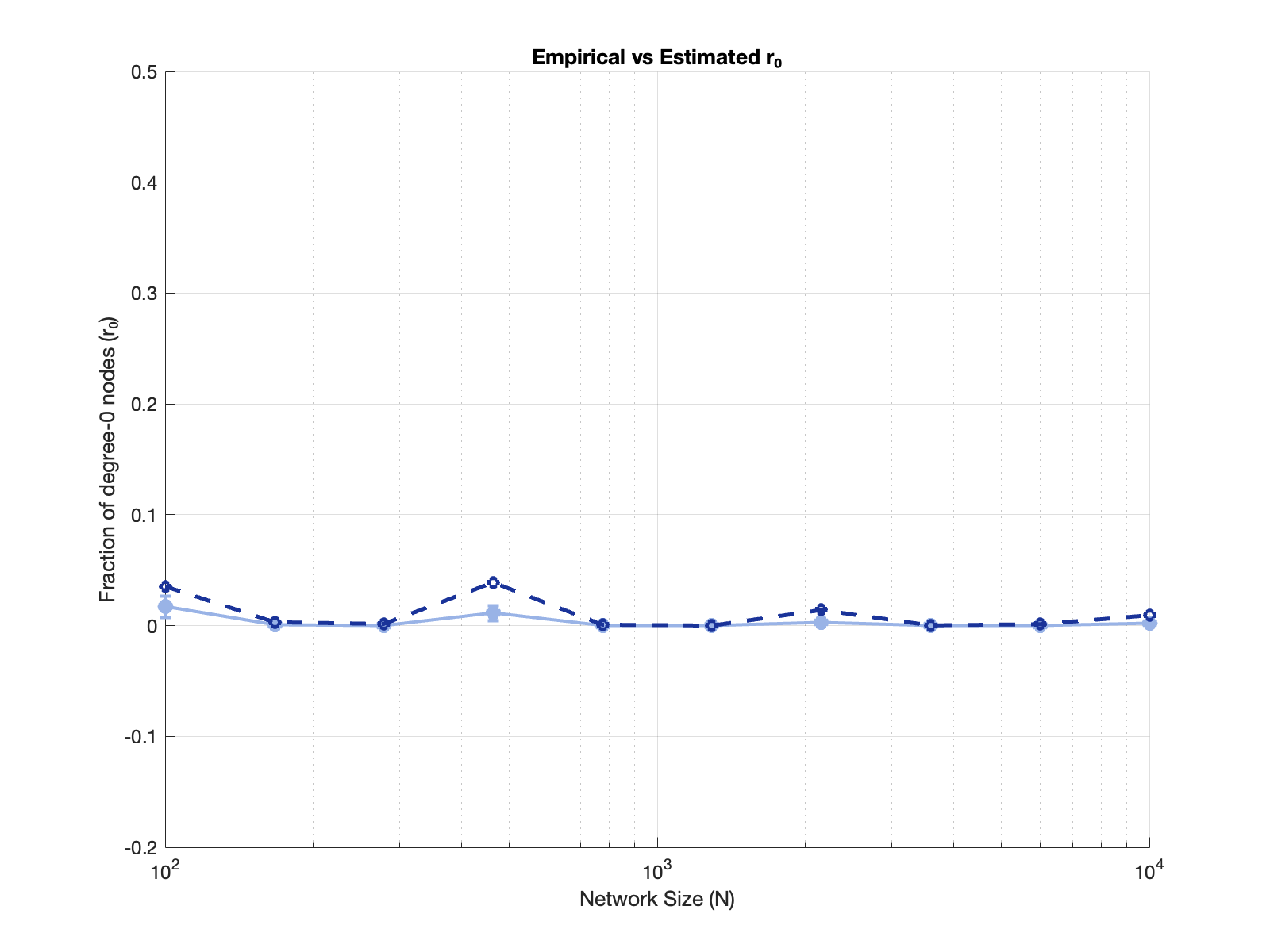}\\
    \includegraphics[width=0.49\linewidth, height= 0.32\linewidth]{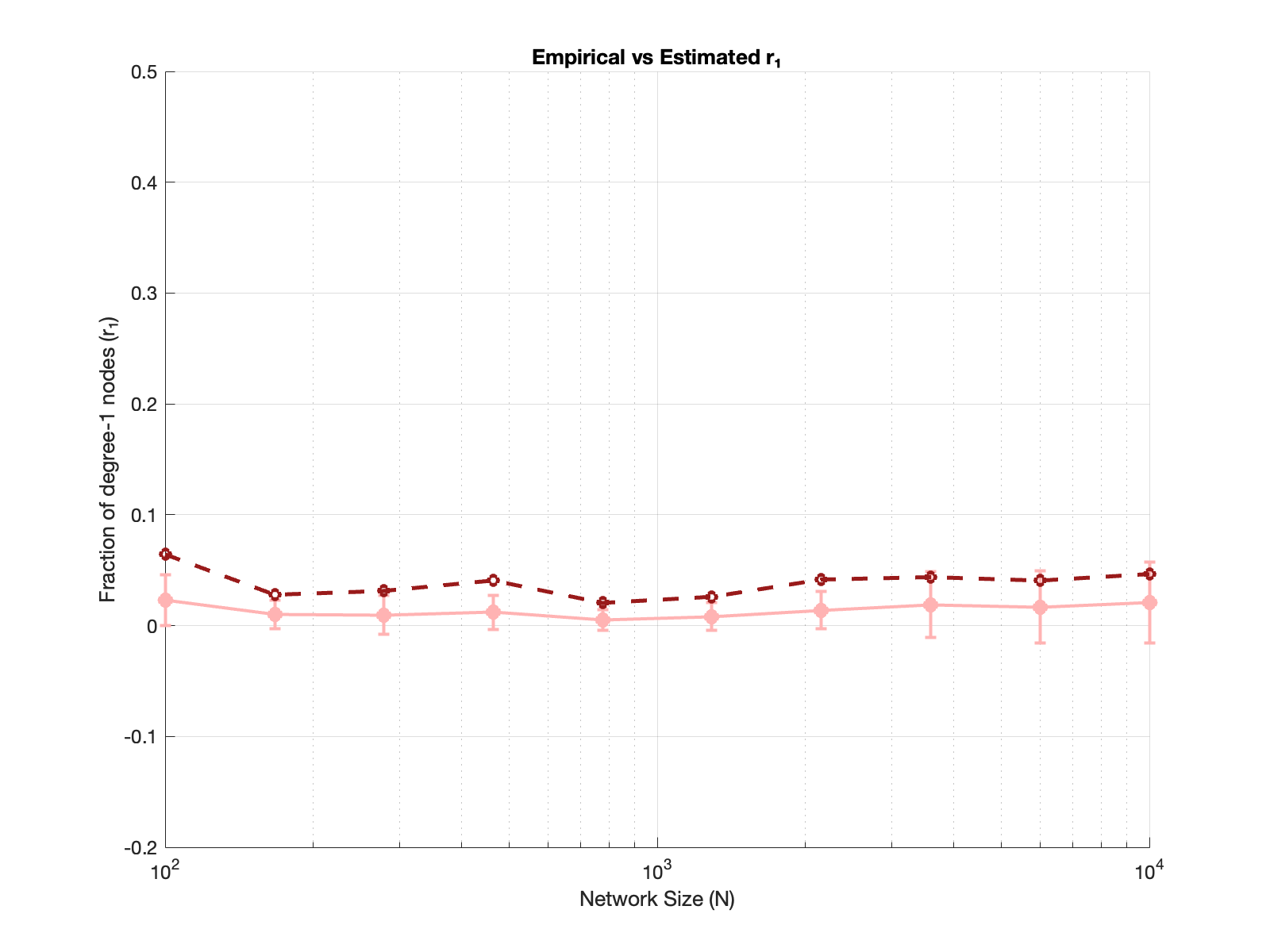}
    \includegraphics[width=0.49\linewidth, height= 0.32\linewidth]{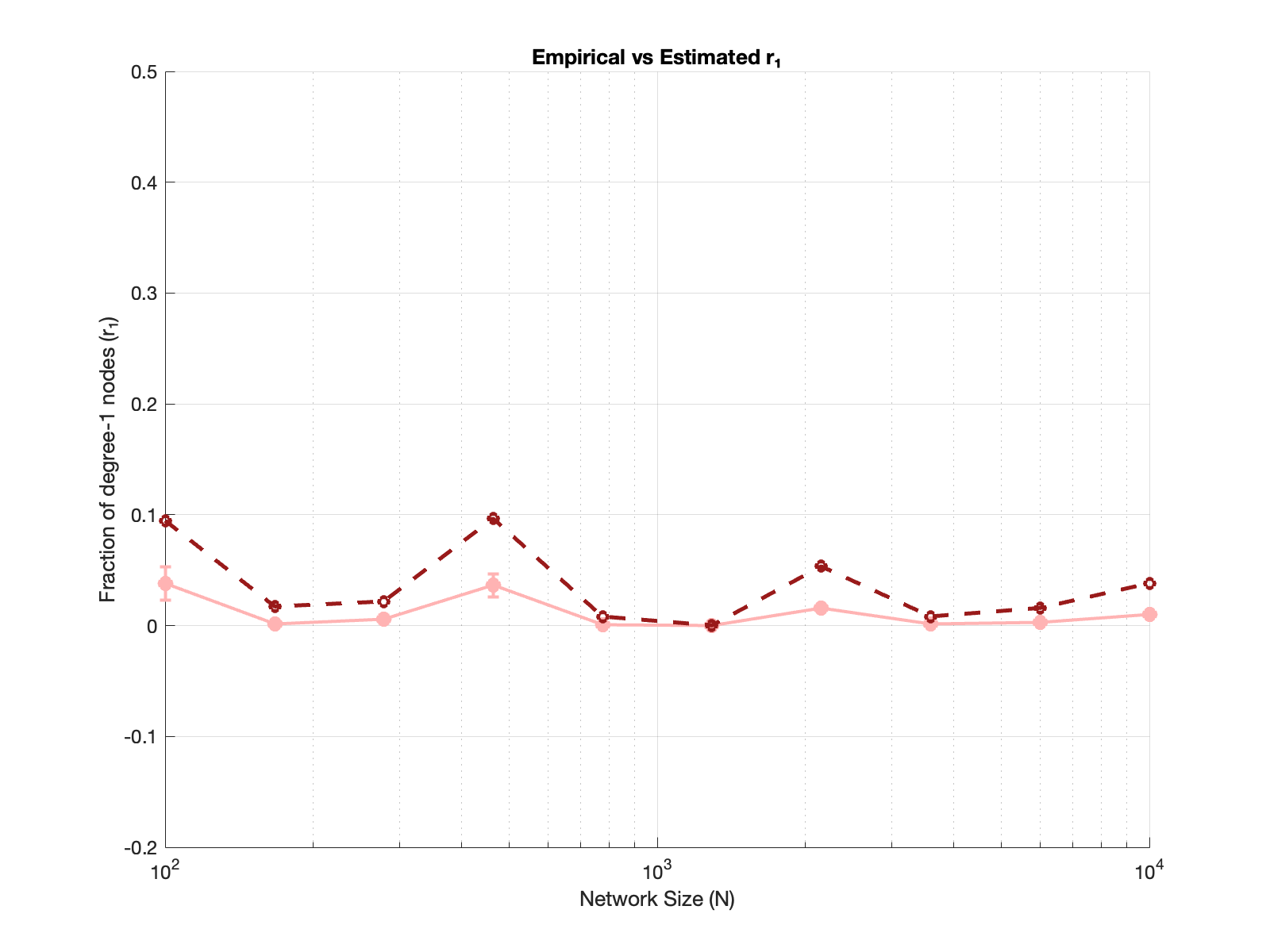}
    \caption{\textbf{Average local clustering coefficient $C$ and distance to $r_{0/1}$ versus network size $n$, with weights drawn from a Stable distribution (for $\alpha=0.7$).} 
    On the left, simulations are done by resampling weights every time an actual network is realized (sample of $10$), while on the left for each $n$ weights are extracted only once, and $10$ different adjacency matrices are realized from the same weights. On top
    we show the node-averaged local clustering coefficient ${C}$, both including (blue symbols) and excluding (purple symbols) nodes with degree $k=0,1$ (note that the latter evolves smoothly towards 1 with shrinking error bars as $n$ increases, while the former fluctuates with non-vanishing error bars, as a result of non-self-averaging).    
    The dashed blue line is $1$ minus the average of $r_{0/1}$ over realizations. Finally, in green the difference between $1-r_{0/1}$ and ${C}$ computed including nodes with $k<2$ (notice the shrinking error bars). In the middle, the light blue line with error bars is the actual value of $r_0$, while the dashed darker one is the approximation we derive in~\eqref{r0-approximation}. At the bottom, the light red line with error bars is the actual value of $r_1$, while the dashed darker one is the approximation we derive in~\eqref{r1-approximation}. }
    \label{fig:pan_07_stable}
\end{figure}

\end{document}